\documentclass[10pt]{amsart}
\usepackage{graphicx}
\usepackage{mathpazo}
\usepackage{amssymb}
\usepackage{epstopdf}
\usepackage[cmtip,all]{xy}
\usepackage{hyperref}
\usepackage{mathrsfs}
\usepackage{enumerate}
\hypersetup{linktocpage}
\newtheorem{prop}{Proposition}[section]
\newtheorem{lemma}[prop]{Lemma}

\usepackage{mathrsfs}
\usepackage{amsthm}
\usepackage{amssymb}
\usepackage{amsfonts}
\usepackage{tikz} 
\usepackage{mathtools}
\usepackage{setspace}
\usepackage{cancel}
\usepackage{chemarrow}
\usetikzlibrary{matrix,arrows} 
\usepackage{amscd}

\newtheorem*{theorem*}{Theorem}
\newtheorem{theorem}[prop]{Theorem}

\newtheorem{cor}[prop]{Corollary}

\newtheorem{assumption}[prop]{Assumption}

\newtheorem{remark}[prop]{Remark}
\newtheorem{defn}[prop]{Definition}

\begin{document}
\title{Higher rank stable pairs and virtual localization}
\author{Artan Sheshmani}
\maketitle
\begin{abstract}
We introduce a higher rank analog of the Pandharipande-Thomas theory  of stable pairs \cite{a17} on a Calabi-Yau threefold $X$.  More precisely,  we develop a moduli theory for frozen triples given by the data ${\mathcal O}_X^{\oplus r}(-n)\xrightarrow{\phi} F$ where $F$ is a sheaf of pure dimension $1$. The moduli space of such objects does not naturally determine an enumerative theory: that is, it does not naturally possess a perfect  
symmetric obstruction theory.  Instead, we build a zero-dimensional virtual fundamental class by hand, by truncating a deformation-obstruction theory coming from the moduli of objects in the derived  
category of $X$.  This yields the first deformation-theoretic construction of a higher-rank enumerative theory for Calabi-Yau threefolds.  We calculate this enumerative theory for local ${\mathbb  
P}^1$ using the Graber-Pandharipande \cite{a25} virtual localization technique. 
\end{abstract}
\section{Introduction}
Pandharipande and Thomas (PT for short) in \cite{a18, a17}  introduced the notion of stable pairs, given by a tuple $(F,s)$ where $s\in H^{0}(X,F)$ and $F$ is a pure sheaf with scheme-theoretically  one dimensional support (with fixed Hilbert polynomial and fixed second Chern character). 

Roughly speaking, these were sheaf theoretic objects representing a system of curves and points (on the curves) embedded in a fixed ambient Calabi-Yau threefold. The advantage of working with stable pairs was that they were realized as two term complexes $\mathcal{O}\xrightarrow{s}F$, and the authors proved that they were deforming as objects in the derived category. The robust yet powerful machinery of deformation theory of objects in the derived category then enabled the authors to construct well defined deformation-obstruction complexes over the moduli space of stable pairs and their associated virtual fundamental classes and compute their corresponding invariants. These enabled algebraic geometers to verify many more exciting results in the context of enumerative geometry of Calabi Yau manifolds, as well as mathematical string theory among some of which, one can mention the proof of MNOP conjecture \cite{a109} as well as the proof of KKV conjecture over K3 surfaces \cite{KKVPT}.

This article consists of two parts, Part I (theory) and Part II (Calculations).  In Part I, we develop a higher rank analog of the theory of stable pairs, in particular we produce a truncated well defined deformation-obstruction theory of correct amplitude for this higher rank theory which, based on a crucial assumption, provides a globally well defined virtual fundamental class for the theory. 

In Part II of the paper we apply our constructions to a toric variety, the local $\mathbb{P}^{1}$. In Part II we show that over the torus fixed loci of the moduli space, the higher rank objects of rank $r$ become isomorphic to $r$ copies of (\textit{twisted}) PT stable pairs. This fact  implies two outcomes; firstly we can show that over the torus fixed loci of the moduli space, the deformation-obstruction theory of higher rank objects becomes isomorphic to multiple copies of PT deformation-obstruction theory which is proven to be perfect of correct amplitude \cite{a17}. Therefore, this implies that the restriction of the truncated theory, constructed in Part I, to the torus fixed loci becomes perfect, which automatically implies that the assumption made in Part I, independently, holds true over the torus-fixed loci. Secondly via direct equivariant calculation of the theory of higher rank objects, and using the identification of their torus fixed loci, we verify at the end of the article that  their corresponding equivariant vertex matches precisely with a twisted version of an $r$-fold product of PT equivariant vertex. Below we elaborate further about the content of Part I and Part II;

\section*{Part I (Theory)} First we need the notion of triples; Let $X$ be a nonsingular Calabi-Yau 3-fold over $\mathbb{C}$ with $H^{1}(\mathcal{O}_{X})=0$ and with a fixed polarization $L$. A triple of type $(P_{E},P_{F})$ over $X$ is given by a triple $(E,F,\phi)$ where $E$ and $F$ are coherent sheaves with fixed Hilbert polynomials $P_{E}$ and $P_{F}$ respectively, $F$ is given as a pure sheaf with one dimensional support over $X$ and $\phi:E\rightarrow F$ is a holomorphic morphism. 

In this article however, in order to generalize the PT theory of stable pairs, we will use two special cases of the above triples; First we will introduce the notion of ``\textit{frozen triples}" of type $(r, P_F)$ which are given as a special case of the triples $(E,F,\phi)$ of type $(P_{E}, P_{F})$, in the sense that $E\cong \mathcal{O}_{X}^{\oplus r}(-n)$ and $F$ has fixed Hilbert polynomial $P_{F}$. In other words we ``\textit{freeze}" $E$ to be isomorphic to $\mathcal{O}_{X}^{\oplus r}(-n)$ for fixed choice of $r$, but the choice of this isomorphism is not fixed. Secondly, we will also work with closely related objects, called ``\textit{highly frozen triples}", given as quadruples $(E,F,\phi,\psi)$ where $E$, $F$ and $\phi$ have the same definition as above however, this time we have ``\textit{highly}" frozen the triple by fixing a choice of isomorphism $$\psi: E\xrightarrow{\cong}  \mathcal{O}_{X}^{\oplus r}(-n).$$ The stability condition for frozen and highly frozen triples is compatible with PT stability of stable pairs (c.f. Remark  \ref{PT-stab}). We call this stability condition $\tau'$-limit-stability or in short $\tau'$-stability. As it turns out (Lemma \ref{lemma2}), a frozen (respectively highly frozen) triple  $(E,F,\phi)$ of type $(r, P_F)$ is $\acute{\tau}$-stable if and only if the map $E\xrightarrow{\phi}F$ has zero dimensional cokernel.  

Naturally, the moduli spaces of $\tau'$-stable frozen and highly frozen triples are given as algebraic stacks. In fact the notion of $\tau'$-stability condition turns out to be a limiting GIT stability and thus we apply the results of Wandel \cite{a62} (Section 3) to prove in Remark \ref{forget-BG}, Theorem \ref{theorem8'} and Theorem \ref{pf-DM} that the moduli stack of $\tau'$-stable highly frozen triples of type $(r, P_F)$, which we denote by $\mathfrak{M}^{(r,P_{F})}_{\text{HFT}}(\tau')$, is given as a scheme which is a $\text{GL}_{r}(\mathbb{C})$-torsor over the (Artin) moduli stack of $\tau'$-stable frozen triples, denoted by  $\mathfrak{M}^{(r,P_{F})}_{\text{FT}}(\tau')$.

\begin{remark}
\emph{Our constructions in this article depend on the fixed choice of large enough integer $n$ for which $H^i(F(n))=0$ for all $i>0$. In fact our computation of the deformation-obstruction theories for frozen and highly frozen triples hold true over the sub-loci, $\mathfrak{N}^{(r,P_{F})}_{\text{HFT}}(\tau')\subset \mathfrak{M}^{(r,P_{F})}_{\text{HFT}}(\tau')$ and $\mathfrak{N}^{(r,P_{F})}_{\text{FT}}(\tau')\subset \mathfrak{M}^{(r,P_{F})}_{\text{FT}}(\tau')$ over which the higher cohomologies of $F(n)$ vanish (c.f. Definition \ref{open-sub}).}
\end{remark}

\begin{remark}
\emph{Note that, as the choice of $n, (r,P_{F})$, as well as the stability condition $\tau'$ throughout this article is fixed, later whenever needed, we will omit from our notation the fixed parameters in our construction and use $\mathfrak{N}_{\text{HFT}}, \mathfrak{M}_{\text{HFT}}$ and/or $\mathfrak{N}_{\text{FT}}, \mathfrak{M}_{\text{FT}}$ for our moduli spaces.}
\end{remark}
For a 3-fold $X$ the natural deformation-obstruction theories of stable frozen and highly frozen triples fail to provide well behaved complexes of correct amplitude over $\mathfrak{N}_{\text{FT}}$ and $\mathfrak{N}_{\text{HFT}}$ and they do not admit virtual fundamental cycles. Therefore, in order to find a remedy to this issue, we carry out a careful study of deformation theory of frozen and highly frozen triples viewed as objects in $\mathcal{D}^{b}(X)$ given by $I^{\bullet}: {\mathcal O}_X^{\oplus r}(-n)\rightarrow F$ and compute the fixed-determinant  obstruction theory of $I^{\bullet}$. It then turns out that, despite the fact that the object $I^{\bullet}$ (with the fixed determinant) in the derived category does not distinguish between a frozen or a highly frozen triple, \textit{its deformation space does}; In other words, it can be shown that given a frozen triple $(E,F,\phi)$ and a highly frozen triple $(E,F,\phi,\psi)$, both associated to the same object $I^{\bullet}\in \mathcal{D}^{b}(X)$, the space of flat deformations of $(E,F,\phi)$ and $I^{\bullet}$ are equally given by the group $\text{Ext}^{1}(I^{\bullet},I^{\bullet})_{0}$ while the space of flat deformations of $(E,F,\phi,\psi)$ is \textbf{not} equal to that of $I^{\bullet}$ given by $\text{Ext}^{1}(I^{\bullet},I^{\bullet})_{0}$. 

We use this fact to compute the deformation space of highly frozen triples (Proposition \ref{tripledef2}) and then, using a comparison between the two moduli spaces (Proposition \ref{tripledef}), obtain the deformation space of frozen triples (Theorem \ref{Gamma1}); Finally we prove that the deformations of frozen triples are equal to that of objects in the derived category (Theorem \ref{theorem10}). We summarize these results as follows;

\begin{theorem*} (Proposition \ref{tripledef2}, Proposition \ref{tripledef}, Theorem \ref{Gamma1}, Theorem \ref{theorem10}).
\begin{itemize}
\item Fix a map $f: S\rightarrow \mathfrak{N}_{\text{HFT}}$. Let $S'$ be a square-zero extension of $S$ with ideal $\mathcal{I}$. Let $\mathcal{D}ef_{S}(S',\mathfrak{N}_{\text{HFT}})$ denote the deformation space of the map $f$ obtained by the set of possible deformations, $f':S'\rightarrow \mathfrak{N}_{\text{FT}}$. The following statement is true:
\begin{equation*}
\mathcal{D}ef_{S}(S',\mathfrak{N}_{\text{HFT}})\cong \text{Hom}(I^{\bullet}_{S},F)\otimes \mathcal{I}
\end{equation*}
\item Fix a map $f: S\rightarrow \mathfrak{N}_{\text{FT}}$. Fixing $f$ corresponds to fixing an $S$-flat family of frozen triples given by $[\mathcal{O}_{X}(-n)\boxtimes \mathcal{M}_{S}\rightarrow \mathcal{F}]$ as in Definition \ref{defn4}. Let $S'$ be a square-zero extension of $S$ with ideal $\mathcal{I}$. Let $\mathcal{D}ef_{S}(S',\mathfrak{N}_{\text{FT}})$ denote the deformation space of the map $f$ obtained by the set of possible deformations, $f':S'\rightarrow \mathfrak{N}_{\text{FT}}$. The following statement is true:
\begin{equation*}
\mathcal{D}ef_{S}(S',\mathfrak{N}_{\text{FT}})\cong \text{Hom}(I^{\bullet}_{S},F)\otimes \mathcal{I}\slash\text{Im}\bigg((\mathfrak{g}l_{r}(\mathcal{O}_{S})\rightarrow \text{Hom}(I^{\bullet}_{S},\mathcal{F}))\otimes \mathcal{I}\bigg)
\end{equation*} 
\item Let $p \in \mathfrak{N}_{\text{FT}}$ be a point represented by the complex with fixed determinant $I^{\bullet}:=\mathcal{O}_{X}(-n)^{\oplus r}\xrightarrow{\phi} F$. The following is true: $$\text{T}_{p}\mathfrak{N}_{\text{FT}} \cong \text{Ext}^{1}(I^{\bullet},I^{\bullet})_{0}.$$
\end{itemize}
\end{theorem*}
We will then show that over $\mathfrak{N}_{\text{FT}}$, the deformation of the universal object in the derived category induces a 4-term deformation-obstruction complex of perfect amplitude $[-2,1]$:
\begin{theorem*}(Theorem \ref{reldef-f}). There exists a map in the derived category given by:$$R\pi_{\mathfrak{N}\ast}\left(R\mathscr{H}om(\mathbb{I}^{\bullet},\mathbb{I}^{\bullet})_{0}\otimes \pi_{X}^{*}\omega_{X}\right)[2]\xrightarrow{ob} \mathbb{L}^{\bullet}_{\mathfrak{N}_{\text{FT}}}.$$ 
After suitable truncations, there exists a 4 term complex $\mathbb{E}^{\bullet}$ of locally free sheaves , such that $\mathbb{E}^{\bullet\vee}$ is self-symmetric of amplitude $[-2,1]$ and there exists a map in the derived category:
\begin{equation}\label{defobs-froz}
\mathbb{E}^{\bullet\vee}\xrightarrow{ob^{t}} \mathbb{L}^{\bullet}_{\mathfrak{N}_{\text{FT}}},
\end{equation}
such that $h^{-1}(ob^{t})$ is surjective, and $h^{0}(ob^{t})$ and $h^{1}(ob^{t})$ are isomorphisms. Here $\mathbb{L}^{\bullet}_{\mathfrak{N}_{\text{FT}}}$ stands for the truncated cotangent complex of the Artin stack $\mathfrak{M}_{\text{FT}}$ which is of amplitude $[-1,1]$.
\end{theorem*}

Now we give a brief overview of the main result of the paper which is: \textit{to construct a virtual fundamental class from the non-perfect obstruction theory of stable frozen and highly frozen triples.} 

Let $\pi:\mathfrak{N}_{\text{HFT}}\rightarrow \mathfrak{N}_{\text{FT}}$ denote the natural forgetful map in Diagram \eqref{diagram}. The complex $\pi^{*}\mathbb{E}^{\bullet\vee}$ is perfect of amplitude $[-2,1]$ and the main obstacle in constructing a well-behaved deformation-obstruction theory over $\mathfrak{N}_{\text{HFT}}$ is to truncate $\pi^{*}\mathbb{E}^{\bullet\vee}$ in to a 2-term complex, and define (globally) a well-behaved deformation-obstrution theory of perfect amplitude $[-1,0]$. 

The simplest solution to this problem is to apply a cohomological truncation operation. Doing so requires obtaining a certain \textit{lifting} map from $g:\Omega_{\pi}\rightarrow \pi^{*}\mathbb{E}^{\bullet\vee}$ (c.f. Proposition \ref{existence}),  taking the mapping cone of this lift (and shifting by $-1$) and proving that the resulting complex satisfies the conditions of being a perfect deformation-obstruction theory for $\mathfrak{N}_{\text{HFT}}$. Here $\Omega_{\pi}$ is the relative cotangent sheaf of $\pi:\mathfrak{N}_{\text{HFT}}\rightarrow \mathfrak{N}_{\text{FT}}$. This procedure will remove the degree 1 term from the complex $\pi^{*}\mathbb{E}^{\bullet\vee}$. We also require to remove the degree $-2$ term of $\pi^{*}\mathbb{E}^{\bullet\vee}$. 

Then we use the self-symmetry of  $\mathbb{E}^{\bullet\vee}$ and apply the same procedure as above to the dual map $g^{\vee}:\pi^{*}\mathbb{E}^{\bullet}\rightarrow \text{T}_{\pi}$ (c.f. Theorem \ref{2step-trunc}) obtained from dualizing the map $g$. We finally obtain a local truncation of $\pi^{*}\mathbb{E}^{\bullet\vee}$ of perfect amplitude $[-1,0]$ which we denote by $\mathbb{G}^{\bullet}$. Knowing that $\pi^{*}\mathbb{E}^{\bullet\vee}$ is quasi-isomorphic to a 4 term complex of vector bundles:
\begin{equation}
\pi^{*}E^{-2}\rightarrow \pi^{*}E^{-1}\rightarrow \pi^{*}E^{0}\rightarrow \pi^{*}E^{1}
\end{equation}it can be seen from our construction (Lemma 8.18) that, locally, the complex $\mathbb{G}^{\bullet}$ is given by$$\pi^{*}E^{-2}\xrightarrow{d'} \pi^{*}E^{-1}\oplus T_{\pi}\rightarrow \pi^{*}E^{0}\oplus \Omega_{\pi}\xrightarrow{d} \pi^{*}E^{1}$$which is quasi-isomorphic to a 2-term complex of vector bundles
\begin{equation}\label{G-dot}
\text{Coker}(d')\rightarrow \text{Ker}(d)
\end{equation}
concentrated in degrees $-1$ and $0$.
The existence of the lifting map $g$ is guaranteed Zariski locally over $\mathfrak{N}_{\text{HFT}}$ but not globally. Hence our strategy is to locally truncate $\pi^{*}\mathbb{E}^{\bullet\vee}$ as explained above, construct the corresponding \textit{local} virtual cycles and, assuming that certain technical condition holds true (c.f. Assumption \ref{descent assumption}), glue the local cycles together to define a globally-defined virtual fundamental class. Our main summarizing theorems of this part are as follows:
\begin{theorem*}(Theorem \ref{2step-trunc} and Theorem \ref{2step-trunc2}). 
\begin{itemize}
\item Consider the 4-term deformation-obstruction theory $\mathbb{E}^{\bullet\vee}$ of perfect amplitude $[-2,1]$ over $\mathfrak{N}_{\text{FT}}$. Locally in the Zariski topology over $\mathfrak{N}_{\text{HFT}}$ there exists a perfect two-term deformation-obstruction theory of perfect amplitude $[-1,0]$ which is obtained from the suitable local truncation of the pullback $\pi^{*}\mathbb{E}^{\bullet\vee}$ via the map $\pi:\mathfrak{N}_{\text{HFT}}\rightarrow \mathfrak{N}_{\text{FT}}$.
\item Assuming the technical condition in Assumption \ref{descent assumption}, the local deformation-obstruction theory in Theorem \ref{2step-trunc} satisfies the conditions of being a semi perfect deformation-obstruction theory in the sense of \cite[Definition 3.1]{a70} and hence, it defines a globally well-behaved virtual fundamental class over $\mathfrak{N}_{\text{HFT}}$.
\end{itemize}
\end{theorem*}

\section*{Part II (Calculations)}

If $X$ has a torus action, then the moduli space of highly frozen triples inherits it (once we have chosen an equivariant structure of $\mathcal{O}_{X}(-n)$). Let $\mathcal{T}$ denote the torus action induced on $\mathfrak{N}_{\text{HFT}}$. It can be shown that a torus fixed point in the moduli scheme corresponds to a $\mathcal{T}$-equivariant highly frozen triple of type $(r, P_F)$ (Proposition \ref{equiv1}). The key observation is that a $\mathcal{T}$-equivariant highly frozen triple of rank $r$ is always written as a direct sum of $r$-copies of $\textbf{T}:=\mathbb{C}^{*3}$-equivariant PT stable pair:
 \begin{equation}\label{decompos-01}
I^{\bullet,\mathcal{T}}\cong\bigoplus_{i=1}^{r}\left(\mathcal{O}_{X}(-n)\rightarrow F_{i}\right)^{\textbf{T}}.
\end{equation}  

The consequence of identity \eqref{decompos-01} is of significant importance, since it enables one to immediately realize that the $\mathcal{T}$-fixed loci of $\mathfrak{N}_{\text{HFT}}$ are given as $r$-fold product of $\textbf{T}$-fixed loci of PT moduli space of stable pairs which are conjectured by Pandharipande and Thomas in \cite{a17} (Conjecture 2) to be nonsingular and compact. Therefore, it turns out that even though, originally, our moduli scheme is non-compact, its torus-fixed loci are compact, over which we are able to carry out localization computations. 

Moreover here we can see that Assumption \ref{descent assumption} for Theorem \ref{2step-trunc2} holds true independently over $\mathcal{T}$-fixed loci (c.f. Lemma \ref{auto}); the reason is due to the fact that over the $\mathcal{T}$-fixed loci of $\mathfrak{N}_{\text{HFT}}$ the restriction of the truncated deformation-obstruction complex in Theorem \ref{2step-trunc2} becomes isomorphic to $r$ copies of $\textbf{T}$-fixed PT deformation obstruction complex (c.f. Proposition \ref{r-PT}) which is perfect of amplitude $[-1,0]$ naturally (also look at Remark \ref{important}).

Finally we apply our results to threefolds given as toric varieties (Section \ref{sec15}) and directly compute the 1-legged equivariant Calabi-Yau vertex for when X is given as the total space of $\mathcal{O}_{\mathbb{P}^{1}}(-1)\oplus \mathcal{O}_{\mathbb{P}^{1}}(-1)\rightarrow \mathbb{P}^{1}$ (Equation \eqref{one-leg}); 
 \begin{equation}
 W^{\text{HFT}}_{1,\emptyset,\emptyset}=\left((1+q)^{\frac{(n+1)(s_{2}+s_{3})}{s_{1}}}\right)^{r}
 \end{equation}
 \section*{Acknowledgment}
The author thanks Sheldon Katz and Tom Nevins for their endless help and insight over years. Thanks to Richard Thomas for the invaluable help. Thanks to Emanuel 
Diaconescu for several discussions at RTG summer school at U.Penn in 2009. Thanks 
to Amin Gholampour for reading different versions of this article and many helpful discussions. Thanks to Kai Behrend for discussions in wallcrossing conference at UIUC and for pointing out the work in \cite{a43} on the construction of the virtual fundamental classes for Artin stacks.   
Thanks to Rahul Pandharipande, Davesh Maulik, Jim Bryan, Daniel Huybrechts, 
Yokinobu Toda, Tom Bridgeland, Dominic Joyce, Sam Payne, Julius Ross, Max Lieblich, Martin Olsson, Brent Doran, David Smyth, Francesco Noseda, Markus Perling for kindly answering questions about different aspects related to this project. Finally the author sincerely thanks the referee for many valuable comments on the content and presentation of this paper and specially, the author would like to thank the referee for pointing out the fact that, the existence of a 2-term locally free resolution of the truncated complex in the bottom row of diagram \eqref{q-isom} is the key to prove the existence of semi-perfect deformation-obstruction theory in the sense of Chang-Li, which enables one to construct the corresponding vector bundle stacks and do intersection theory. The author acknowledges partial support from NSF grants DMS 0244412, DMS 0555678 and DMS 08-38434 EMSW21-MCTP (R.E.G.S). 
\section{Part I (theory)}\label{chap1}
\begin{defn}(\textit{Holomorphic Triples})\label{defn3}
Let $X$ be a nonsingular projective Calabi-Yau 3-fold over $\mathbb{C}$ (i.e $K_{X}\cong \mathcal{O}_{X}$ and $\pi_{1}(X)=0$ which implies $H^{1}(\mathcal{O}_{X})=0$) with a fixed polarization $L$ . A holomorphic triple supported on $X$ is given by $(E,F,\phi)$ consisting of  a torsion free coherent sheaf $E$ and  a pure sheaf $F$ with one dimensional support, together with a holomorphic morphism $\phi:E \rightarrow F$.
A homomorphism of triples from $(E',F',\phi')$ to $(E,F,\phi)$ is a commutative diagram:
\begin{center}
\begin{tikzpicture}
back line/.style={densely dotted}, 
cross line/.style={preaction={draw=white, -, 
line width=6pt}}] 
\matrix (m) [matrix of math nodes, 
row sep=1em, column sep=3.5em, 
text height=1.5ex, 
text depth=0.25ex]{  
E'&F'\\
E&F\\};
\path[->]
(m-1-1) edge node [above] {$\phi'$} (m-1-2)
(m-1-1) edge (m-2-1)
(m-1-2) edge (m-2-2)
(m-2-1) edge node [above] {$\phi$} (m-2-2);
\end{tikzpicture}
\end{center}
Now let $S$ be a $\mathbb{C}$ scheme of finite type and let $\pi_{X}:X \times S \rightarrow X$ and $\pi_{S}: X\times S \rightarrow S$ be the corresponding projections. An \textit{$S$-flat family} of triples over $X$ is a triple $(\mathcal{E},\mathcal{F},\phi)$ consisting of a morphism of $\mathcal{O}_{X\times S}$ modules $\mathcal{E}\xrightarrow{\phi} \mathcal{F}$ such that $\mathcal{E}$ and $\mathcal{F}$ are flat over $S$ and for every point $s\in S$ the fiber  $(\mathcal{E},\mathcal{F},\phi)\mid_{s}$ is given by a holomorphic triple over $X$. Two $S$-flat families of triples $(\mathcal{E},\mathcal{F},\phi)$ and $(\mathcal{E}',\mathcal{F}',\phi')$ are isomorphic if there exists a commutative diagram  of the form:
\begin{center}
\begin{tikzpicture}
back line/.style={densely dotted}, 
cross line/.style={preaction={draw=white, -, 
line width=6pt}}] 
\matrix (m) [matrix of math nodes, 
row sep=1em, column sep=3.5em, 
text height=1.5ex, 
text depth=0.25ex]{ 
\acute{\mathcal{E}}_{1}&\acute{\mathcal{E}}_{2}\\
\mathcal{E}&\mathcal{F}\\};
\path[->]
(m-1-1) edge node [above] {$\phi'$} (m-1-2)
(m-1-1) edge node [left] {$\cong$} (m-2-1)
(m-1-2) edge node [right] {$\cong$} (m-2-2)
(m-2-1) edge node [above] {$\phi$} (m-2-2);
\end{tikzpicture}
\end{center}
\end{defn}
\begin{defn}(\textit{Type of a triple})
A triple of type $(P_{E},P_{F})$ is given by a triple $(E, F,\phi)$ such that   $P(E(m))=P_{E}$ and $P(F(m))=P_{F}$. Note that by the Grothendieck-Riemann-Roch theorem, fixing the Hilbert polynomial of $E, F$ is equivalent to fixing their Chern characters. 
\end{defn}
\begin{defn}(\textit{Frozen triples})\label{defn4}
Define a \textit{frozen triple} as a holomorphic triple in Definition \ref{defn3} such that $E\cong \mathcal{O}_{X}(-n)^{\oplus r}$ for some fixed large enough $n\in \mathbb{Z}$. For the frozen triples we simplify the notation for type of a frozen triple by writing ``\textit{of type} $(r, P_F)$". Now, an $S$-flat family of frozen-triples of type $(r,P_{F})$ is a triple $(\mathcal{E},\mathcal{F},\phi)$ consisting of a morphism of $\mathcal{O}_{X\times S}$ modules $\phi:\mathcal{E}\rightarrow \mathcal{F}$, such that $\mathcal{E}$ and $\mathcal{F}$ satisfy the condition of Definition \ref{defn3} and moreover $\mathcal{E}\cong \pi_{X}^{*}\mathcal{O}_{X}(-n)\otimes \pi_{S}^{*}\mathcal{M}_{S}$ where $\mathcal{M}_{S}$ is a vector bundle of rank $r$ on $S$. Two $S$-flat families of frozen-triples $(\mathcal{E},\mathcal{F},\phi)$ and $(\mathcal{E}',\mathcal{F}',\phi')$ are isomorphic if there exists a commutative diagram:
\begin{center}
\begin{tikzpicture}
back line/.style={densely dotted}, 
cross line/.style={preaction={draw=white, -, 
line width=6pt}}] 
\matrix (m) [matrix of math nodes, 
row sep=1em, column sep=3.5em, 
text height=1.5ex, 
text depth=0.25ex]{ 
\acute{\mathcal{E}}_{1}&\acute{\mathcal{E}}_{2}\\
\mathcal{E}&\mathcal{F}\\};
\path[->]
(m-1-1) edge node [above] {$\phi'$} (m-1-2)
(m-1-1) edge node [left] {$\cong$}(m-2-1)
(m-1-2) edge node [right] {$\cong$} (m-2-2)
(m-2-1) edge node [above] {$\phi$} (m-2-2);
\end{tikzpicture}
\end{center}
\end{defn}
\begin{defn}(\textit{Highly frozen triples})\label{defn7}
A \textit{highly frozen triple} of type $(r, P_{F})$ is a quadruple $(E,F,\phi,\psi)$ where $(E,F,\phi)$ is a frozen triple of type $(r, P_{F})$ as in Definition \ref{defn4} and $\psi:E\xrightarrow{\cong} \mathcal{O}_{X}(-n)^{\oplus r}$ is a fixed choice of isomorphism. A morphism between highly frozen triples $(E',F',\phi',\psi')$ and  $(E,F,\phi,\psi)$ is a morphism $F'\xrightarrow {\rho}F$ such that the following diagram is commutative.
\begin{center}
\begin{tikzpicture}
back line/.style={densely dotted}, 
cross line/.style={preaction={draw=white, -, 
line width=6pt}}] 
\matrix (m) [matrix of math nodes, 
row sep=1em, column sep=3.5em, 
text height=1.5ex, 
text depth=0.25ex]{ 
\mathcal{O}_{X}(-n)^{\oplus r}&E'&F'\\
\mathcal{O}_{X}(-n)^{\oplus r}&E&F\\};
\path[->]
(m-1-2) edge node [above] {$\phi'$} (m-1-3)
(m-1-1) edge node [left] {$id$}(m-2-1)
(m-1-1) edge node [above] {$\psi'^{-1}$}(m-1-2)
(m-1-2) edge (m-2-2)
(m-1-3) edge node [right] {$\rho$} (m-2-3)
(m-2-1) edge node [above] {$\psi^{-1}$}(m-2-2)
(m-2-2) edge node [above] {$\phi$} (m-2-3);
\end{tikzpicture}
\end{center}
An $S$-flat family of highly frozen triples is a quadruple $(\mathcal{E},\mathcal{F},\phi,\psi)$ consisting of a morphism of $\mathcal{O}_{X\times S}$ modules $\mathcal{E}\xrightarrow{\phi} \mathcal{F}$ such that $\mathcal{E}$ and $\mathcal{F}$ satisfy the condition of Definition \ref{defn3} and moreover $\psi: \mathcal{E}\xrightarrow{\cong}\pi_{X}^{*}\mathcal{O}_{X}(-n)\otimes \pi_{S}^{*}\mathcal{O}^{\oplus r}_{S}$ is a fixed choice of isomorphism. Two $S$-flat families of highly frozen-triples $(\mathcal{E},\mathcal{F},\phi,\psi)$ and $(\mathcal{E}',\mathcal{F}',\phi',\psi')$ are isomorphic if there exists a commutative diagram:
\begin{center}
\begin{tikzpicture}
back line/.style={densely dotted}, 
cross line/.style={preaction={draw=white, -, 
line width=6pt}}] 
\matrix (m) [matrix of math nodes, 
row sep=1em, column sep=3.5em, 
text height=1.5ex, 
text depth=0.25ex]{ 
\pi_{X}^{*}\mathcal{O}_{X}(-n)\otimes \pi_{S}^{*}\mathcal{O}^{\oplus r}_{S}&\mathcal{E}'&\mathcal{F}'\\
\pi_{X}^{*}\mathcal{O}_{X}(-n)\otimes \pi_{S}^{*}\mathcal{O}^{\oplus r}_{S}&\mathcal{E}&\mathcal{F}\\};
\path[->]
(m-1-2) edge node [above] {$\phi'$} (m-1-3)
(m-1-1) edge node [left] {$id$}(m-2-1)
(m-1-2) edge (m-2-2)
(m-1-1) edge node [above] {$\psi'^{-1}$} (m-1-2)
(m-2-1) edge node [above] {$\psi^{-1}$} (m-2-2)
(m-1-3) edge node [right] {$\cong$} (m-2-3)
(m-2-2) edge node [above] {$\phi$} (m-2-3);
\end{tikzpicture}
\end{center}
\end{defn}
\section*{Stability condition}\label{sec3} 
\begin{defn}(\textit{Stability of holomorphic triples})\label{defn8}
 Let $q_{1}(m)$ and $q_{2}(m)$ be positive rational polynomials of degree at least 2. A holomorphic triple $T=(E,F,\phi)$ of type $(P_{E},P_{F})$ is called $\acute{\tau}$-semistable (respectively, stable) if for any subsheaves $E'$ of $E$ and $F'$ of $F$ such that $0\neq E'\oplus F'\neq E\oplus F$ and $\phi(E')\subset F'$:
\begin{align}
&
q_{2}(m)\left(P_{E'}-rk(E')\left( \frac{P_{E}}{rk(E)}-\frac{q_{1}(m)}{rk(E)}\right)\right)\notag\\
&
+q_{1}(m)\left( P_{F'}-rk(F')\left( \frac{P_{F}}{rk(F)}+\frac{q_{2}(m)}{rk(F)}\right)\right)\leq0 / resp. <0.\notag\\
\end{align}
\end{defn}
Now we simplify Definition \ref{defn8} and obtain a tailored version of $\tau'$-stability condition for the frozen triples of type $(r,P_{F})$. Fix a frozen triple $(E,F,\phi)$ of type $(r, P_F)$ and let the subtriple $(E',F', \phi')$ be given as $(\mathcal{O}_{X}(-n)^{\oplus r},F',\phi')$ such that $F'\subset F$ and $\phi$ factors through $F'$, then the stability condition is written as:
\begin{align*}
&
q_{2}(m)\left(\cancel{P_{E'}}-\cancel{r\left(\frac{P_{E'}}{r}\right)}+q_{1}(m)\right)
+q_{1}(m)\left(P_{F'}-rk(F')\left(\frac{P_{F}}{rk(F)}+\frac{q_{2}(m)}{rk(F)}\right)\right)< 0.\notag\\
\end{align*}
Divinding by $q_{1}(m)$ and setting the new variable $q(m)=\frac{q_{2}(m)}{q_{1}(m)}$ as the ratio of the two, we obtain:
\begin{equation}\label{nonsense1}
\frac{P_{F'}}{rk(F')}+\frac{q(m)}{rk(F')}\leq\frac{P_{F}}{rk(F)}+\frac{q(m)}{rk(F)}.
\end{equation}
Which is similar to the notion of Le Potier's stability condition for coherent systems \cite{a13}.
 Now we state our conclusion as the following definition;
\begin{defn}(\textit{$\tau'$-stability for frozen triples})\label{defn13}
Let $q(m)$ be given by a polynomial with rational coefficients such that its leading coefficient is positive.  A frozen triple $(E,F,\phi)$ of type $(r, P_F)$ is $\acute{\tau}$-stable with respect to $q(m)$ if and only if the following conditions are satisfied;
\begin{enumerate}
\item For all proper nonzero subsheaves $G\subset F$ for which $\phi$ does not factor through $G$ we have: $$\frac{P_{G}}{rk(G)}<\frac{P_{F}}{rk(F)}+\frac{q(m)}{rk(F)}.$$
\item For all proper subsheaves, $G\subset F$ for which the map $\phi$ factors through:
\begin{align}\label{nonsense2}
&q(m)+\left(P_{G}-rk(G)\left(\frac{P_{F}}{rk(F)}+\frac{q(m)}{rk(F)}\right)\right)<0.\notag\\
\end{align}
\end{enumerate}
\end{defn}
\subsection*{Limit stability}
We show here that the $\tau'$-stability condition for frozen and highly frozen triples is asymptotically similar to stability of PT pairs \cite[Lemma 1.3]{a17}. 
\begin{lemma}(\textit{$\tau'$-limit-stability})\label{lemma2}
Fix $q(m)$ to be given as a polynomial of degree at least 2 with rational coefficients such that its leading coefficient is positive. A frozen triple $(E,F,\phi)$ of type $(r, P_F)$ is $\acute{\tau}$-limit-stable if and only if the map $E\xrightarrow{\phi}F$ has zero dimensional cokernel.
\end{lemma}
\begin{proof}  For simplicity, we use $\mathcal{O}_{X}^{\oplus r}(-n)$ instead of $E$. The exact sequence $0\rightarrow K\rightarrow \mathcal{O}_{X}(-n)^{\oplus r}\xrightarrow{\phi} F\rightarrow Q\rightarrow 0$ induces a short exact sequence $0\rightarrow \text{Im}(\phi)\rightarrow F\rightarrow Q \rightarrow 0$. Therefore one obtains the following commutative diagram of the triples:
\begin{center}
\begin{tikzpicture}
back line/.style={densely dotted}, 
cross line/.style={preaction={draw=white, -, 
line width=6pt}}] 
\matrix (m) [matrix of math nodes, 
row sep=1em, column sep=3.25em, 
text height=1.5ex, 
text depth=0.25ex]{ 
\mathcal{O}_{X}(-n)^{\oplus r}&\text{Im}(\phi)\\
\mathcal{O}_{X}(-n)^{\oplus r}&F\\};
\path[->]
(m-1-1) edge node [above] {$\phi$} (m-1-2)
(m-1-1) edge node [left] {$=$} (m-2-1);
\path[right hook->](m-1-2) edge (m-2-2);
\path[->](m-2-1) edge (m-2-2);
\end{tikzpicture}
\end{center}
Now we use the assumption for $q(m)$ and assume that $\mathcal{O}_{X}(-n)^{\oplus r}\xrightarrow{\phi}F$ is stable, therefore we obtain:
\begin{align*}
&
q(m)+\left(P_{\text{Im}(\phi)}-rk(\text{Im}(\phi))\cdot \left(\frac{P_{F}}{rk(F)}+\frac{q(m)}{rk(F)}\right)\right)<0.\notag\\
\end{align*}
In other words by rearrangement:
 \begin{equation}\label{gagool}
 q(m)\left(1-\frac{rk(\text{Im}(\phi)}{rk(F)}\right)<rk(\text{Im}(\phi))\frac{P_{F}}{rk(F)}-P_{\text{Im}(\phi)}.
 \end{equation}
 Consider the polynomials on both sides of inequality \eqref{gagool}. One finds that the right hand side of \eqref{gagool} is a polynomial in $m$ of degree at most 1. However by the assumed choice of $q(m)$, the left hand side of  inequality is given by a  polynomial of degree at least two with positive leading coefficient. Hence the left hand side becomes larger than the right hand side and the only way for the inequality to make sense is to have the left hand side to be equal to zero, i.e $rk(\text{Im}(\phi))=rk(F)$ and therefore $Q$ must be a zero dimensional sheaf. For the other direction: assume that $Q$ is  \textit{not} a zero dimensional sheaf and the triple is $\acute{\tau}$-limit-stable. Now by similar argument, since degree of $q(m)$ is chosen to be sufficiently large enough, $rk(\text{Im}(\phi))=rk(F)$ which contradicts the assumption of $Q$ not being zero \text{\text{dim}}ensional sheaf and this finishes the proof.\end{proof}
\begin{remark}\label{PT-stab}
(\textit{Relation to PT stability}) \emph{Note that setting $r=1$ and following the same steps to obtain Equation \eqref{nonsense1}, will enable us to see that the stability conditions defined in part (1) and (2) of Definition \ref{defn13} are the same as the ones used by Pandharipande-Thomas \cite[Equation (1.1) and (1.2) ]{a17}. Moreover, the result of Lemma \ref{lemma2} shows that the limit-stable frozen triples do satisfy the conditions, same as conditions (i) and (ii) of PT pair stability \cite[Lemma 1.3]{a17}. This result is certainly expected, due to the fact that PT stable pairs are special case of frozen triples where the value of  $r$ is set to be equal to $1$}.
\end{remark} 
Throughout the rest of this article by $\tau'$-stability we mean $\tau'$-limit-stability.

\section{Construction of moduli stacks}\label{chap2}
\begin{defn}(\textit{Moduli stack of highly frozen triples})\label{HFT-def}
Define $\mathfrak{M}^{(r,P_{F})}_{\text{HFT}}(\tau')$ to be the fibered category $\mathfrak{p}:\mathfrak{M}^{(r,P_{F})}_{\text{HFT}}(\tau') \rightarrow Sch/\mathbb{C}$ such that for all $S\in Sch/\mathbb{C}$ the objects in $\mathfrak{M}^{(r,P_{F})}_{\text{HFT}}(\tau')$ are $S$-flat families of $\tau'$-stable highly frozen triples of type $(r, P_F)$. Given a morphism of $\mathbb{C}$-schemes $g:S\rightarrow K$ and two families of highly frozen triples $T_{S}:=(\mathcal{E},\mathcal{F},\phi,\psi)_{S}$ and $\acute{T}_{K}:=(\mathcal{E}',\mathcal{F}',\phi',\psi')_{K}$ as in Definition \ref{defn7} (sub-index indicates the base parameter scheme over which the family is constructed), a morphism $T_{S}\rightarrow \acute{T}_{K}$ in $\mathfrak{M}^{(r,P_{F})}_{\text{HFT}}(\tau')$ is defined by an isomorphism: $$\nu_{S}: T_{S}\xrightarrow{\cong}(g\times \textbf{1}_{X})^{*}\acute{T}_{K}.$$
\end{defn}
\begin{defn}(\textit{Moduli stack of frozen triples})\label{defn141}
Define $\mathfrak{M}^{(r,P_{F})}_{\text{FT}}(\tau')$ to be the fibered category $\mathfrak{p}:\mathfrak{M}^{(r,P_{F})}_{\text{FT}}(\tau') \rightarrow Sch/\mathbb{C}$ such that for all $S\in Sch/\mathbb{C}$ the objects in $\mathfrak{M}^{(r,P_{F})}_{\text{FT}}(\tau')$ are $S$-flat families of $\tau'$-stable frozen triples of type $(r, P_F)$ as in Definition \ref{defn4}. Given a morphism of $\mathbb{C}$-schemes $g:S\rightarrow K$ and two families of frozen triples $T_{S}:=(\mathcal{E},\mathcal{F},\phi)_{S}$ and $\acute{T}_{K}:=(\mathcal{E}',\mathcal{F}',\phi')_{K}$ as in Definition \ref{defn4}, a morphism $T_{S}\rightarrow \acute{T}_{K}$ in $\mathfrak{M}^{(r,P_{F})}_{\text{FT}}(\tau')$ is defined by an isomorphism: $$\nu_{S}: T_{S}\xrightarrow{\cong}(g\times \textbf{1}_{X})^{*}\acute{T}_{K}.$$
\end{defn}
\begin{prop}
Use definitions \ref{HFT-def} and \ref{defn141}. The fibered categories $\mathfrak{M}^{(r,P_{F})}_{\text{HFT}}(\tau')$ and $\mathfrak{M}^{(r,P_{F})}_{\text{FT}}(\tau')$ are stacks.
\end{prop}
\begin{proof} This is immediate from faithfully flat descent of coherent sheaves and homomorphisms of coherent sheaves \cite[Theorem 4.23]{a67}.
\end{proof}
\begin{remark}\label{forget-BG} 
\emph{(Fibered diagram of moduli stacks) There exists a forgetful morphism $g':\mathfrak{M}^{(r,P_{F})}_{\text{FT}}(\tau')\rightarrow \mathcal{B}\text{GL}_{r}(\mathbb{C})$ which is given by taking a frozen triple $\{(E,F,\phi)\}\in \mathfrak{M}^{(r,P_{F})}_{\text{FT}}(\tau')$ to $\{E\}\in \mathcal{B}\text{GL}_{r}(\mathbb{C})$ by forgetting $F$ and $\phi$. Moreover, there exists a forgetful map $\mathfrak{M}^{(r,P_{F})}_{\text{HFT}}(\tau')\xrightarrow{\pi} \mathfrak{M}^{(r,P_{F})}_{\text{FT}}(\tau')$ which takes $(E,F,\phi,\psi)$ to $(E,F,\phi)$. Following the definition of fiber products of stacks, it can be shown that the natural diagram:}
\begin{equation}\label{diagram}
\begin{tikzpicture}
back line/.style={densely dotted}, 
cross line/.style={preaction={draw=white, -, 
line width=6pt}}] 
\matrix (m) [matrix of math nodes, 
row sep=1.5em
, column sep=3.25em, 
text height=1.5ex, 
text depth=0.25ex]{ 
\mathfrak{M}^{(r,P_{F})}_{\text{HFT}}(\tau')&pt=\text{Spec}(\mathbb{C})\\
\mathfrak{M}^{(r,P_{F})}_{\text{FT}}(\tau')&\mathcal{B}\text{GL}_{r}(\mathbb{C})=\left[\frac{\text{Spec}(\mathbb{C})}{\text{GL}_{r}(\mathbb{C})}\right]\\};
\path[->]
(m-1-1) edge node [above] {$g$} (m-1-2)
(m-1-1) edge node [left] {$\pi$} (m-2-1)
(m-1-2) edge node [right] {$i$}(m-2-2)
(m-2-1) edge node [above] {$\acute{g}$} (m-2-2);
\end{tikzpicture},
\end{equation} 
\emph{is a fibered diagram in the category of stacks. We leave the details to the interested reader to verify this.}
\end{remark}
 Next we show that the moduli stacks of frozen and highly frozen triples are given as algebraic stacks; The main requirement to construct the moduli stacks is the boundedness property for the family of triples of fixed given type. Wandel \cite[Definition 1.1]{a62}  studies the construction of the moduli space of objects $\phi:\mathcal{D}\rightarrow \mathcal{E}$, denoted as pairs. These objects are defined similar to triples in Definition \ref{defn3}. The author introduces the notion of Hilbert polynomial for a pair \cite[Definition 1.3]{a62} and $\delta$-semistability \cite[Definition 1.4]{a62} where $\delta$ is given as a stability parameter. The author then shows \cite[Proposition 2.1]{a62} that the family of $\delta$-stable pairs with Fixed Hilbert polynomial is bounded. 

Replacing $\delta$ with $q(m)$, it is easily seen that Wandel's notion of $\delta$-semistability is completely compatible with our notion of $\tau'$-semistability in Definition \ref{defn13} and therefore the family of highly frozen triples, $(E,F,\phi,\psi)$ of type $(r,P_{F})$ is bounded. Now the boundedness property implies that there exists an integer $n'$ (we call it $n'$ to distinguish it from the fixed integer $n$ appearing in the description of $E=\mathcal{O}^{r}_{X}(-n)$) for which, $E$ and $F$ satisfy regularity condition. In other words for all sheaves appearing in the family we have that $E(n')$ and $F(n')$ are globally generated. Therefore, there exists a surjective morphism $V_{F}\otimes \mathcal{O}_{X}(-n')\to F$ where $V_{F}:=H^{0}(F\otimes L^{n'})$ is a complex vector space of dimension $d_{V}=P_{F}(n')$. Now following the usual constructions,  one constructs the Quot-scheme $$\mathcal{Q}_{F}:=\text{Quot}_{P_{F}}(V_{F}\otimes \mathcal{O}_{X}(-n')),$$and so the scheme parameterizing the morphisms $\phi:\mathcal{O}_{X}^{\oplus r}(-n)\rightarrow F$ is given by a bundle $\mathcal{P}$ over $\mathcal{Q}_{F}$ whose fibers are given by $H^{0}(F(n))^{\oplus r}$. Now let $\mathfrak{S}(\tau')\subset \mathcal{P}$
be given as an open subscheme of $\mathcal{P}$ whose fibers parametrize $\tau'$-stable highly frozen triples $E\rightarrow F$.
$$
\begin{tikzpicture}
back line/.style={densely dotted}, 
cross line/.style={preaction={draw=white, -, 
line width=6pt}}] 
\matrix (m) [matrix of math nodes, 
row sep=1em, column sep=2.5em, 
text height=1.5ex, 
text depth=0.25ex]{ 
&V_{F}\otimes \mathcal{O}_{X}(-n')\\
\mathcal{O}_{X}(-n)^{\oplus r}&F\\};
\path[->]
(m-2-1) edge node [above] {$\phi$} (m-2-2);
\path[->>] (m-1-2) edge (m-2-2);
\end{tikzpicture}
$$ 
\begin{theorem}\label{theorem8'}
The following isomorphism of stacks holds true:
\begin{enumerate}
\item $$\mathfrak{M}^{(r,P_{F})}_{\text{HFT}}(\tau')\cong \left[\frac{\mathfrak{S}(\tau')}{\text{GL}(V_{F})}\right].$$
\item Moreover, there exists an isomorphism of stacks:$$\mathfrak{M}^{(r,P_{F})}_{\text{FT}}(\tau')\cong \left[\frac{\mathfrak{S}(\tau')}{\text{GL}_{r}(\mathbb{C})\times \text{GL}(V_{F})}\right].$$
\end{enumerate}
\end{theorem}
\begin{proof} Part (1): The proof is immediate, following the definition of algebraic quotient stacks. Part (2) follows from part (1) and the fact that $\mathfrak{M}^{(r,P_{F})}_{\text{HFT}}(\tau')$ is a $\text{GL}_{r}(\mathbb{C})$ torsor over $\mathfrak{M}^{(r,P_{F})}_{\text{FT}}(\tau')$ (c.f. Remark \ref{forget-BG}).
\end{proof}

Now we study automorphisms of highly frozen triples.
\begin{lemma}\label{pf-auto-1}
Given a $\tau'$-stable highly frozen triple $(E,F,\phi,\psi)$ as in Definition \ref{defn7} and a commutative diagram
\begin{center}
\begin{tikzpicture}
back line/.style={densely dotted}, 
cross line/.style={preaction={draw=white, -, 
line width=6pt}}] 
\matrix (m) [matrix of math nodes, 
row sep=1em
, column sep=3.5em, 
text height=1.5ex, 
text depth=0.25ex]{ 
\mathcal{O}_{X}(-n)^{\oplus r}&E&F\\
\mathcal{O}_{X}(-n)^{\oplus r}&E&F\\};
\path[->]
(m-1-2) edge node [above] {$\phi$} (m-1-3)
(m-1-1) edge node [left] {$id$}(m-2-1)
(m-1-1) edge node [above] {$\psi^{-1}$}(m-1-2)
(m-1-2) edge (m-2-2)
(m-1-3) edge node [right] {$\rho$} (m-2-3)
(m-2-1) edge node [above] {$\psi^{-1}$}(m-2-2)
(m-2-2) edge node [above] {$\phi$} (m-2-3);
\end{tikzpicture},
\end{center}
the map $\rho$ is given by $\text{id}_{F}$.
\end{lemma}
\begin{proof} Since $\psi$ is a choice of isomorphism, for simplicity replace $E$ by $\mathcal{O}_{X}(-n)^{\oplus r}$ and consider the diagram: 
\begin{equation}\label{auto-1}
\begin{tikzpicture}
back line/.style={densely dotted}, 
cross line/.style={preaction={draw=white, -, 
line width=6pt}}] 
\matrix (m) [matrix of math nodes, 
row sep=1em
, column sep=3.5em, 
text height=1.5ex, 
text depth=0.25ex]{ 
\mathcal{O}_{X}(-n)^{\oplus r}&F\\
\mathcal{O}_{X}(-n)^{\oplus r}&F\\};
\path[->]
(m-1-1) edge node [above] {$\phi$} (m-1-2)
(m-1-1) edge node [left] {$id$}(m-2-1)
(m-1-2) edge node [right] {$\rho$} (m-2-2)
(m-2-1) edge node [above] {$\phi$} (m-2-2);
\end{tikzpicture},
\end{equation}
Then, diagram \eqref{auto-1} induces:
\begin{center}
\begin{tikzpicture}
back line/.style={densely dotted}, 
cross line/.style={preaction={draw=white, -, 
line width=6pt}}] 
\matrix (m) [matrix of math nodes, 
row sep=1em
, column sep=3.5em, 
text height=1.5ex, 
text depth=0.25ex]{ 
\mathcal{O}_{X}(-n)^{\oplus r}&\text{Im}(\phi)&E\\
\mathcal{O}_{X}(-n)^{\oplus r}&\text{Im}(\phi)&E\\};
\path[right hook->] (m-1-2) edge (m-1-3);
\path[->] (m-1-1) edge node [left] {$id$}(m-2-1);
\path[->>] (m-1-1) edge node [above] {$\phi$} (m-1-2);
\path[->] (m-1-2) edge node [right] {$\rho\mid_{\text{Im}(\phi)}$} (m-2-2);
\path[->] (m-1-3) edge node [right] {$\rho$} (m-2-3);
\path[->>] (m-2-1) edge node [above] {$\phi$} (m-2-2);
\path[right hook->] (m-2-2) edge (m-2-3);
\end{tikzpicture}.
\end{center}
By commutativity of \eqref{auto-1}, $\rho\circ \phi=\phi\circ \text{id}=\phi$,  then $\rho(\text{Im}(\phi))=\text{Im}(\phi)$. Hence $\rho(\text{Im}(\phi))\subset \text{Im}(\phi)$. It follows that $\rho\mid_{\text{Im}(\phi)}=\text{id}_{\text{Im}(\phi)}$. Indeed if $s\in \text{Im}(\phi)(\mathcal{U})$, where $\mathcal{U}\subset X$ is affine open with $\tilde{s}\in \mathcal{O}_{X}(-n)^{\oplus r}(\mathcal{U})$ satisfying $\phi(\tilde{s})=s$, then $\rho(s)=\rho(\phi(\tilde{s}))=\phi(\text{id}(\tilde{s}))=\phi(\tilde{s})=s$. Now apply $\text{Hom}(-,F)$ to the short exact sequence $0\rightarrow \text{Im}(\phi)\rightarrow F\rightarrow Q\rightarrow 0$, where $Q$ denotes the corresponding cokernel. One obtains:$$0\rightarrow \text{Hom}(Q,F)\rightarrow \text{Hom}(F,F)\rightarrow \text{Hom}(\text{Im}(\phi),F).$$Since $(E,F,\phi,\psi)$ is $\tau'$-stable then by Lemma \ref{lemma2}, $Q$ is a sheaf with 0-dimensional support. Hence by purity of $F$, $\text{Hom}(Q,F)\cong 0$. Hence one obtains an injection $\text{Hom}(F,F)\hookrightarrow \text{Hom}(\text{Im}(\phi),F).$ Now $$\rho\mid_{\text{Im}(\phi)}=\text{id}_{\text{Im}(\phi)}=(\text{id}_{F})\mid_{\text{Im}(\phi)}.$$So $\rho=\text{id}_{F}$.
\end{proof}

\begin{theorem}\label{pf-DM}
The moduli stack $\mathfrak{M}^{(r,P_{F})}_{\text{HFT}}(\tau')$ is given by an algebraic scheme. 
\end{theorem}
\begin{proof} It is enough to show that for every $\mathbb{C}$-point $p:= \mathfrak{M}^{(r,P_{F})}_{\text{HFT}}(\tau')(\text{Spec}(\mathbb{C}))$ its stabilizer group $\textit{Stab}_{\mathfrak{M}^{(r,P_{F})}_{\text{HFT}}(\tau')}(p)$ is finite group with one element, given as identity.  Since the point $p$ is represented by a $\tau'$-stable highly frozen triple $(E,F,\phi,\psi)$, then $\textit{Stab}_{\mathfrak{M}^{(r,P_{F})}_{\text{HFT}}(\tau')}(p)$ is obtained by the automorphism group of $(E,F,\phi,\psi)$ which is given by the identity element, following Lemma \ref{pf-auto-1}

\end{proof}

\section{Deformation theory of of triples}\label{sec5}
As was shown, the construction of the moduli stack of stable frozen triples depends on choices of two fixed large enough integers $n\gg0$ and $n'\gg 0$. The first integer appears in the description $E:=\mathcal{O}_{X}(-n)^{\oplus r}$ and the latter is the one for which $F(n')$ is globally generated. The fact that the sheaf $F(n')$ is globally generated for large enough values of $n'$ does not a priori imply that $H^{i}(F(n))=0$ for all $i>0$ and our fixed choice of $n$. Therefore, we fix this issue by introducing the following definition:
\begin{defn}\label{open-sub}
Consider $\mathfrak{M}^{(r,P_{F})}_{\text{HFT}}(\tau')$ and $\mathfrak{M}^{(r,P_{F})}_{\text{FT}}(\tau')$ in definitions \ref{HFT-def} and \ref{defn141} respectively. Define the open substacks $\mathfrak{N}^{(r,P_{F})}_{\text{HFT}}(\tau')\subset \mathfrak{M}^{(r,P_{F})}_{\text{HFT}}(\tau')$ and $\mathfrak{N}^{(r,P_{F})}_{\text{FT}}(\tau')\subset \mathfrak{M}^{(r,P_{F})}_{\text{FT}}(\tau')$ such that:
\begin{enumerate}
\item $\mathfrak{N}^{(r,P_{F})}_{\text{HFT}}(\tau')=\{(E,F,\phi,\psi)\in \mathfrak{M}^{(r,P_{F})}_{\text{HFT}}(\tau')\mid H^{1}(F(n))=0\}$.
\item $\mathfrak{N}^{(r,P_{F})}_{\text{FT}}(\tau')=\{(E,F,\phi)\in \mathfrak{M}^{(r,P_{F})}_{\text{FT}}(\tau')\mid H^{1}(F(n))=0\}$.
\end{enumerate}
\end{defn}
From now on, we will omit ``$(r,P_{F})$" and ``$(\tau')$" in the notation used for our moduli stacks, in order to avoid notational complexity. Moreover, all our calculations are carried out over $\mathfrak{N}_{\text{HFT}}$ and $\mathfrak{N}_{\text{FT}}$ and the results in the following sections hold true for $\mathfrak{N}_{\text{HFT}}$ and $\mathfrak{N}_{\text{FT}}$ only. Also we assume that it is implicitly understood that in the following sections by the ``\textit{moduli stack of frozen or highly frozen triples}" we mean the open substack of the corresponding moduli stacks as  in Definition \ref{open-sub}. 

\subsection*{Deformation space of frozen and highly frozen triples} \label{sec6}
First we state the result of Illusie with no proof;
\begin{theorem}\label{Illusie1}
\cite[Section IV 3.2.12]{a29}. Given a graded morphism of graded modules $\mathcal{O}_{X\times S}(-n)^{\oplus r}\xrightarrow{\phi}\mathcal{F}$, there exists an element $$ob\in \text{Ext}^{2}_{\mathcal{D}^{b}(X\times S)}(\text{Cone}(\phi),\mathcal{I}\otimes \mathcal{F})$$ whose vanishing is necessary and sufficient in order to obtain nontrivial flat deformations of $\mathcal{O}_{X\times S}(-n)^{\oplus r}\xrightarrow{\phi}\mathcal{F}$. If $ob=0$ then the set of isomorphism classes of deformations forms a torsor under $\text{Ext}^{1}_{\mathcal{D}^{b}(X\times S)}(\text{Cone}(\phi),\mathcal{I}\otimes \mathcal{F})$.
\end{theorem} 
Now we apply Theorem \ref{Illusie1} to our case and obtain the following proposition;
\begin{prop}\label{tripledef2}
Given a $\tau'$-stable highly frozen triple $(E,F,\phi,\psi)$, represented by the complex $I^{\bullet}:\mathcal{O}_{X}(-n)^{\oplus r}\rightarrow F$, its space of infinitesimal deformations is given by $\text{Hom}(I^{\bullet},F)$.
\end{prop}
\begin{proof} A square zero embedding $S\hookrightarrow S'$ is a closed immersion whose defining ideal $\mathcal{I}$ satisfies $\mathcal{I}^{2}=0$. Note that here, $\text{Cone}(\phi)=I^{\bullet}_{S}[1]$. Now apply Illusie's result and see that the obstructions $ob:\text{Cone}(\phi)\rightarrow \mathcal{I}\otimes \mathcal{F}$ are given by the composite morphism \cite[Section 3.2.14.3]{a29}:
\begin{align} \label{Illusie2}
 &\text{Cone}(\phi)\rightarrow k^{1}\left(L_{\mathcal{O}_{X\times S}\oplus \mathcal{O}_{X\times S}(-n)^{\oplus r}/\mathcal{O}_{X}}\otimes \mathcal{F}[1]\right)\notag\\
 &
 \rightarrow k^{1}\left(\mathcal{I}\otimes (\mathcal{O}_{X\times S}(-n)\oplus \mathcal{O}_{X\times S}(-n)^{\oplus r})\otimes (\mathcal{O}_{X\times S}\oplus \mathcal{F})\right)\rightarrow \mathcal{I}\otimes \mathcal{F}[2].\notag\\
 \end{align}
 which induce the composite morphism:
 \begin{equation}
 \text{Cone}(\phi)\rightarrow L_{\mathcal{O}_{X\times S}/\mathcal{O}_{X}}\otimes \mathcal{F}[1]\rightarrow \mathcal{I} \otimes \mathcal{F}[2],
 \end{equation}
the set of such composite homomorphisms is given by $\text{Hom}(I^{\bullet}_{S}[1],\mathcal{I}\otimes \mathcal{F}[2])\cong \text{Ext}^{1}(I^{\bullet}_{S},\mathcal{I}\otimes \mathcal{F})\cong \text{Ext}^{1}(I^{\bullet}_{S},\mathcal{F})\otimes \mathcal{I}$, similarly if $ob=0$, then the set of isomorphism classes of deformations of highly frozen triples makes a torsor under $$\text{Ext}^{1}(I^{\bullet}_{S}[1],\mathcal{I}\otimes \mathcal{F})\cong \text{Hom}(I^{\bullet}_{S},\mathcal{I}\otimes \mathcal{F})\cong  \text{Hom}(I^{\bullet}_{S},\mathcal{F})\otimes \mathcal{I}.$$
\end{proof} 
 \begin{prop}\label{tripledef}
 The tangent space of $\mathfrak{N}_{FT}$ at a point $\{p\}:(E,F,\phi)$ represented by a complex $I^{\bullet}:=[E\rightarrow F]$ satisfies the following identity:
 \begin{equation}
 \text{T}_{p}\mathfrak{N}_{\text{FT}}\cong \text{Hom}(I^{\bullet},F)/\text{Im}(\mathfrak{g}l_{r}(\mathbb{C})\rightarrow \text{Hom}(I^{\bullet},F)).
 \end{equation}
 or equivalently; 
 \begin{equation}
 \text{T}_{p}\mathfrak{N}_{\text{FT}}\cong \text{Coker}\left[\mathfrak{g}l_{r}(\mathbb{C})\rightarrow \text{Hom}(I^{\bullet},F)\right].
 \end{equation}
\end{prop} 
\begin{proof} Since our analysis is over a point in the moduli stack, we assume that $S=Spec(\mathbb{C})$ and $S'$ is a square-zero extension over $S$. Following Remark \ref{forget-BG} we have that, $\mathfrak{N}_{\text{HFT}}$ is a $\text{G}l_{r}(\mathbb{C})$-torsor over $\mathfrak{N}_{\text{FT}}$. Therefore, in the level of tangent spaces we obtain:
\begin{equation}
\mathfrak{g}l_{r}(\mathbb{C})\rightarrow T_{p}\mathfrak{N}_{\text{HFT}}\rightarrow \text{T}_{p}\mathfrak{N}_{\text{FT}}\rightarrow 0,
\end{equation} 
hence it is immediately seen that $$T_{p}\mathfrak{N}_{\text{FT}}\cong \text{Coker}[\mathfrak{g}l_{r}(\mathbb{C})\rightarrow T_{p}\mathfrak{N}_{\text{HFT}}].$$ Now use the fact that, $T_{p}\mathfrak{N}_{\text{HFT}} \cong \text{Hom}(I^{\bullet},\mathcal{F})$ by Proposition \ref{tripledef2}.
\end{proof}
A similar analysis can be carried out when $S$ is given by an affine scheme and $S'$ is an $S$-scheme. We state this result without further proof;
\begin{theorem}\label{Gamma1}
Fix a map $f: S\rightarrow \mathfrak{N}_{\text{FT}}$. Fixing $f$ corresponds to fixing an $S$-flat family of frozen triples given by $[\mathcal{O}_{X}(-n)\boxtimes \mathcal{M}_{S}\rightarrow \mathcal{F}]$ as in Definition \ref{defn4}. Let $S'$ be a square-zero extension of $S$ with ideal $\mathcal{I}$. Let $\mathcal{D}ef_{S}(S',\mathfrak{N}_{\text{FT}})$ denote the deformation space of the map $f$ obtained by the set of possible deformations, $f':S'\rightarrow \mathfrak{N}_{\text{FT}}$. The following statement is true:
\begin{equation}
\mathcal{D}ef_{S}(S',\mathfrak{N}_{\text{FT}})\cong \text{Hom}(I^{\bullet}_{S},F)\otimes \mathcal{I}\slash\text{Im}\bigg((\mathfrak{g}l_{r}(\mathcal{O}_{S})\rightarrow \text{Hom}(I^{\bullet}_{S},\mathcal{F}))\otimes \mathcal{I}\bigg)
\end{equation} 
\end{theorem}
\subsection*{Deformation of frozen triples as objects in the derived category}
\begin{lemma}\label{vanish-2}
Let $I^{\bullet}:=[\mathcal{O}_{X}(-n)^{\oplus r}\xrightarrow{\phi}F]$ correspond to a point of $\mathfrak{N}_{\text{HFT}}$ or $\mathfrak{N}_{\text{FT}}$. Then:
\begin{equation*}
\text{Ext}^{2}(F, \mathcal{O}_{X}(-n))\cong 0\cong \text{Ext}^{1}(F,\mathcal{O}_{X}(-n)).
\end{equation*} 
\end{lemma}
\begin{proof} Use Serre duality and obtain:
\begin{align*}
&
\text{Ext}^{i}(F,\mathcal{O}_{X}(-n))\cong (\text{Ext}^{3-i}(\mathcal{O}_{X}(-n), F \otimes \omega_{X})^{\vee}
\cong \text{Ext}^{3-i}(\mathcal{O}_{X}(-n), F)^{\vee} \cong H^{3-i}(F(n))^{\vee}.
\end{align*}
The statement follows from the definitions of $\mathfrak{N}_{\text{HFT}}$ and $\mathfrak{N}_{\text{FT}}$.
\end{proof}
Let $I^{\bullet}\in \mathcal{D}^{b}(X)$ represent the complex $I^{\bullet}:=\mathcal{O}_{X}(-n)^{\oplus r}\xrightarrow{\phi}F$ with $\mathcal{O}_{X}(-n)^{\oplus r}$ in degree 0 and $F$ in degree 1.
Let $K:=\text{Ker}(\phi)$ and $Q:=\text{Coker}(\phi)$. There exist the following exact triangles in the derived category:
\begin{equation}\label{artan1}
F[-1]\rightarrow I^{\bullet} \rightarrow \mathcal{O}_{X}(-n)^{\oplus r}\rightarrow F \rightarrow \cdots
\end{equation}
\begin{equation}\label{artan2}
K\rightarrow I^{\bullet} \rightarrow Q[-1]\rightarrow K[1]\rightarrow \cdots
\end{equation}
\begin{lemma}
Suppose that a frozen triple $(E,F,\phi)$ of type $(r, P_F)$ is $\tau'$-stable. Then $\text{Ext}^{\leq -1}(I^{\bullet},I^{\bullet})=0$. 
\end{lemma}
\begin{proof} Note that $\text{Ext}^{k}(I^{\bullet},I^{\bullet})=0$ for $k\leq -2$ for degree reasons. We now consider $k=-1$. Apply $\text{Hom}(I^{\bullet},\cdot)$ to \eqref{artan1} and obtain:
\begin{align}\label{I,E}
&
\text{Ext}^{-2}(I^{\bullet},F)\rightarrow \text{Ext}^{-1}(I^{\bullet},I^{\bullet})
\rightarrow \text{Ext}^{-1}(I^{\bullet},\mathcal{O}_{X}^{\oplus r}(-n))\rightarrow \text{Ext}^{-1}(I^{\bullet},F)
\end{align}
Now apply $\text{Hom}(\cdot,F)$ to \eqref{artan2} and obtain:
\begin{align}\label{Q,E}
&
\cdots\rightarrow \text{Ext}^{i}(Q[-1],F)\rightarrow \text{Ext}^{i}(I^{\bullet},F)\rightarrow 
\text{Ext}^{i}(K,F)\rightarrow \text{Ext}^{i+1}(Q[-1],F)\cdots
\end{align}
It is easy to see that $\text{Ext}^{-2}(Q[-1],F)\cong 0$, $\text{Ext}^{-2}(K,F)\cong 0$ and $\text{Ext}^{-1}(K,F)\cong 0$ for degree reasons. Moreover, $\text{Ext}^{-1}(Q[-1],F)= \text{Hom}(Q,F)\cong 0$ since $Q$ is zero dimensional (by $\tau'$-stability) and $F$ is of pure dimension one. Hence 
\begin{equation}\label{eq12}
\text{Ext}^{-2}(I^{\bullet},F)\cong 0  \;\; \text{and} \;\;  \text{Ext}^{-1}(I^{\bullet},F)\cong 0,
\end{equation}
and therefore $\text{Ext}^{-1}(I^{\bullet},I^{\bullet})\cong \text{Ext}^{-1}(I^{\bullet},\mathcal{O}^{\oplus r}_{X}(-n))$. Now apply $\text{Hom}(\cdot, \mathcal{O}^{\oplus r}_{X}(-n))$ to \eqref{artan1} and obtain:
\begin{align}
&
\text{Ext}^{-1}(F,\mathcal{O}^{\oplus r}_{X}(-n))\rightarrow \text{Ext}^{-1}(\mathcal{O}^{\oplus r}_{X}(-n),\mathcal{O}^{\oplus r}_{X}(-n))\rightarrow\text{Ext}^{-1}(I^{\bullet},\mathcal{O}^{\oplus r}_{X}(-n))\notag\\
&
\rightarrow \text{Hom}(F,\mathcal{O}^{\oplus r}_{X}(-n)).
\end{align} 
Now $\text{Ext}^{-1}(\mathcal{O}^{\oplus r}_{X}(-n),\mathcal{O}^{\oplus r}_{X}(-n))\cong 0$ by degree reasons and $\text{Hom}(F, \mathcal{O}^{\oplus r}_{X}(-n))\cong 0$ by purity of $\mathcal{O}^{\oplus r}_{X}(-n)$. Hence $\text{Ext}^{-1}(I^{\bullet},\mathcal{O}^{\oplus r}_{X}(-n))\cong 0$ and $\text{Ext}^{-1}(I^{\bullet},I^{\bullet})\cong \text{Ext}^{-1}(I^{\bullet},\mathcal{O}^{\oplus r}_{X}(-n))\cong 0.$
\end{proof} 
\begin{lemma}\label{lemma-higherank}
Let $I^{\bullet}\in \mathcal{D}^{b}(X)$ represent a $\tau'$-stable frozen triple $(E,F,\phi)$ of type $(r, P_F)$. Then there exists an injective map: $\text{Hom}(I^{\bullet},I^{\bullet})\hookrightarrow \text{End}(\mathcal{O}_{X}(-n)^{\oplus r})$.
\end{lemma}
\begin{proof} Apply $\text{Hom}(I^{\bullet},\cdot)$ to \eqref{artan1} and obtain the following exact sequence:
\begin{align}\label{5-24}
&
\text{Ext}^{-1}(I^{\bullet},F)\rightarrow \text{Hom}(I^{\bullet},I^{\bullet}) \rightarrow \text{Hom}(I^{\bullet},\mathcal{O}_{X}(-n)^{\oplus r})\rightarrow \text{Hom}(I^{\bullet},F)\notag\\
&
\rightarrow \text{Ext}^{1}(I^{\bullet},I^{\bullet})\rightarrow \text{Ext}^{1}(I^{\bullet},\mathcal{O}_{X}(-n)^{\oplus r})\rightarrow \text{Ext}^{1}(I^{\bullet},F)\rightarrow \text{Ext}^{2}(I^{\bullet},I^{\bullet})\cdot
\end{align}
Observe that the leftmost term in \eqref{5-24} vanishes because of degree reasons:
\begin{align}\label{5-23}
&
0\rightarrow \text{Hom}(I^{\bullet},I^{\bullet}) \rightarrow \text{Hom}(I^{\bullet},\mathcal{O}_{X}(-n)^{\oplus r})\rightarrow \text{Hom}(I^{\bullet},F)\notag\\
&
\rightarrow \text{Ext}^{1}(I^{\bullet},I^{\bullet})\rightarrow \text{Ext}^{1}(I^{\bullet},\mathcal{O}_{X}(-n)^{\oplus r})\rightarrow \text{Ext}^{1}(I^{\bullet},F)\rightarrow \cdots
\end{align}

Now apply $\text{Hom}(\cdot,\mathcal{O}_{X}(-n)^{\oplus r})$ to \eqref{artan1} and obtain:
\begin{align}\label{corner}
&
\text{Hom}(F,\mathcal{O}_{X}(-n)^{\oplus r})\rightarrow \text{End}(\mathcal{O}_{X}(-n)^{\oplus r})
\rightarrow \text{Hom}(I^{\bullet},\mathcal{O}_{X}(-n)^{\oplus r}) \notag\\
&\rightarrow \text{Ext}^{1}(F,\mathcal{O}_{X}(-n)^{\oplus r})
\end{align}
Using Lemma \ref{vanish-2}, it is immediately seen that the leftmost and the rightmost terms in \eqref{corner} vanish. Hence $\text{End}(\mathcal{O}_{X}(-n)^{\oplus r})\cong \text{Hom}(I^{\bullet},\mathcal{O}_{X}(-n)^{\oplus r}).$ Hence it is seen from \eqref{5-23} that $\text{Hom}(I^{\bullet},I^{\bullet}) \rightarrow \text{End}(\mathcal{O}_{X}(-n)^{\oplus r})$ is injective.
\end{proof}
\begin{theorem}\label{theorem10}
Let $p \in \mathfrak{N}_{\text{FT}}$ be a point represented by the complex with fixed determinant $I^{\bullet}:=\mathcal{O}_{X}(-n)^{\oplus r}\xrightarrow{\phi} F$. The following is true: $$\text{T}_{p}\mathfrak{N}_{\text{FT}} \cong \text{Ext}^{1}(I^{\bullet},I^{\bullet})_{0}.$$
\end{theorem}
\begin{proof} Consider the exact sequence in \eqref{5-23}. Now use the result of Lemma \ref{lemma-higherank} to obtain the following exact sequence:
\begin{align}\label{5-234}
&
\text{Hom}(I^{\bullet},I^{\bullet}) \hookrightarrow \mathfrak{g}l_{r}(\mathbb{C})\rightarrow \text{Hom}(I^{\bullet},F)
\rightarrow \text{Ext}^{1}(I^{\bullet},I^{\bullet})\rightarrow \text{Ext}^{1}(I^{\bullet},\mathcal{O}_{X}(-n)^{\oplus r})\notag\\
&
\rightarrow \text{Ext}^{1}(I^{\bullet},F)\rightarrow \cdot\notag\\
\end{align}
where we have replaced $\text{End}(\mathcal{O}_{X}(-n)^{\oplus r})$ with $\mathfrak{g}l_{r}(\mathbb{C})$.
Now recall that $H^{1}(\mathcal{O}_{X})\cong 0$ by assumption. In that case, assuming that the complex $I^{\bullet}$ has fixed determinant, then we have $\text{Ext}^{i}(I^{\bullet},I^{\bullet})_{0}\cong \text{Ext}^{i}(I^{\bullet},I^{\bullet})$. Hence, the exact sequence in \eqref{5-234} is rewritten as:
\begin{align}\label{aslekar}
&
0\rightarrow \text{Hom}(I^{\bullet},I^{\bullet})\rightarrow \mathfrak{g}l_{r}(\mathbb{C})\rightarrow \text{Hom}(I^{\bullet},F)
\rightarrow \text{Ext}^{1}(I^{\bullet},I^{\bullet})_{0}\rightarrow 0\rightarrow \text{Ext}^{1}(I^{\bullet},F).
\end{align} 
Hence we obtain $\text{Hom}(I^{\bullet},F)/\text{Im}[\mathfrak{g}l_{r}(\mathbb{C})\rightarrow \text{Hom}(I^{\bullet},F)]\cong \text{Ext}^{1}(I^{\bullet},I^{\bullet})_{0}$. Now use Proposition \ref{tripledef} and obtain $\text{T}_{p}\mathfrak{N}_{\text{FT}}\cong \text{Ext}^{1}(I^{\bullet},I^{\bullet})_{0}.$
\end{proof}
\begin{cor}\label{theorem12}
Let $I^{\bullet}_{S}$ be defined as in Theorem \ref{theorem10}. The higher order deformation $I^{\bullet}_{S'}$ over $\acute{S}$ of $I^{\bullet}_{S}$ with trivial determinant is quasi-isomorphic to a complex: $$[\mathcal{O}_{X\times_{\mathbb{C}} \acute{S}}(-n)^{\oplus r}\xrightarrow{\phi'}\mathcal{\acute{F}}]$$
\end{cor}
\begin{proof} 
This is a direct consequence of Theorem \ref{theorem10}.
\end{proof}
\section*{Deformation-obstruction theories and virtual fundamental class}\label{sec7}
By Theorem \ref{theorem8'} the moduli stack of stable frozen triples $ \mathfrak{N}_{\text{FT}}$ is an Artin stack. The definition of a perfect deformation-obstruction theory for $ \mathfrak{N}_{\text{FT}}$ is as follows:
\begin{defn}\label{perfect-amp}
Following \cite{a61} and \cite{a66}, a perfect deformation-obstruction theory for $\mathfrak{N}_{\text{FT}}$ is given by a perfect  3-term complex $\mathbb{E}^{\bullet\vee}$ of strongly perfect amplitude $[-1,1]$ and a map in the derived category $\mathbb{E}^{\bullet\vee}\xrightarrow{\phi} \mathbb{L}^{\bullet}_{\mathfrak{N}_{\text{FT}}}$ such that $h^{1}(\phi)$ and $h^{0}(\phi)$ are isomorphisms and $h^{-1}(\phi)$ is an epimorphism. Here $ \mathbb{L}^{\bullet}_{\mathfrak{N}_{\text{FT}}}$ is the truncated cotangent complex of the Artin moduli stack of $\tau'$-stable frozen triples concentrated in degrees $-1$, 0 and 1 whose pullback via the projection map $\pi:\mathfrak{N}_{\text{HFT}}\rightarrow \mathfrak{N}_{\text{FT}}$ has the form:$$\pi^{*}\mathbb{L}^{\bullet}_{\mathfrak{N}_{\text{FT}}}:= \mathcal{I}/\mathcal{I}^{2}\rightarrow \Omega_{\mathfrak{A}}\mid_{ \mathfrak{N}_{\text{HFT}}}\rightarrow (\mathfrak{g}l_{r}(\mathbb{C}))^{\vee}\otimes \mathcal{O}_{ \mathfrak{N}_{\text{HFT}}}.$$Note that $(\mathfrak{g}l_{r}(\mathbb{C}))^{\vee}\otimes \mathcal{O}_{ \mathfrak{N}_{\text{HFT}}}\cong \Omega_{\pi}$, $\mathfrak{A}$ denotes an ambient smooth Artin stack and $\mathcal{I}$ is the ideal corresponding to the embedding $\mathfrak{N}_{\text{HFT}}\hookrightarrow \mathfrak{A}$. 
\end{defn}

\subsection*{Deformation-obstruction theory of amplitude $[-2,1]$ over $\mathfrak{N}_{\text{FT}}$}\label{sec8}
In what follows, we use the following notation:$$\pi_{\mathfrak{N}}:X\times \mathfrak{N}_{\text{FT}}\rightarrow \mathfrak{N}_{\text{FT}}\,\,\, \text{and}\,\,\, \pi_{X}:X\times\mathfrak{N}_{\text{FT}}\rightarrow X.$$ 
\begin{theorem}\label{reldef-f}
(a). There exists a map in the derived category given by:$$R\pi_{\mathfrak{N}}\left(R\mathscr{H}om(\mathbb{I}^{\bullet},\mathbb{I}^{\bullet})_{0}\otimes \pi_{X}^{*}\omega_{X}\right)[2]\xrightarrow{ob} \mathbb{L}^{\bullet}_{\mathfrak{N}_{\text{FT}}}.$$ 
(b). After suitable truncations, there exists a 4 term complex $\mathbb{E}^{\bullet}$ of locally free sheaves , such that $\mathbb{E}^{\bullet\vee}$ is self-symmetric of amplitude $[-2,1]$ and there exists a map in the derived category:
\begin{equation}\label{defobs-froz}
\mathbb{E}^{\bullet\vee}\xrightarrow{ob^{t}} \mathbb{L}^{\bullet}_{\mathfrak{N}_{\text{FT}}},
\end{equation}
such that $h^{-1}(ob^{t})$ is surjective, and $h^{0}(ob^{t})$ and $h^{1}(ob^{t})$ are isomorphisms. 
\end{theorem}
\begin{proof} 
Here we prove (a). Consider the universal complex: 
$$\mathbb{I}^{\bullet}=[M\otimes \mathcal{O}_{X\times \mathfrak{N}_{\text{FT}}}(-n) \rightarrow \mathbb{F}]\in \mathcal{D}^{b}(X\times \mathfrak{N}).$$ 
Since the composition of the maps $id:\mathcal{O}_{X\times \mathfrak{N}_{\text{FT}}}\rightarrow R\mathscr{H}om(\mathbb{I}^{\bullet},\mathbb{I}^{\bullet})$ and $tr:R\mathscr{H}om(\mathbb{I}^{\bullet},\mathbb{I}^{\bullet})\rightarrow \mathcal{O}_{X\times \mathfrak{N}_{\text{FT}}}$ is multiplication by $rk(\mathbb{I}^{\bullet})$, one obtains a splitting
$$R\mathscr{H}om(\mathbb{I}^{\bullet},\mathbb{I}^{\bullet})\cong R\mathscr{H}om(\mathbb{I}^{\bullet},\mathbb{I}^{\bullet})_{0}\oplus \mathcal{O}_{X\times \mathfrak{N}_{\text{FT}}}$$
Recall that by part (2) of Theorem \ref{theorem8'}, $\mathfrak{N}_{\text{FT}}=[\frac{\mathfrak{S}(\tau')}{G}]$ where $G=\text{GL}_{r}(\mathbb{C})\times \text{GL}(V_{F})$. For simplicity denote $\mathfrak{S}:=\mathfrak{S}(\tau')$. Let $\mathbb{I}^\bullet_{\mathfrak{S}}$ denote the pullback of $\mathbb{I}^\bullet$ to $X\times \mathfrak{S}$.  We write 
$L^\bullet$ to mean the full, untruncated cotangent complex, and write $\mathbb{L}^\bullet = \tau^{\geq -1}L^\bullet$ for the truncated cotangent complex.  
Consider the {\em Atiyah class} $\mathbb{I}_{\mathfrak{S}}^\bullet\rightarrow L_{X\times \mathfrak{S}}^\bullet \otimes \mathbb{I}_{\mathfrak{S}}^\bullet[1]$ defined by
Illusie \cite[Section IV2.3.6]{a29}.  The Atiyah class can be identified with a class in 
$\text{Ext}^1(\mathbb{I}_{\mathfrak{S}}, L_{X\times \mathfrak{S}}^\bullet \otimes \mathbb{I}_{\mathfrak{S}}^\bullet)$.  
  The composite 
$$\mathbb{I}_{\mathfrak{S}}^\bullet\rightarrow L_{X\times \mathfrak{S}}^\bullet \otimes \mathbb{I}_{\mathfrak{S}}^\bullet[1]
\rightarrow \tau^{\geq -1}L^\bullet_{X\times \mathfrak{S}}\otimes \mathbb{I}_{\mathfrak{S}}^\bullet[1]
= \mathbb{L}^\bullet_{X\times \mathfrak{S}}\otimes \mathbb{I}^\bullet_{\mathfrak{S}}[1]
$$ is the truncated Atiyah class of 
\cite[Section 2.2]{a10}. By \cite[Proposition 2.1.10]{a9}  the complex $\mathbb{I}^\bullet$ is perfect.  It then follows from \cite[Corollaire IV.2.3.7.4]{a29} that the composite $\mathbb{I}_{\mathfrak{S}}^\bullet\rightarrow \mathbb{L}^\bullet_{X\times \mathfrak{S}}\otimes \mathbb{I}^\bullet_{\mathfrak{S}}[1] \rightarrow \Omega^1_{X\times \mathfrak{S}}\otimes \mathbb{I}_{\mathfrak{S}}^\bullet[1]$, when identified with a 1-extension, agrees with the canonical 1-extension
\begin{equation}\label{PP}
0\rightarrow \Omega^1_{X\times\mathfrak{S}}\otimes \mathbb{I}_{\mathfrak{S}}^\bullet \rightarrow \mathcal{P}^1_{X\times \mathfrak{S}}\otimes\mathbb{I}^\bullet_{\mathfrak{S}} \rightarrow \mathbb{I}^\bullet_{\mathfrak{S}} \rightarrow 0,
\end{equation}
defined by tensoring with the first-order principal parts
$\mathcal{P}^1_{X\times \mathfrak{S}}$. We want to show that the Atiyah class descends to $X\times \mathfrak{N}_{\text{FT}} = X\times[\frac{\mathfrak{S}}{G}]$ where $G=\text{GL}_{r}(\mathbb{C})\times\text{GL}(V_{F})$.  

More precisely, this means the following;  Let
$q_{\mathfrak{N}}: X\times \mathfrak{S}\rightarrow X\times \mathfrak{N}_{\text{FT}}$ denote the projection.   Then we want a morphism
$$\mathbb{I}^\bullet\rightarrow L_{X\times\mathfrak{N}_{\text{FT}}}^\bullet \otimes \mathbb{I}^\bullet[1]$$ on $\mathfrak{N}_{\text{FT}}$, such that the natural composite 
$$q_{\mathfrak{N}}^*\mathbb{I}^\bullet\rightarrow q_{\mathfrak{N}}^*L_{X\times\mathfrak{N}_{\text{FT}}}^\bullet \otimes q^*_{\mathfrak{N}}\mathbb{I}^\bullet[1]\rightarrow L_{X\times \mathfrak{S}}^\bullet\otimes \mathbb{I}_{\mathfrak{S}}^\bullet[1]$$ agrees with the Atiyah class of Illusie.  
The complex $\mathbb{I}^\bullet_{\mathfrak{S}}$ is $G$-equivariant by construction (it comes via pullback from $X\times\mathfrak{N}_{\text{FT}}$), and the construction of the Atiyah class shows that it too is naturally $G$-equivariant.  The pulled back cotangent complex 
 $q_{\mathfrak{N}}^*L_{X\times\mathfrak{N}_{\text{FT}}}^\bullet$ has the following description;
 
   There is a natural composite map
 $L_{X\times\mathfrak{S}}^\bullet \rightarrow \Omega^1_{X\times\mathfrak{S}}\rightarrow \mathfrak{g}^{\vee}\otimes\mathcal{O}_{X\times\mathfrak{S}}$, where
 the second map is dual to the infinitesimal $\mathfrak{g}$-action (and $\mathfrak{g} = \text{Lie}(G)$).  Then
 $q_{\mathfrak{N}}^*L_{X\times\mathfrak{N}_{\text{FT}}}^\bullet \simeq \text{Cone}[L_{X\times\mathfrak{S}}^\bullet\rightarrow \mathfrak{g}^{\vee}\otimes\mathcal{O}_{X\times\mathfrak{S}}][-1]$.  Thus, to prove that the Atiyah class descends to $X\times\mathfrak{N}_{\text{FT}}$ in the sense explained above, it suffices to show that the composite
 $$\mathbb{I}_{\mathfrak{S}}^\bullet\rightarrow \mathbb{L}^\bullet_{X\times\mathfrak{S}}\otimes \mathbb{I}^\bullet_{\mathfrak{S}}[1] \rightarrow \Omega^1_{X\times\mathfrak{S}}\otimes \mathbb{I}_{X\times\mathfrak{S}}^\bullet[1]\rightarrow \mathfrak{g}^\vee\otimes \mathbb{I}_{\mathfrak{S}}^\bullet[1]$$ represents an equivariantly split extension.  By the above discussion, this extension is obtained by pushing out the principal parts extension \eqref{PP} along the natural map 
 $\Omega^1_{X\times\mathfrak{S}}\otimes \mathbb{I}_{\mathfrak{S}}^\bullet \rightarrow \mathfrak{g}^{\vee}\otimes \mathbb{I}_{\mathfrak{S}}^\bullet$.  Just as a splitting of the principal parts extension corresponds to a choice of connection, however, a splitting of its pushout
 corresponds to a choice of an $L$-connection \cite[Section 4]{a68} where $L = \mathfrak{g}\otimes\mathcal{O}_{X\times\mathfrak{S}}$ is the action Lie algebroid associated to the
 infinitesimal $G$-action.  Since $\mathbb{I}^\bullet$ is $G$-equivariant, it comes equipped with a $\mathfrak{g}\otimes\mathcal{O}_{X\times\mathfrak{S}}$-connection, hence a $G$-equivariant splitting of the required $1$-extension.  It follows that the Atiyah class descends
 to $X\times\mathfrak{N}_{\text{FT}}$. 
 
 We now have the truncated Atiyah class of the universal complex, given by a class in
\begin{align}
&\text{Ext}^{1}_{X\times \mathfrak{N}_{\text{FT}}}(\mathbb{I}^{\bullet},\mathbb{I}^{\bullet} \otimes \mathbb{L}^{\bullet}_{X\times \mathfrak{N}_{\text{FT}}})\cong \text{Ext}^{1}_{X\times\mathfrak{N}_{\text{FT}}}(R\mathscr{H}om(\mathbb{I}^{\bullet},\mathbb{I}^{\bullet}),\mathbb{L}^{\bullet}_{X\times\mathfrak{N}_{\text{FT}}})\notag \\
&
\cong \text{Ext}^{1}_{X\times\mathfrak{N}_{\text{FT}}}(R\mathscr{H}om(\mathbb{I}^{\bullet},\mathbb{I}^{\bullet})_{0}\oplus \mathcal{O}_{X\times \mathfrak{N}_{\text{FT}}},\mathbb{L}^{\bullet}_{X\times\mathfrak{N}_{\text{FT}}}),
\end{align}
where $\mathbb{L}^{\bullet}_{X\times \mathfrak{N}_{\text{FT}}}$ denotes the truncated cotangent complex of $X\times \mathfrak{N}_{\text{FT}}$. Note that over $X\times \mathfrak{N}_{\text{FT}}$, $\mathbb{L}^{\bullet}_{X\times \mathfrak{N}_{\text{FT}}}=\pi_{X}^{*}\mathbb{L}^{\bullet}_{X}\oplus \pi_{\mathfrak{N}}^{*}\mathbb{L}^{\bullet}_{\mathfrak{N}_{\text{FT}}}$ and so  one obtains the following map between the $\text{Ext}$ groups:
\begin{align}
&
\text{Ext}^{1}_{X\times\mathfrak{N}_{\text{FT}}}(R\mathscr{H}om(\mathbb{I}^{\bullet},\mathbb{I}^{\bullet})_{0}\oplus \mathcal{O}_{X\times \mathfrak{N}_{\text{FT}}},\mathbb{L}^{\bullet}_{X\times\mathfrak{N}_{\text{FT}}})\rightarrow \text{Ext}^{1}_{X\times\mathfrak{N}_{\text{FT}}}(R\mathscr{H}om(\mathbb{I}^{\bullet},\mathbb{I}^{\bullet})_{0},\pi_{\mathfrak{N}}^{*}\mathbb{L}^{\bullet}_{\mathfrak{N}_{\text{FT}}}).
\end{align}
On the other hand:
\begin{align}\label{isom-series2}
&
\text{Ext}^{1}_{X\times\mathfrak{N}_{\text{FT}}}\left(R\mathscr{H}om(\mathbb{I}^{\bullet},\mathbb{I}^{\bullet})_{0},\pi_{\mathfrak{N}}^{*}\mathbb{L}^{\bullet}_{\mathfrak{N}_{\text{FT}}}\right)\notag\\
&
\cong \text{Ext}^{\text{\text{dim}}(X)+\text{\text{dim}}(\mathfrak{N}_{\text{FT}})-1}_{X\times \mathfrak{N}_{\text{FT}}}\left(\pi_{\mathfrak{N}}^{*}\mathbb{L}^{\bullet}_{\mathfrak{N}_{\text{FT}}},R\mathscr{H}om(\mathbb{I}^{\bullet},\mathbb{I}^{\bullet})_{0}\otimes \omega_{X\times \mathfrak{N}_{\text{FT}}}\right)^{\vee}\notag\\
&
\cong\text{Ext}^{\text{\text{dim}}(X)+\text{\text{dim}}(\mathfrak{N}_{\text{FT}})-1}_{\mathfrak{N}_{\text{FT}}}\left(\mathbb{L}^{\bullet}_{\mathfrak{N}_{\text{FT}}},R\pi_{\mathfrak{N}\,\ast}\left(R\mathscr{H}om(\mathbb{I}^{\bullet},\mathbb{I}^{\bullet})_{0}\otimes \omega_{X\times \mathfrak{N}_{\text{FT}}}\right)\right)^{\vee}\notag \\
&
\cong\text{Ext}^{\text{\text{dim}}(\mathfrak{N}_{\text{FT}})-[\text{\text{dim}}(X)+\text{\text{dim}}(\mathfrak{N}_{\text{FT}})-1]}_{\mathcal{M}^{P}_{\text{FT}}}\bigg(R\pi_{\mathfrak{N}\,\ast}\left(R\mathscr{H}om(\mathbb{I}^{\bullet},\mathbb{I}^{\bullet})_{0}\otimes \omega_{X\times \mathfrak{N}_{\text{FT}}}\right),\mathbb{L}^{\bullet}_{\mathfrak{N}_{\text{FT}}}\otimes \omega_{\mathfrak{N}_{\text{FT}}}\bigg)\notag\\
& 
\cong\text{Ext}^{-\text{\text{dim}}(X)+1}_{\mathfrak{N}}\bigg(R\pi_{\mathfrak{N}\,\ast}\left(R\mathscr{H}om(\mathbb{I}^{\bullet},\mathbb{I}^{\bullet})_{0}\otimes \omega_{X\times \mathfrak{N}}\right)\otimes R\pi_{\mathfrak{N}\,\ast}\pi_{\mathfrak{N}}^{*}\omega_{\mathfrak{N}_{\text{FT}}}^{-1}, \mathbb{L}^{\bullet}_{\mathfrak{N}_{\text{FT}}}\bigg),\notag\\
\end{align}
where the first isomorphism is obtained by Serre duality, the second isomorphism is induced by the adjointness property of  the left derived pullback and the right derived pushforward and the third isomorphism is obtained by Serre duality. By projection formula and the definition of the relative dualizing sheaf $\omega_{\pi_{\mathfrak{N}}}=\omega_{X\times\mathfrak{N}}\otimes \omega_{\mathfrak{N}}^{-1}=\pi_{X}^{*}\omega_{X}$ and since $X$ is a threefold, the last term in \eqref{isom-series2} is rewritten as:
\begin{equation}\label{obs}
R\pi_{\mathfrak{N}\,\ast}\left(R\mathscr{H}om(\mathbb{I}^{\bullet},\mathbb{I}^{\bullet})_{0}\otimes \pi_{X}^{*}\omega_{X}\right)[2]\rightarrow \mathbb{L}^{\bullet}_{\mathfrak{N}_{\text{FT}}}.
\end{equation}
Now for part (b) use Theorem \ref{theorem10} as well as \cite{a10,a17} and see that \eqref{obs} induces isomorphisms in $\text{h}^{0}$ and $\text{h}^{1}$ levels and a surjection at $\text{h}^{-1}$ level. Finally, we apply \cite[Lemma 2.10]{a17} to \eqref{obs} and see that the complex on the left hand side of \eqref{obs} is quasi-isomorphic to a 4 term complex of vector bundles. We briefly review their strategy here; Start from $A^{\bullet}$, a finite locally free resolution of $\mathbb{I}^{\bullet}$. Then$A^{\bullet\vee}\otimes A^{\bullet}$ is a finite locally free resolution of $R\mathscr{H}om(\mathbb{I}^{\bullet},\mathbb{I}^{\bullet})$ and moreover,$$A^{\bullet\vee}\otimes A^{\bullet}\cong \mathcal{O}_{X\times \mathfrak{N}_{\text{FT}}}\oplus (A^{\bullet\vee}\otimes A^{\bullet})_{0}$$where the complex $(A^{\bullet\vee}\otimes A^{\bullet})_{0}$ is the locally free resolution of $R\mathscr{H}om(\mathbb{I}^{\bullet},\mathbb{I}^{\bullet})_{0}$. Now let  $B^{\bullet}$ be a finite locally free resolution of $(A^{\bullet\vee}\otimes A^{\bullet})_{0}$, trimmed to start at least 4 places earlier than $(A^{\bullet\vee}\otimes A^{\bullet})_{0}$, such that $B^{\bullet}$ satisfies the property that $R^{\leq2}\pi_{*}B^{j}\cong 0$ and $R^{3}\pi_{*}B^{j}$ is locally free. Then following the same strategy as in \cite[Lemma 2.10]{a17}, we can define a complex $\mathbb{E}^{\bullet}$, whose terms are defined by $$E^{k}:=R^{3}\pi_{*}B^{k+3}$$ which is a finite complex of locally free sheaves, quasi-isomorphic to $R\mathscr{H}om(\mathbb{I}^{\bullet},\mathbb{I}^{\bullet})_{0}$ and by base change, the restriction of $\mathbb{E}^{\bullet}$ to a point $I^{\bullet}\in \mathfrak{N}_{\text{FT}}$ computes $\text{Ext}^{i}(I^{\bullet},I^{\bullet})_{0}$ which, by the local to global spectral sequence and Serre duality, is nonvanishing only for $0\leq i\leq 3$. Finally, we can use the vanishing cohomologies at degrees less than $0$ and greater than $3$ and trim the complex $\mathbb{E}^{\bullet}$ \cite[Lemma 2.10]{a17} (by replacing $E^{i}$ at its ends with the corresponding kernel and cokernel sheaves) and obtain a 4 term complex of vector bundles such that
\begin{equation}\label{vec-cplx}
\mathbb{E}^{\bullet}\cong^{q-isom}\{ E_{0}\to E_{1}\to E_{2}\to E_{3}\}.
\end{equation} 
Now by the above construction, we can see that $\mathbb{E}^{\bullet\vee}$ in \eqref{defobs-froz} satisfies the conditions of being a deformation-obstruction theory of perfect amplitude $[-2,1]$. 
\end{proof}
\subsection*{Deformation-obstruction theory of amplitude $[-1,0]$ over $\mathfrak{N}_{\text{HFT}}$}\label{section-fix}\label{sec9}
\begin{lemma}\label{self-dual}
The complex $\mathbb{E}^{\bullet}$ in Theorem \ref{reldef-f} is self-dual in the sense of \cite{a1}. In other words, there exists a quasi-isomorphism of complexes $\mathbb{E}^{\bullet}\xrightarrow{\cong}\mathbb{E}^{\bullet\vee}[1]$. 
\end{lemma}
\begin{proof} The derived dual of $\mathbb{E}^{\bullet}$ over $\mathfrak{N}_{\text{FT}}$ is given by $\mathbb{E}^{\bullet\vee}:=R\mathscr{H}om(\mathbb{E}^{\bullet},\mathcal{O}_{\mathfrak{N}_{\text{FT}}}).$ Use Grothendieck duality and obtain the following isomorphisms:
\begin{align}\label{duality}
&
R\mathscr{H}om(\mathbb{E}^{\bullet},\mathcal{O}_{\mathfrak{N}_{\text{FT}}})\cong R\pi_{\mathfrak{N}\,\ast}(R\mathscr{H}om_{X\times \mathfrak{N}_{\text{FT}}}\left(R\mathscr{H}om(\mathbb{I}^{\bullet},\mathbb{I}^{\bullet})_{0}\otimes \pi_{X}^{*}\omega_{X}\right)[2],\pi^{!}\mathcal{O}_{\mathfrak{N}_{\text{FT}}}))\notag\\
&
\cong  R\pi_{\mathfrak{N}\,\ast}(R\mathscr{H}om_{X\times \mathfrak{N}_{\text{FT}}}\left(R\mathscr{H}om(\mathbb{I}^{\bullet},\mathbb{I}^{\bullet})_{0}\otimes \pi_{X}^{*}\omega_{X}\right)[2],\pi_{X}^{\ast}\omega_{X}[3])\notag\\
&
\cong R\pi_{\mathfrak{N}\,\ast}R\mathscr{H}om(\mathcal{O}_{X \times \mathfrak{N}_{\text{FT}}}, R\mathscr{H}om(\mathbb{I}^{\bullet},\mathbb{I}^{\bullet})_{0}[1]) \cong \mathbb{E}^{\bullet}[-1].\notag\\
\end{align} 
Hence we conclude that $\mathbb{E}^{\bullet\vee}[1]\cong \mathbb{E}^{\bullet}$; The second isomorphism in \eqref{duality} is obtained using the fact that $X$ is a Calabi-Yau threefold and $ \omega_{X}\cong \mathcal{O}_{X}$.
\end{proof}
Let $\pi: \mathfrak{N}_{\text{HFT}}\to \mathfrak{N}_{\text{FT}}$ denote the natural projection map (c.f. Diagram \eqref{diagram}). Note that as was mentioned in Definition \ref{perfect-amp}, the following isomorphism holds true in $\mathcal{D}^{b}(\mathfrak{N}_{\text{HFT}})$$$\pi^{*}\mathbb{L}^{\bullet}_{\mathfrak{N}_{\text{FT}}}= \mathcal{I}/\mathcal{I}^{2}\rightarrow \Omega_{\mathfrak{A}}\mid_{ \mathfrak{N}_{\text{HFT}}}\rightarrow  \Omega_{\pi}.$$ Now we study the properties of  the pullback of $\mathbb{E}^{\bullet\vee}$ to $\mathfrak{N}_{\text{HFT}}$.

\begin{prop}\label{existence}
Let $\mathcal{U}=\coprod_{\alpha}\mathcal{U}_{\alpha}$ be an atlas of affine schemes for $\mathfrak{N}_{\text{HFT}}$. Fix one of the maps $q:\mathcal{U}_{\alpha}\rightarrow \mathfrak{N}_{\text{HFT}}$. The following isomorphism holds true in $\mathcal{D}^{b}(\mathcal{U}_{\alpha})$:
\begin{equation}\label{existence2}
\text{Hom}(q^{\ast}\Omega_{\pi}, q^{\ast}(\pi^{*}\mathbb{E}^{\bullet}_{\mathfrak{N}_{\text{FT}}}[1]))\cong \text{Hom}(q^{\ast}\Omega_{\pi}, q^{\ast}(\pi^{*}\mathbb{L}^{\bullet}_{\mathfrak{N}_{\text{FT}}}[1])).
\end{equation}
Hence, in particular it is true that there exists a nontrivial lift $g_{\alpha}: q^{\ast}\Omega_{\pi}\rightarrow q^{\ast}(\pi^{*}\mathbb{E}^{\bullet}_{\mathfrak{N}_{\text{FT}}}[1])$.
\end{prop}
\begin{proof} Consider the exact triangle
\begin{equation} \label{euler-t-2}
q^{\ast}(\pi^{*}\mathbb{E}^{\bullet\vee})\xrightarrow{ob^{t}_{\alpha}} q^{\ast}(\pi^{*}\mathbb{L}^{\bullet}_{\mathfrak{N}_{\text{FT}}})\rightarrow \text{Cone}(ob^{t}_{\alpha}),
\end{equation}
induced by pulling back the deformation-obstruction theory in Theorem \ref{reldef-f} via $\pi\circ q:\mathcal{U}_{\alpha}\rightarrow \mathfrak{N}_{\text{FT}}$.
Note that $\Omega_{\pi}\cong (\mathfrak{g}l_{r}(\mathbb{C}))^{\vee}\otimes \mathcal{O}_{\mathfrak{N}_{\text{HFT}}}.$ Hence, $q^{*}\Omega_{\pi}\cong (\mathfrak{g}l_{r}(\mathbb{C}))^{\vee}\otimes \mathcal{O}_{\mathcal{U}_{\alpha}}$. Now apply $\text{Hom}^{0}(q^{*}\Omega_{\pi},)$ to the exact triangle \eqref{euler-t-2} and obtain
\begin{align}
&
\text{Hom}^{0}(q^{*}\Omega_{\pi}, \text{Cone}(ob^{t}_{\alpha}))\rightarrow \text{Hom}^{0}(q^{*}\Omega_{\pi}, q^{\ast}(\pi^{*}\mathbb{E}^{\bullet\vee})[1])\notag\\
&
\rightarrow \text{Hom}^{0}(q^{*}\Omega_{\pi}, q^{\ast}(\pi^{*}\mathbb{L}^{\bullet})[1])\rightarrow  \text{Hom}^{0}(q^{*}\Omega_{\pi}, \text{Cone}(ob^{t}_{\alpha})[1]).
\end{align}
We prove the statement of the proposition by showing that
\begin{equation}\label{vanish-O}
\text{Hom}^{0}(q^{*}\Omega_{\pi}, \text{Cone}(ob^{t}_{\alpha}))\cong 0\cong \text{Hom}^{0}(q^{*}\Omega_{\pi}, \text{Cone}(ob^{t}_{\alpha})[1]).
\end{equation}
To prove \eqref{vanish-O}, we will use the fact that by construction $q^{*}\Omega_{\pi}\cong (\mathfrak{g}l_{r}(\mathbb{C}))^{\vee}\otimes \mathcal{O}_{\mathcal{U}_{\alpha}}$. Therefore, our proof reduces to showing that 
\begin{equation}\label{latter}
\text{Hom}^{0}(\mathcal{O}_{\mathcal{U}_{\alpha}},\text{Cone}(ob^{t}_{\alpha})[l])\cong 0,
\end{equation}
for $l=1,0$. In order to prove \eqref{latter}, first we will show that the complex $\text{Cone}(ob^{t}_{\alpha})$ has no cohomologies in degrees greater than $-1$.  For this, consider the long exact sequence of cohomology induced by the exact triangle \eqref{euler-t-2}:
\begin{align} 
&
0\rightarrow\cancel{\text{h}^{-3}(q^{\ast}(\pi^{*}\mathbb{E}^{\bullet\vee}))}\rightarrow \cancel{\text{h}^{-3}(q^{\ast}(\pi^{*}\mathbb{L}^{\bullet}))}\rightarrow \text{h}^{-3}(\text{Cone}(ob^{t}_{\alpha}))\xrightarrow{\cong}\text{h}^{-2}(q^{\ast}(\pi^{*}\mathbb{E}^{\bullet\vee}))\notag\\
&\rightarrow\cancel{\text{h}^{-2}(q^{\ast}(\pi^{*}\mathbb{L}^{\bullet}))}\rightarrow\text{h}^{-2}(\text{Cone}(ob^{t}_{\alpha}))\rightarrow \text{h}^{-1}(q^{\ast}(\pi^{*}\mathbb{E}^{\bullet\vee}))
\twoheadrightarrow \text{h}^{-1}(q^{\ast}(\pi^{*}\mathbb{L}^{\bullet}))\notag\\
&
\rightarrow \cancel{\text{h}^{-1}(\text{Cone}(ob^{t}_{\alpha}))}\rightarrow \text{h}^{0}(q^{\ast}(\pi^{*}\mathbb{E}^{\bullet\vee}))\xrightarrow{\cong} \text{h}^{0}(q^{\ast}(\pi^{*}\mathbb{L}^{\bullet}))\rightarrow\cancel{\text{h}^{0}(\text{Cone}(ob^{t}_{\alpha}))} \notag\\
&
\rightarrow \text{h}^{1}(q^{\ast}(\pi^{*}\mathbb{E}^{\bullet\vee}))\xrightarrow{\cong} \text{h}^{1}(q^{\ast}(\pi^{*}\mathbb{L}^{\bullet}))\rightarrow\cancel{\text{h}^{1}(\text{Cone}(ob^{t}_{\alpha}))}\notag\\
&
\rightarrow \cancel{\text{h}^{2}(q^{\ast}(\pi^{*}\mathbb{E}^{\bullet\vee}))}\rightarrow\cancel{\text{h}^{2}(q^{\ast}(\pi^{*}\mathbb{L}^{\bullet}))}\rightarrow0,\notag\\
\end{align}
where we have used the fact that $q^{\ast}(\pi^{*}\mathbb{L}^{\bullet})$ and $q^{\ast}(\pi^{*}\mathbb{E}^{\bullet\vee})$ are perfect complexes of amplitudes $[-1,1]$ and $[-2,1]$ respectively, and $\text{h}^{i}(ob^{t}_{\alpha})$ is an isomorphism for $i=0,1$ and a surjection for $i=-1$. Hence, we conclude that $\text{Cone}(ob^{t}_{\alpha}))$ only has cohomologies on degrees $-2$ and $-3$ and so,  we can replace the complex $\text{Cone}(ob^{t}_{\alpha})$ with a representative complex $\textbf{A}^{\bullet}$ such that $\textbf{A}^{k}=0$ for $k\geq -1$. Therefore, the proof of \eqref{latter} follows by showing that $\text{Hom}^{0}(\mathcal{O}_{\mathcal{U}_{\alpha}}, \textbf{A}^{\bullet}[l])\cong 0$ for all $l\geq 0$. 

Now use the general fact that given complexes $\mathcal{T}$ and $\textbf{F}$, in order to compute the Grothendieck hypercohomolgy $\text{Hom}^{i}(\mathcal{T},\textbf{F})$, one replaces $\textbf{F}$ with its injective resolution $\textbf{F}\rightarrow \textbf{K}^{\bullet}$. Moreover replacing $\textbf{F}$ with $\textbf{K}^{\bullet}$ is equivalent with replacing $\mathcal{T}$ with $\textbf{P}^{\bullet}$ such that $\textbf{P}^{\bullet}\rightarrow \mathcal{T}$ is a projective resolution. On the other hand, locally over $\mathcal{U}_{\alpha}$, $\mathcal{O}_{\mathcal{U}_{\alpha}}$ is given as a free and in particular projective module hence its projective resolution consists of one term only and we can use $\mathcal{O}_{\mathcal{U}_{\alpha}}$ itself, instead of the complex $\textbf{P}^{\bullet}$. Now it is seen that, $\text{Hom}^{0}(\mathcal{O}_{\mathcal{U}_{\alpha}},\textbf{A}^{\bullet}[l])\cong 0$ since $\mathcal{O}_{\mathcal{U}_{\alpha}}$ is a flasque sheaf sitting in degree zero. This proves the isomorphism \eqref{latter}, which implies the isomorphism \eqref{vanish-O}, which proves the statement of Proposition \ref{existence}.
\end{proof}

\begin{theorem}\label{2step-trunc}
Consider the 4-term deformation-obstruction theory $\mathbb{E}^{\bullet\vee}$ of perfect amplitude $[-2,1]$ over $\mathfrak{N}_{\text{FT}}$. Locally in the Zariski topology over $\mathfrak{N}_{\text{HFT}}$ there exists a perfect two-term deformation-obstruction theory of perfect amplitude $[-1,0]$ which is obtained from the suitable local truncation of the pullback $\pi^{*}\mathbb{E}^{\bullet\vee}$ via the map $\pi:\mathfrak{N}_{\text{HFT}}\rightarrow \mathfrak{N}_{\text{FT}}$.
\end{theorem}
\begin{proof} We prove this theorem by cohomologically truncating the 4-term complex $\pi^{*}\mathbb{E}^{\bullet\vee}$ on right side and then left side; By Theorem \ref{reldef-f}, $\mathbb{E}^{\bullet\vee}\xrightarrow{ob} \mathbb{L}^{\bullet}_{\mathfrak{N}_{\text{FT}}}$ is a perfect deformation-obstruction theory of amplitude $[-2,1]$ for $\mathfrak{N}_{\text{FT}}$, such that $h^{0}(ob),h^{1}(ob)$ are isomorphisms and $h^{-1}(ob)$ is an epimorphism.  Let $q:\mathcal{U}_{\alpha}\rightarrow \mathfrak{N}_{\text{HFT}}$ and $q':\mathcal{U}_{\beta}\rightarrow\mathfrak{N}_{\text{HFT}}$ be given as fixed affine charts over $\mathfrak{N}_{\text{HFT}}$ such that the isomorphism in Proposition \ref{existence} holds true over $\mathcal{U}_{\alpha}$ and $\mathcal{U}_{\beta}$. Let $f_{\alpha}: \mathcal{U}_{\alpha}\times_{q\times q'} \mathcal{U}_{\beta}\rightarrow \mathcal{U}_{\alpha}$ and $f_{\beta}: \mathcal{U}_{\alpha}\times_{q\times q'} \mathcal{U}_{\beta}\rightarrow \mathcal{U}_{\beta}$ be the corresponding projection to $\mathcal{U}_{\alpha}$ and $\mathcal{U}_{\beta}$. Then $$ \text{Hom}^{0}(f_{\alpha}^{*}(q^{*}\Omega_{\pi}, q^{\ast}(\pi^{*}\mathbb{E}^{\bullet\vee})[1]))\cong \text{Hom}^{0}(f_{\beta}^{*}(q^{*}\Omega_{\pi}, q^{\ast}(\pi^{*}\mathbb{L}^{\bullet\vee})[1])).$$Moreover, the same statement is true for $f_{\alpha}$ and $q$ replaced by $f_{\beta}$ and $q'$.

Because $\mathfrak{N}_{\text{HFT}}$ is a quasi-projective scheme, then an intersection of affine subschemes of $\mathfrak{N}_{\text{HFT}}$ is affine, which implies that Proposition \ref{existence} holds true over $\mathcal{U}_{\alpha}\times_{q\times q'}\mathcal{U}_{\beta}$ also. Now fix $\mathcal{U}_{\alpha}$. By the local existence of the map $g_{\alpha}$ in Proposition \ref{existence}, there exists a commutative diagram over $\mathcal{U}_{\alpha}$:
\begin{equation}\label{can0-cot}
\begin{tikzpicture}
back line/.style={densely dotted}, 
cross line/.style={preaction={draw=white, -, 
line width=6pt}}] 
\matrix (m) [matrix of math nodes, 
row sep=1em, column sep=1.5em, 
text height=1.5ex, 
text depth=0.25ex]{ 
\pi^{*}\mathbb{E}^{\bullet\vee}\mid_{\mathcal{U}_{\alpha}}&\text{Cone}(g_{\alpha})[-1]&\Omega_{\pi}\mid_{\mathcal{U}_{\alpha}}&\pi^{*}\mathbb{E}^{\bullet\vee}[1]\mid_{\mathcal{U}_{\alpha}}&\text{Cone}(g_{\alpha})\\
\pi^{*}\mathbb{L}^{\bullet}_{\mathfrak{N}_{\text{FT}}}\mid_{\mathcal{U}_{\alpha}}&\mathbb{L}^{\bullet}_{\mathfrak{N}_{\text{HFT}}}\mid_{\mathcal{U}_{\alpha}}&\Omega_{\pi}\mid_{\mathcal{U}_{\alpha}}&\pi^{*}\mathbb{L}^{\bullet}_{\mathfrak{N}_{\text{FT}}}[1]\mid_{\mathcal{U}_{\alpha}}&\mathbb{L}^{\bullet}_{\mathfrak{N}_{\text{HFT}}}[1]\mid_{\mathcal{U}_{\alpha}}\\};
\path[->]
(m-1-1) edge node [right] {$\pi^{*}(ob)\mid_{\mathcal{U}_{\alpha}}$} (m-2-1)
(m-1-1) edge (m-1-2)
(m-1-2) edge (m-1-3)
(m-1-2) edge node [right] {$ob'$} (m-2-2)
(m-1-4) edge (m-1-5)
(m-1-5) edge (m-2-5)
(m-2-1) edge (m-2-2)
(m-1-3) edge node [above] {$g_{\alpha}$} (m-1-4)
(m-2-2) edge (m-2-3)
(m-2-3) edge (m-2-4)
(m-2-4) edge (m-2-5)
(m-1-3) edge node [left] {$\text{id}$} (m-2-3)
(m-1-4) edge node [right] {$\pi^{*}ob[1]\mid_{\mathcal{U}_{\alpha}}$}(m-2-4);
\end{tikzpicture}
\end{equation}  
Now we show that the map $ob':\text{Cone}(g_{\alpha})[-1]\rightarrow \mathbb{L}^{\bullet}_{\mathfrak{N}_{\text{HFT}}}\mid_{\mathcal{U}_{\alpha}}$ defines a perfect 3-term deformation-obstruction theory of amplitude $[-2,0]$ for $\mathfrak{N}_{\text{HFT}}$ over $\mathcal{U}_{\alpha}$.

We show that $\text{Cone}(g)[-1]$ is concentrated in degrees $-2$, $-1$ and 0. Moreover, $h^{0}(ob')$ is an isomorphism and $h^{-1}(ob')$ is an epimorphism. The proof uses the long exact sequence of cohomologies. For $h^{-1}(ob')$ one obtains:

\begin{equation}\label{h(-1)-pf}
\begin{tikzpicture}
back line/.style={densely dotted}, 
cross line/.style={preaction={draw=white, -, 
line width=6pt}}] 
\matrix (m) [matrix of math nodes, 
row sep=1em, column sep=1.5em, 
text height=1.5ex, 
text depth=0.25ex]{ 
0&h^{-1}(\pi^{*}\mathbb{E}^{\bullet\vee}\mid_{\mathcal{U}_{\alpha}})\cong \pi^{*}h^{-1}(\mathbb{E}^{\bullet\vee}\mid_{\mathcal{U}_{\alpha}})&h^{-1} (\text{Cone}(g_{\alpha})[-1])&0\\
0&h^{-1}(\pi^{*}\mathbb{L}^{\bullet}_{\mathfrak{N}_{\text{FT}}}\mid_{\mathcal{U}_{\alpha}})\cong \pi^{*}h^{-1}(\mathbb{L}^{\bullet}_{\mathfrak{N}_{\text{FT}}}\mid_{\mathcal{U}_{\alpha}})&h^{-1}(\mathbb{L}^{\bullet}_{\mathfrak{N}_{\text{HFT}}}\mid_{\mathcal{U}_{\alpha}})&0\\};
\path[->]
(m-1-1) edge (m-2-1)
(m-1-1) edge (m-1-2)
(m-1-2) edge node [above] {$\cong$} (m-1-3)
(m-2-1) edge (m-2-2)
(m-1-3) edge (m-1-4)
(m-2-2) edge node [above] {$\cong$} (m-2-3)
(m-2-3) edge (m-2-4)
(m-1-3) edge node [right] {$h^{-1}(ob')$} (m-2-3)
(m-1-4) edge (m-2-4);
\path[->>]
(m-1-2) edge node [right] {$\pi^{*}(h^{-1}(ob))$} (m-2-2);
\end{tikzpicture},
\end{equation}
where the top horizontal isomorphism is due to the fact that $$\text{Cone}(g_{\alpha})[-1]: \pi^{*}E^{-2}\rightarrow \pi^{*}E^{-1}\rightarrow\pi^{*} E^{0}\oplus \Omega_{\pi}\rightarrow \pi^{*}E^{1},$$where $E^{i}$ correspond to the terms of the complex $\mathbb{E}^{\bullet\vee}\mid_{\mathcal{U}_{\alpha}}$. The vanishings on the left and right of the top and bottom rows of \eqref{h(-1)-pf} are due to the fact that $\Omega_{\pi}$ is a sheaf concentrated in degree zero. By Theorem \ref{reldef-f}, the second vertical map (from left) is a surjection and by commutativity of the diagram \eqref{h(-1)-pf} the map $h^{-1}(ob')$ is surjective. In degrees 0 and 1 one obtains:  
\begin{equation}\label{can-cot}
\begin{tikzpicture}
back line/.style={densely dotted}, 
cross line/.style={preaction={draw=white, -, 
line width=6pt}}] 
\matrix (m) [matrix of math nodes, 
row sep=2.em, column sep=.5em, 
text height=1.5ex, 
text depth=0.25ex]{ 
0&\pi^{*}\text{h}^{0}(\mathbb{E}^{\bullet\vee}\mid_{\mathcal{U}_{\alpha}})&\text{h}^{0}(\text{Cone}(g_{\alpha})[-1])&\Omega_{\pi}\mid_{\mathcal{U}_{\alpha}}&\pi^{*}\text{h}^{1}(\mathbb{E}^{\bullet\vee}\mid_{\mathcal{U}_{\alpha}})&\text{h}^{1}(\text{Cone}(g)[-1])&0\\
0&\pi^{*}\text{h}^{0}(\mathbb{L}^{\bullet}_{\mathfrak{N}_{s,\text{FT}}}\mid_{\mathcal{U}_{\alpha}})&\text{h}^{0}(\mathbb{L}^{\bullet}_{\mathfrak{N}_{\text{HFT}}}\mid_{\mathcal{U}_{\alpha}})&\Omega_{\pi}\mid_{\mathcal{U}_{\alpha}}&\pi^{*}\text{h}^{1}(\mathbb{L}^{\bullet}_{\mathfrak{N}_{\text{FT}}}\mid_{\mathcal{U}_{\alpha}})&\text{h}^{1}(\mathbb{L}^{\bullet}_{\mathfrak{N}_{\text{HFT}}}\mid_{\mathcal{U}_{\alpha}})&0\\};
\path[->]
(m-1-1) edge (m-2-1)
(m-1-1) edge (m-1-2)
(m-1-2) edge (m-1-3)
(m-1-2) edge node [right] {$\pi^{*}\text{h}^{0}(ob)\mid_{\mathcal{U}_{\alpha}}$} (m-2-2)
(m-1-4) edge (m-1-5)
(m-2-1) edge (m-2-2)
(m-1-3) edge (m-1-4)
(m-2-2) edge (m-2-3)
(m-2-3) edge (m-2-4)
(m-2-4) edge (m-2-5)
(m-1-3) edge node [right] {$\text{h}^{0}(ob')$} (m-2-3)
(m-1-4) edge node [right] {$\text{id}$} (m-2-4)
(m-1-5) edge node [right] {$\pi^{*}\text{h}^{1}(ob)\mid_{\mathcal{U}_{\alpha}}$} (m-2-5)
(m-1-6) edge node [right] {$\text{h}^{1}(ob')$} (m-2-6)
(m-1-7) edge (m-2-7)
(m-1-6) edge (m-1-7)
(m-2-6) edge (m-2-7)
(m-1-5) edge (m-1-6)
(m-2-5) edge (m-2-6);
\end{tikzpicture}.
\end{equation} 
In this diagram, $h^{1}(\mathbb{L}^{\bullet}_{\mathfrak{N}_{\text{HFT}}}\mid_{\mathcal{U}_{\alpha}})\cong 0$ since over $ \mathfrak{N}_{\text{HFT}}$ the truncated cotangent complex does not have cohomology in degree 1. Moreover, $\pi^{*}h^{1}(ob)\mid_{\mathcal{U}_{\alpha}}$ is an isomorphism by Theorem \ref{reldef-f}. Hence $h^{1}(ob')\cong 0$. Moreover by Theorem \ref{reldef-f}, $\pi^{*}h^{0}(ob)$ is an isomorphism, hence by the commutativity of the diagram \eqref{can-cot}, $h^{0}(ob')$ is an isomorphism.  

Now in order to obtain a perfect deformation-obstruction theory of amplitude $[-1,0]$, one needs to truncate the complex $\text{Cone}(g_{\alpha})[-1]$ on the left side, so that it does not have any cohomology in degree $-2$. The self-duality of $\mathbb{E}^{\bullet}$ gives a diagram of morphisms in the derived category:
 \begin{equation}\label{self-dual}
\begin{tikzpicture}
back line/.style={densely dotted}, 
cross line/.style={preaction={draw=white, -, 
line width=6pt}}] 
\matrix (m) [matrix of math nodes, 
row sep=1em, column sep=2.5em, 
text height=1.5ex, 
text depth=0.25ex]{ 
\mathbb{E}^{\bullet}&\mathbb{E}^{\bullet\vee}[1]&\text{Cone}(g_{\alpha})&\\
\text{T}_{\pi}\mid_{\mathcal{U}_{\alpha}}[1]&&&\\};
\path[->]
(m-1-1) edge node [left] {$g_{\alpha}^{\vee}$}(m-2-1)
(m-1-1) edge node [below] {$\text{q-isom}$}(m-1-2)
(m-1-1) edge node [above] {$\cong$} (m-1-2)
(m-1-2) edge (m-1-3);
\end{tikzpicture}
\end{equation}
and we need to show that the natural map
\begin{equation}\label{lift-it-again}
\text{Hom}^{0}_{\mathcal{D}(\mathcal{U}_{\alpha})}(\text{Cone}(g_{\alpha}),\text{T}_{\pi}\mid_{\mathcal{U}_{\alpha}}[1])\rightarrow\text{Hom}^{0}_{\mathcal{D}(\mathcal{U}_{\alpha})}(\mathbb{E}^{\bullet},\text{T}_{\pi}\mid_{\mathcal{U}_{\alpha}}[1])
\end{equation}
is an isomorphism.

Note that $\mathcal{U}_{\alpha}$ is affine and $\text{T}_{\pi}\mid_{\mathcal{U}_{\alpha}}[1]\cong \mathcal{O}_{\mathcal{U}_{\alpha}}^{\text{dim}(\mathfrak{g})}[1]$, so the statement reduces to knowing that $\text{h}^{1}(\text{Cone}(g_{\alpha})^{\vee})\rightarrow \text{h}^{1}(\mathbb{E}^{\bullet\vee})$ is an isomorphism. This follows, since $\mathbb{E}^{\bullet\vee}[1]\rightarrow \text{Cone}(g_{\alpha})$ is an isomorphism on $\text{h}^{-1}$ as shown in diagram \eqref{h(-1)-pf}.

By the isomorphism \eqref{lift-it-again}, it is now seen that  the map $g_{\alpha}^{\vee}$ in diagram \eqref{self-dual} factors through the map $$\text{Cone}(g_{\alpha})\rightarrow \text{T}_{\pi}\mid_{\mathcal{U}_{\alpha}}[1]$$ which is unique up to homotopy. We make a choice of such lift, shift the degrees by $-1$ and denote the resulting map by ${g'_{\alpha}}^{\vee}$. So we get an exact triangle
\begin{equation}\label{lift-g'}
\text{Cone}({g'_{\alpha}}^{\vee})[-1]\rightarrow \text{Cone}(g_{\alpha})[-1]\xrightarrow{{g'_{\alpha}}^{\vee}}T_{\pi}\mid_{\mathcal{U}_{\alpha}}\rightarrow \text{Cone}({g'_{\alpha}}^{\vee}).
\end{equation}
By this construction, we are now ready to define our ``\textit{two-fold}" truncated complex as
\begin{equation}\label{two-folded}
\mathbb{G}^{\bullet}\mid_{\mathcal{U}_{\alpha}}:=\text{Cone}({g'_{\alpha}}^{\vee})[-1]
\end{equation}
 In order to finish the proof of Theorem \ref{2step-trunc}, we need show that the complex $\mathbb{G}^{\bullet}\mid_{\mathcal{U}_{\alpha}}$ in \eqref{two-folded} defines a perfect deformation-obstruction theory of amplitude $[-1,0]$ for $\mathcal{U}_{\alpha}$.
By Theorem \ref{reldef-f}, the complex $\pi^*\mathbb{E}^{\bullet\vee}$ is given by a $4$-term complex of vector bundles of amplitude $[-2,1]$ 
as follows$$\pi^*\mathbb{E}^{\bullet\vee}\mid_{\mathcal{U}_{\alpha}}:=\pi^{*}E^{-2}\rightarrow \pi^{*}E^{-1}\rightarrow \pi^{*}E^{0} \rightarrow \pi^{*}E^{1}.$$Then, it is seen that by construction$$\mathbb{G}^{\bullet}\mid_{\mathcal{U}_{\alpha}}:=\pi^{*}E^{-2}\xrightarrow{d'} \pi^{*}E^{-1}\oplus T_{\pi}\mid_{\mathcal{U}_{\alpha}}\rightarrow \pi^{*}E^{0}\oplus \Omega_{\pi}\mid_{\mathcal{U}_{\alpha}}\xrightarrow{d} \pi^{*}E^{1},$$which means $\mathbb{G}^{\bullet}\mid_{\mathcal{U}_{\alpha}}$ has no cohomologies in degree $-2$ and $1$, however its cohomologies in degrees $-1$ and $0$ are the same as that of $\text{Cone}(g_{\alpha})[-1]$\footnote{Note that the $-1$ and $0$ cohomologies of $\text{Cone}(g_{\alpha})[-1]$ themselves, are equal to $-1$ and $0$ cohomologies of $\pi^*\mathbb{E}^{\bullet\vee}\mid_{\mathcal{U}_{\alpha}}$ by diagrams \eqref{h(-1)-pf} and \eqref{can-cot}.}. In other words, the $\mathbb{G}^{\bullet}\mid_{\mathcal{U}_{\alpha}}$ is the truncation of $\text{Cone}(g_{\alpha})[-1]$ on degree $-2$ and it can be seen that in the following commutative diagram, the top row is quasi-isomorphic to the bottom row\footnote{Note that by self-symmetry of $\mathbb{E}^{\bullet\vee}$ the morphism $d'$ in diagram \eqref{q-isom} is the dual of the morphism $d$. Therefore, since $\text{Ker}(d)$ is locally free then it is implied that $\text{Coker}(d')$ is also locally free.}:
\begin{equation}\label{q-isom}
\begin{tikzpicture}
back line/.style={densely dotted}, 
cross line/.style={preaction={draw=white, -, 
line width=6pt}}] 
\matrix (m) [matrix of math nodes, 
row sep=1em, column sep=3.25em, 
text height=1.5ex, 
text depth=0.25ex]{ 
\pi^{*}E^{-2}&\pi^{*}E^{-1}\oplus T_{\pi}\mid_{\mathcal{U}_{\alpha}}&\pi^{*}E^{0}\oplus \Omega_{\pi}\mid_{\mathcal{U}_{\alpha}}&\pi^{*}E^{1}\\
0&\text{Coker}(d')&\text{Ker}(d)&0\\};
\path[->]
(m-1-1) edge (m-2-1)
(m-1-2) edge (m-1-3)
(m-1-1) edge node [above] {$d'$} (m-1-2)
(m-1-3) edge node [above] {$d$} (m-1-4)
(m-1-2) edge (m-2-2)
(m-1-4) edge (m-2-4)
(m-2-3) edge (m-2-4)
(m-2-1) edge (m-2-2)
(m-2-2) edge (m-2-3)
(m-1-3) edge (m-2-3);
\end{tikzpicture},
\end{equation}
Now the composition of morphisms $\phi:=\mathbb{G}^{\bullet}\mid_{\mathcal{U}_{\alpha}}\rightarrow \text{Cone}(g)[-1] \rightarrow  \mathbb{L}^{\bullet}_{\mathfrak{N}_{\text{HFT}}}\mid_{\mathcal{U}_{\alpha}}$, together with the fact that by diagrams \eqref{h(-1)-pf} and \eqref{can-cot}, $h^{0}(ob')$ is an isomorphism and $h^{-1}(ob')$ is an epimorphism, shows that the morphism $\mathbb{G}^{\bullet}\mid_{\mathcal{U}_{\alpha}}\xrightarrow{\phi} \mathbb{L}^{\bullet}_{\mathfrak{N}_{\text{HFT}}}\mid_{\mathcal{U}_{\alpha}}$ satisfies the condition of being a deformation-obstruction theory, i.e we have that $h^0(\phi)$ is an isomorphism and $h^{-1}(\phi)$ is an epimorphism.
\end{proof}
\subsection{Construction of virtual fundamental class}\label{sec10}
Followed by the constructions in \cite[Definition 3.7]{a2}, in order to construct the virtual fundamental class, one needs to choose a local embedding over the moduli stack. Then one constructs the normal cone associated to a perfect local deformation-obstruction theory of amplitude $[-1,0]$ over this local embedding and proves that this normal cone is independent of the local embedding, i.e it remains invariant under the base change. The base-change invariance enables one to glue the local normal cones constructed over each local embedding, and obtain a global cone stack. Essentially a global virtual fundamental class is constructible from a global normal cone stack. 

For us, the gluability of the so-called local normal cone stacks depends on whether the local deformation-obstruction theories over each chart satisfy the condition of being given as a ``\textit{semi-perfect obstruction theory}" in the sense of Chang-Li \cite{a70}. However, we realize that the latter is achievable only if we assume a technical condition (c.f. Assumption \ref{descent assumption}). 

We will see later in Part II that this assumption is satisfied automatically, over the torus fixed loci of the moduli space, when we study highly frozen triples over a toric variety (c.f. Section \ref{sec11}).

\begin{defn}\label{inf-lift-def}
\cite[Definition 2.5]{a70} Let $\iota:T\rightarrow T'$ be a closed subscheme with $T'$ local Artinian. Let $\mathcal{I}$ be the ideal sheaf of $T$ in $T'$ and let $\mathfrak{m}$ be the ideal sheaf of the closed point of $T'$. We call $\iota$ a small extension if $\mathcal{I}\cdot \mathfrak{m}=0$. Now Let $\mathcal{M}$ be an Artin stack and $X\rightarrow \mathcal{M}$ a representable morphism of a DM stack to an Artin stack. Let $\mathcal{U}=\coprod_{\alpha}\mathcal{U}_{\alpha\in \Lambda}$ be a DM cover of $X$ by affine schemes. Consider $\mathcal{U}_{\alpha}\rightarrow \mathcal{M}$ for some $\alpha\in \Lambda$. Consider a small extension $(T,T',\mathcal{I},\mathfrak{m})$ that fits into a commutative square
\begin{equation}\label{inf-lift}
\begin{tikzpicture}
back line/.style={densely dotted}, 
cross line/.style={preaction={draw=white, -, 
line width=6pt}}] 
\matrix (m) [matrix of math nodes, 
row sep=1em
, column sep=3.em, 
text height=1.5ex, 
text depth=0.25ex]{ 
T&\mathcal{U}_{\alpha}\\
T'&\mathcal{M}\\};
\path[->]
(m-1-1) edge node [above] {$g$} (m-1-2)
(m-1-2)edge (m-2-2)
(m-2-1) edge (m-2-2)
(m-1-1) edge node [left] {$\iota$} (m-2-1);
\end{tikzpicture},
\end{equation}
so that the image of $g$ contains a closed point $p\in \mathcal{U}_{\alpha}$. Finding a morphism $g':T'\rightarrow \mathcal{U}_{\alpha}$ that commutes with the arrows in \eqref{inf-lift} is called ``\textit{infinitesimal lifting problem of} $\mathcal{U}_{\alpha}\slash\mathcal{M}$ \textit{at} $p$". 
\end{defn}
\begin{defn}\label{obstruction-lift}
Given a $\mathcal{U}\rightarrow \mathcal{M}$, let $\phi:\mathbb{G}^{\bullet}_{\mathcal{U}}\rightarrow\mathbb{L}^{\bullet}_{\mathcal{U}\slash \mathcal{M}}$ be a perfect obstruction theory. For the infinitesimal lifting problem in Definition \ref{inf-lift-def} we call the image
\begin{equation}
ob(\phi,g,T,T'):=\text{h}^{1}(\phi^{\vee})(\omega(g,T,T'))\in \text{Ext}^{1}(g^{*}\mathbb{G}^{\bullet}_{\mathcal{U}},\mathcal{I})=\text{Ob}(\phi,p)\otimes \mathcal{I}
\end{equation} 
the obstruction class (of $\phi$) to the lifting problem. 
\end{defn}
\begin{defn}
\cite[Definition 2.9]{a70} Given two (local) deformation-obstruction theories $\phi:\mathbb{G}^{\bullet}_{\mathcal{U}}\rightarrow\mathbb{L}^{\bullet}_{\mathcal{U}\slash \mathcal{M}}$ and $\phi':\mathbb{G}'^{\bullet}_{\mathcal{U}}\rightarrow\mathbb{L}^{\bullet}_{\mathcal{U}\slash \mathcal{M}}$ over $\mathcal{U}$ as in Definition \ref{obstruction-lift}, we call them $\nu$-equivalent if there exists an isomorphism of sheaves:
\begin{equation}
\psi:\text{h}^{1}(\mathbb{G}^{\bullet\vee}_{\mathcal{U}})\rightarrow \text{h}^{1}(\mathbb{G}'^{\bullet\vee}_{\mathcal{U}})
\end{equation} 
so that for every closed point $p\in \mathcal{U}_{\alpha}$ and any infinitesimal lifting problem of $\mathcal{U}_{\alpha}/\mathcal{M}$ at $p$ (as in Definition \ref{inf-lift-def}) we have $$\psi\mid_{p}(ob(\phi,g,T,T'))=ob(\phi',g,T,T')\in \text{Ob}(\phi',p)\otimes_{\textbf{k}}\mathcal{I}.$$
\end{defn}
Let $\mathcal{U}_{\alpha}, \mathcal{U}_{\beta}\subset \mathfrak{N}_{\text{HFT}}$ be given as two charts with the lifting property as in Proposition \ref{existence}. Let $\phi_{\alpha}: \mathbb{G}^{\bullet}_{\alpha}\rightarrow \mathbb{L}^{\bullet}_{\mathcal{U}_{\alpha}/\mathcal{M}}$ and $\phi_{\beta}: \mathbb{G}^{\bullet}_{\beta}\rightarrow \mathbb{L}^{\bullet}_{\mathcal{U}_{\beta}/\mathcal{M}}$. Moreover, let $\mathcal{U}_{\alpha\beta}=\mathcal{U}_{\alpha}\cap\mathcal{U}_{\beta}$, $f_{\alpha}:\mathcal{U}_{\alpha\beta}\hookrightarrow\mathcal{U}_{\alpha}$ and $f_{\beta}:\mathcal{U}_{\alpha\beta}\hookrightarrow\mathcal{U}_{\beta}$. \begin{prop}\label{semi-perf1}
Let $f_{\alpha}^{*}\phi_{\alpha}$ and $f_{\beta}^{*}\phi_{\beta}$ denote the pullback of $\phi_{\alpha}$ and $\phi_{\beta}$ to $\mathcal{U}_{\alpha\beta}$. Then $f_{\alpha}^{*}\phi_{\alpha}$ and $f_{\beta}^{*}\phi_{\beta}$ are $\nu$-equivalent over $\mathcal{U}_{\alpha\beta}$.
\end{prop}
\begin{proof} We have to show that given a diagram 
\begin{equation}\label{gluable}
\begin{tikzpicture}
back line/.style={densely dotted}, 
cross line/.style={preaction={draw=white, -, 
line width=6pt}}] 
\matrix (m) [matrix of math nodes, 
row sep=1em, column sep=3.em, 
text height=1.5ex, 
text depth=0.25ex]{ 
T&\mathcal{U}_{\alpha\beta}\\
T'&\mathfrak{N}_{\text{HFT}}\\};
\path[->]
(m-1-1) edge node [above] {$g_{\alpha\beta}$} (m-1-2)
(m-1-2)edge (m-2-2)
(m-2-1) edge (m-2-2)
(m-1-1) edge node [left] {$\iota$} (m-2-1);
\end{tikzpicture},
\end{equation}
there exists a map $\psi: \text{h}^{1}(f_{\alpha}^{*}\mathbb{G}^{\bullet\vee}_{\alpha})\xrightarrow{\cong}\text{h}^{1}(f_{\beta}^{*}\mathbb{G}^{\bullet\vee}_{\beta})$ such that given a class $ob(f_{\alpha}^{*}\phi_{\alpha},g_{\alpha\beta},T,T')\in\text{h}^{1}(f^{*}_{\alpha}(\mathbb{L}^{\bullet}_{\mathfrak{N}_{\text{HFT}}}\mid_{\mathcal{U}_{\alpha}})^{\vee})$ (Look at diagram \eqref{gluable}) and for every point $p\in \mathcal{U}_{\alpha\beta}$ we have$$\psi\mid_{p}ob(f_{\alpha}^{*}\phi_{\alpha},g_{\alpha\beta},T,T')=ob(f_{\beta}^{*}\phi_{\beta},g_{\alpha\beta},T,T').$$
Apply the result of Proposition \ref{existence} to $\mathcal{U}_{\alpha}$ and $\mathcal{U}_{\beta}$ and obtain unique isomorphisms as in \eqref{existence2} over $\mathcal{U}_{\alpha}$ and $\mathcal{U}_{\beta}$. Now use the fact that $\mathcal{U}_{\alpha\beta}$ is affine and pull back the obtained isomorphisms via $f_{\alpha}$ and $f_{\beta}$ to $\mathcal{U}_{\alpha\beta}$ and obtain a unique isomorphism$$\text{Hom}(\Omega_{\pi}\mid_{\mathcal{U}_{\alpha\beta}},\mathbb{E}^{\bullet}_{\alpha\beta})\cong\text{Hom}(\Omega_{\pi}\mid_{\mathcal{U}_{\alpha\beta}},\mathbb{L}^{\bullet}_{\alpha\beta}).$$Now by the uniqueness property, there exists an isomorphism in $\mathcal{D}^{b}(\mathcal{U}_{\alpha\beta})$ given by$$\kappa_{\alpha\beta}:f_{\alpha}^{*}\mathbb{E}^{\bullet}_{\alpha}\rightarrow f_{\beta}^{*}\mathbb{E}^{\bullet}_{\beta}.$$By assumption, $\mathcal{U}_{\alpha}$ and $\mathcal{U}_{\beta}$ are given as charts with lifting property (c.f. Theorem \ref{2step-trunc}), hence there exists lifts $\text{Hom}(\Omega_{\pi}\mid_{\mathcal{U}_{\alpha}},\mathbb{E}^{\bullet}_{\alpha}[1])$ and $\text{Hom}(\Omega_{\pi}\mid_{\mathcal{U}_{\beta}},\mathbb{E}^{\bullet}_{\beta}[1])$ given by $g_{\alpha}:\Omega_{\pi}\mid_{\mathcal{U}_{\alpha}}\rightarrow\mathbb{E}^{\bullet}_{\alpha}[1]$ and $g_{\beta}:\Omega_{\pi}\mid_{\mathcal{U}_{\beta}}\rightarrow\mathbb{E}^{\bullet}_{\beta}[1]$ over $\mathcal{U}_{\alpha}$ and $\mathcal{U}_{\beta}$ respectively. Now consider the pullbacks $f_{\alpha}^{*}\Omega_{\pi}\mid_{\mathcal{U}_{\alpha}}[-1]\rightarrow f_{\alpha}^{*}\mathbb{E}^{\bullet}_{\alpha}$ and $f_{\alpha}^{*}\Omega_{\pi}\mid_{\mathcal{U}_{\alpha}}[-1]\rightarrow f_{\alpha}^{*}\mathbb{E}^{\bullet}_{\alpha}$ and note that by Proposition \ref{existence}, $f_{\alpha}^{*}g_{\alpha}$ and $f_{\beta}^{*}g_{\beta}$ are homotopic to each other over $\mathcal{U}_{\alpha\beta}$ and satisfy the equation:$$f_{\alpha}^{*}g_{\alpha}-f_{\beta}^{*}g_{\beta}=d\circ h_{\alpha\beta}+h_{\alpha\beta}\circ d,$$where $h_{\alpha\beta}$ is given as a choice of homotopy. Now take the cone of $f_{\alpha}^{*}g_{\alpha}$ and $f_{\beta}^{*}g_{\beta}$ and obtain the following commutative diagram:
\begin{equation}\label{homotopy-1}
\begin{tikzpicture}
back line/.style={densely dotted}, 
cross line/.style={preaction={draw=white, -, 
line width=6pt}}] 
\matrix (m) [matrix of math nodes, 
row sep=1em, column sep=3.25em, 
text height=1.5ex, 
text depth=0.25ex]{ 
\text{Cone}(f_{\alpha}^{*}g_{\alpha})[-1]&f_{\alpha}^{*}\Omega_{\pi}\mid_{\mathcal{U}_{\alpha}}&f_{\alpha}^{*}\mathbb{E}^{\bullet}_{\alpha}[1]&\text{Cone}(f_{\alpha}^{*}g_{\alpha})\\
\text{Cone}(f_{\beta}^{*}g_{\beta})[-1]&f_{\beta}^{*}\Omega_{\pi}\mid_{\mathcal{U}_{\beta}}&f_{\beta}^{*}\mathbb{E}^{\bullet}_{\beta}[1]&\text{Cone}(f_{\beta}^{*}g_{\beta}),\\};
\path[->]
(m-1-1) edge node [right] {$J_{\alpha\beta}[-1]$} (m-2-1)
(m-1-2) edge node [above] {$f_{\alpha}^{*}g_{\alpha}$} (m-1-3)
(m-1-1) edge (m-1-2)
(m-1-3) edge  (m-1-4)
(m-1-2) edge node [right] {$\text{id}$} (m-2-2)
(m-1-4) edge node [right] {$J_{\alpha\beta}$} (m-2-4)
(m-2-3) edge (m-2-4)
(m-2-1) edge (m-2-2)
(m-2-2) edge node [above] {$f_{\beta}^{*}g_{\beta}$} (m-2-3)
(m-1-3) edge node [right] {$\text{id}$} (m-2-3);
\end{tikzpicture}
\end{equation}
where $J_{\alpha\beta}:=\left(\begin{array}{cc}  \text{id}&h_{\alpha\beta}\\
0&\text{id}
 \end{array}\right)$. Since the first and the second rows in diagram \eqref{homotopy-1} are given by exact triangles, one computes the long exact sequence of cohomologies and obtains the following commutative diagram:
\begin{equation}\label{homotopy-2}
\begin{tikzpicture}
back line/.style={densely dotted}, 
cross line/.style={preaction={draw=white, -, 
line width=6pt}}] 
\matrix (m) [matrix of math nodes, 
row sep=1em, column sep=2.em, 
text height=1.5ex, 
text depth=0.25ex]{ 
\cdots&\text{h}^{i}(f_{\alpha}^{*}\Omega_{\pi}\mid_{\mathcal{U}_{\alpha}}[-1])&\text{h}^{i}(f_{\alpha}^{*}\mathbb{E}^{\bullet}_{\alpha})&\text{h}^{i}(\text{Cone}(f_{\alpha}^{*}g_{\alpha}))[-1]&\cdots\\
\cdots&\text{h}^{i}(f_{\beta}^{*}\Omega_{\pi}\mid_{\mathcal{U}_{\beta}}[-1])&\text{h}^{i}(f_{\beta}^{*}\mathbb{E}^{\bullet}_{\beta})&\text{h}^{i}(\text{Cone}(f_{\beta}^{*}g_{\beta}))[-1]&\cdots.\\};
\path[->]
(m-1-1) edge(m-1-2)
(m-2-1) edge (m-2-2)
(m-1-2) edge node [right] {$\text{id}$} (m-2-2)
(m-1-3) edge  (m-1-4)
(m-1-2) edge (m-1-3)
(m-1-4) edge  (m-1-5)
(m-1-3) edge node [right] {$\text{id}$} (m-2-3)
(m-2-4) edge (m-2-5)
(m-2-2) edge (m-2-3)
(m-2-3) edge  (m-2-4)
(m-1-4) edge  node [right] {$\text{h}^{i}(J_{\alpha\beta})[-1]$} (m-2-4);
\end{tikzpicture}
\end{equation}
Now use \cite[Proposition 4.10]{a71} and conclude that the left, middle and right squares in \eqref{homotopy-2} are commutative square diagrams for all $i$. By computing the cohomologies in the level of $i=-1$ one obtains:

\begin{equation}\label{homotopy-21}
\begin{tikzpicture}
back line/.style={densely dotted}, 
cross line/.style={preaction={draw=white, -, 
line width=6pt}}] 
\matrix (m) [matrix of math nodes, 
row sep=1em, column sep=3.em, 
text height=1.5ex, 
text depth=0.25ex]{ 
0&\text{h}^{-1}(f_{\alpha}^{*}\mathbb{E}^{\bullet}_{\alpha})&\text{h}^{-1}(\text{Cone}(f_{\alpha}^{*}g_{\alpha})[-1])&0\\
0&\text{h}^{-1}(f_{\beta}^{*}\mathbb{E}^{\bullet}_{\beta})&\text{h}^{-1}(\text{Cone}(f_{\beta}^{*}g_{\beta}))&0,\\};
\path[->]
(m-1-1) edge(m-1-2)
(m-1-3) edge (m-1-4)
(m-2-1) edge (m-2-2)
(m-1-3) edge node [right] {$\text{h}^{-1}(J_{\alpha\beta}[-1])$} (m-2-3)
(m-1-2) edge  node [above]{$\cong$} (m-1-3)
(m-1-2) edge  node [below]{$\rho_{1}$} (m-1-3)
(m-1-2) edge node [right] {$\text{id}$} (m-2-2)
(m-2-2) edge node [above]{$\cong$} (m-2-3)
(m-2-2) edge node [below]{$\rho_{2}$} (m-2-3)
(m-2-3) edge  (m-2-4);
\end{tikzpicture}
\end{equation}
where the vanishings on the ends are due to the fact that $\text{h}^{i}(f_{\alpha}^{*}\Omega_{\pi}\mid_{\mathcal{U}_{\alpha}}[-1])\cong 0$ and $\text{h}^{i}(f_{\beta}^{*}\Omega_{\pi}\mid_{\mathcal{U}_{beta}}[-1])\cong 0$ for $i=-1,0$. Hence, we conclude that by commutativity of the middle square, $\text{h}^{-1}(J_{\alpha\beta}[-1])$ is an isomorphism of cohomologies and moreover, given any $\nu\in \text{h}^{-1}(\text{Cone}(f_{\alpha}^{*}g_{\alpha})[-1])$:
\begin{equation}\label{kharboze1}
\text{id}\circ \rho_{1}^{-1}(\nu)=\rho^{-1}_{2}\circ \text{h}^{-1}(J_{\alpha\beta}[-1])(\nu).
\end{equation}
Note that given a choice of homotopy $h_{\alpha\beta}^{\vee}$ satisfying $$f_{\alpha}^{*}g^{\vee}_{\alpha}-f_{\beta}^{*}g^{\vee}_{\beta}=d\circ h^{\vee}_{\alpha\beta}+h^{\vee}_{\alpha\beta}\circ d,$$ and via restriction of the exact triangle in \eqref{lift-g'} to $\mathcal{U}_{\alpha\beta}$, similar to the above procedure, we obtain a commutative diagram:
\begin{equation}\label{homotopy-3}
\begin{tikzpicture}
back line/.style={densely dotted}, 
cross line/.style={preaction={draw=white, -, 
line width=6pt}}] 
\matrix (m) [matrix of math nodes, 
row sep=2em
, column sep=2.25em, 
text height=1.5ex, 
text depth=0.25ex]{ 
\text{Cone}(f_{\alpha}^{*}g^{\vee}_{\alpha})[-1]&\text{Cone}(f_{\alpha}^{*}g_{\alpha})[-1]&f_{\alpha}^{*}\text{T}_{\pi}\mid_{\mathcal{U}_{\alpha}}&\text{Cone}(f_{\alpha}^{*}g^{\vee}_{\alpha})\\
\text{Cone}(f_{\beta}^{*}g^{\vee}_{\beta})[-1]&\text{Cone}(f_{\beta}^{*}g_{\beta})[-1]&f_{\beta}^{*}\text{T}_{\pi}\mid_{\mathcal{U}_{\beta}}&\text{Cone}(f_{\beta}^{*}g^{\vee}_{\beta}),\\};
\path[->]
(m-1-1) edge node [right] {$J^{\vee}_{\alpha\beta}[-1]$} (m-2-1)
(m-1-2) edge node [above] {$f_{\alpha}^{*}g^{\vee}_{\alpha}$} (m-1-3)
(m-1-1) edge (m-1-2)
(m-1-3) edge  (m-1-4)
(m-1-2) edge node [right] {$J_{\alpha\beta}[-1]$} (m-2-2)
(m-1-4) edge node [right] {$J^{\vee}_{\alpha\beta}$} (m-2-4)
(m-2-3) edge (m-2-4)
(m-2-1) edge (m-2-2)
(m-2-2) edge node [above] {$f_{\beta}^{*}g^{\vee}_{\beta}$} (m-2-3)
(m-1-3) edge node [right] {$\text{id}$} (m-2-3);
\end{tikzpicture}
\end{equation}
and so similarly, we obtain a commutative diagram induced by the long exact sequences of cohomologies:
\begin{equation}\label{homotopy-22}
\begin{tikzpicture}
back line/.style={densely dotted}, 
cross line/.style={preaction={draw=white, -, 
line width=6pt}}] 
\matrix (m) [matrix of math nodes, 
row sep=2em
, column sep=.7em, 
text height=1.5ex, 
text depth=0.25ex]{ 
&\text{h}^{i}(\text{Cone}(f_{\alpha}^{*}g^{\vee}_{\alpha})[-1])&\text{h}^{i}(\text{Cone}(f_{\alpha}^{*}g_{\alpha})[-1])&\text{h}^{i}(f_{\alpha}^{*}\text{T}_{\pi}\mid_{\mathcal{U}_{\alpha}})&\text{h}^{i}(\text{Cone}(f_{\alpha}^{*}g^{\vee}_{\alpha})&\\
&\text{h}^{i}(\text{Cone}(f_{\beta}^{*}g^{\vee}_{\beta})[-1])&\text{h}^{i}(\text{Cone}(f_{\beta}^{*}g_{\beta})[-1])&\text{h}^{i}(f_{\beta}^{*}\text{T}_{\pi}\mid_{\mathcal{U}_{\beta}})&\text{h}^{i}(\text{Cone}(f_{\beta}^{*}g^{\vee}_{\beta})&,\\};
\path[->]
(m-1-2) edge node [right] {$\text{h}^{i}(J^{\vee}_{\alpha\beta}[-1])$} (m-2-2)
(m-1-3) edge (m-1-4)
(m-1-2) edge (m-1-3)
(m-1-4) edge  (m-1-5)
(m-1-3) edge node [right] {$\text{h}^{i}(J_{\alpha\beta}[-1])$} (m-2-3)
(m-1-5) edge node [right] {$\text{h}^{i}(J^{\vee}_{\alpha\beta})$} (m-2-5)
(m-2-4) edge (m-2-5)
(m-2-2) edge (m-2-3)
(m-2-3) edge  (m-2-4)
(m-1-4) edge node [right] {$\text{id}$} (m-2-4);
\end{tikzpicture}
\end{equation}
Now use \cite[Proposition 4.10]{a71} and conclude that the left, middle and right squares in \eqref{homotopy-2} are commutative square diagrams for all $i$ and in particular for $i=-1$:
\begin{equation}\label{homotopy-23}
\begin{tikzpicture}
back line/.style={densely dotted}, 
cross line/.style={preaction={draw=white, -, 
line width=6pt}}] 
\matrix (m) [matrix of math nodes, 
row sep=2.em
, column sep=1.em, 
text height=1.5ex, 
text depth=0.25ex]{ 
0&\text{h}^{-1}(\text{Cone}(f_{\alpha}^{*}g^{\vee}_{\alpha})[-1])&\text{h}^{-1}(\text{Cone}(f_{\alpha}^{*}g_{\alpha})[-1])&0&\text{h}^{-1}(\text{Cone}(f_{\alpha}^{*}g^{\vee}_{\alpha})&\\
0&\text{h}^{-1}(\text{Cone}(f_{\beta}^{*}g^{\vee}_{\beta})[-1])&\text{h}^{i}(\text{Cone}(f_{\beta}^{*}g_{\beta})[-1])&0&\text{h}^{-1}(\text{Cone}(f_{\beta}^{*}g^{\vee}_{\beta})&.\\};
\path[->]
(m-1-1) edge(m-1-2)
(m-2-1) edge (m-2-2)
(m-1-2) edge node [right] {$\text{h}^{-1}(J^{\vee}_{\alpha\beta}[-1])$} (m-2-2)
(m-1-3) edge (m-1-4)
(m-1-2) edge node [above] {$q_{1}$} (m-1-3)
(m-1-4) edge  (m-1-5)
(m-1-3) edge node [right] {$\text{h}^{-1}(J_{\alpha\beta}[-1])$} (m-2-3)
(m-1-5) edge node [right] {$\text{h}^{-1}(J^{\vee}_{\alpha\beta})$} (m-2-5)
(m-2-4) edge (m-2-5)
(m-2-2) edge node [above] {$q_{2}$} (m-2-3)
(m-2-3) edge  (m-2-4)
(m-1-4) edge node [right] {$\text{id}$} (m-2-4);
\end{tikzpicture}
\end{equation}
Hence by commutativity of the left square, and the fact that $\text{h}^{-1}(J_{\alpha\beta}[-1])$ is an isomorphism, then $\text{h}^{-1}(J^{\vee}_{\alpha\beta}[-1])$ is an isomorphism and moreover, for any $\mu\in \text{h}^{-1}(\text{Cone}(f_{\alpha}^{*}g^{\vee}_{\alpha})[-1])$ we have:
\begin{equation}\label{kharboze2}
\text{h}^{-1}(J_{\alpha\beta}[-1])\circ q_{1}(\mu)=q_{2}\circ \text{h}^{-1}(J^{\vee}_{\alpha\beta}[-1])(\mu).
\end{equation}
Now take an element $\mu\in \text{h}^{-1}(\text{Cone}(f_{\alpha}^{*}g^{\vee}_{\alpha})[-1])$ and note that by \eqref{kharboze1} and \eqref{kharboze2} we have:
\begin{equation}\label{hendoone1}
 \text{id}\circ \rho_{1}^{-1}\circ \text{h}^{-1}(J_{\alpha\beta}[-1])\circ q_{1}(\mu)=\rho^{-1}_{2}\circ \text{h}^{-1}(J_{\alpha\beta}[-1])\circ q_{2}\circ \text{h}^{-1}(J^{\vee}_{\alpha\beta}[-1])(\mu).
\end{equation}
Moreover, $\mathbb{L}^{\bullet}_{\mathfrak{N}_{\text{HFT}}}$ and $\mathbb{L}^{\bullet}_{\mathfrak{N}_{\text{FT}}}$ satisfy the condition that  $\text{h}^{-1}(\mathbb{L}^{\bullet}_{\mathfrak{N}_{\text{HFT}}})\cong \text{h}^{-1}(\mathbb{L}^{\bullet}_{\mathfrak{N}_{\text{FT}}})$. Hence, one easily observes that there exist maps $\lambda_{1}: \text{h}^{-1}(f^{*}_{\alpha}\mathbb{E}^{\bullet}_{\alpha})\rightarrow \text{h}^{-1}(f^{*}_{\alpha}(\mathbb{L}^{\bullet}_{\mathfrak{N}_{\text{HFT}}}\mid_{\mathcal{U}_{\alpha}}))$ and $\lambda_{2}: \text{h}^{-1}(f^{*}_{\beta}\mathbb{E}^{\bullet}_{\beta})\rightarrow \text{h}^{-1}(f^{*}_{\beta}(\mathbb{L}^{\bullet}_{\mathfrak{N}_{\text{HFT}}}\mid_{\mathcal{U}_{\beta}}))$ such that the following diagram commutes:
\begin{equation}\label{E-L-commute}
\begin{tikzpicture}
back line/.style={densely dotted}, 
cross line/.style={preaction={draw=white, -, 
line width=6pt}}] 
\matrix (m) [matrix of math nodes, 
row sep=1em
, column sep=3.em, 
text height=1.5ex, 
text depth=0.25ex]{ 
\text{h}^{-1}(f^{*}_{\alpha}\mathbb{E}^{\bullet}_{\alpha})&\text{h}^{-1}(f^{*}_{\alpha}(\mathbb{L}^{\bullet}_{\mathfrak{N}_{\text{HFT}}}\mid_{\mathcal{U}_{\alpha}}))\\
\text{h}^{-1}(f^{*}_{\beta}\mathbb{E}^{\bullet}_{\beta})&\text{h}^{-1}(f^{*}_{\beta}(\mathbb{L}^{\bullet}_{\mathfrak{N}_{\text{HFT}}}\mid_{\mathcal{U}_{\beta}})).\\};
\path[->]
(m-1-1) edge node [above] {$\lambda_{1}$} (m-1-2)
(m-1-2)edge node [left] {$\text{id}$} (m-2-2)
(m-2-1) edge node [above] {$\lambda_{2}$} (m-2-2)
(m-1-1) edge node [left] {$\text{id}$} (m-2-1);
\end{tikzpicture}
\end{equation}
 Now by \eqref{E-L-commute} and \eqref{hendoone1} it is seen that given $\mu\in \text{h}^{-1}(\text{Cone}(f_{\alpha}^{*}g^{\vee}_{\alpha})[-1])$,  we obtain an identity
 \begin{equation}
 \text{id}\circ \lambda_{1}\circ \text{id}\circ \rho_{1}^{-1}\circ \text{h}^{-1}(J_{\alpha\beta}[-1])\circ q_{1}(\mu)=\lambda_{2}\circ\text{id}\circ \rho^{-1}_{2}\circ \text{h}^{-1}(J_{\alpha\beta}[-1])\circ q_{2}\circ \text{h}^{-1}(J^{\vee}_{\alpha\beta}[-1])(\mu).
 \end{equation}
Let $\psi^{\vee}:=\text{id}\circ \rho^{-1}_{2}\circ \text{h}^{-1}(J_{\alpha\beta}[-1])\circ q_{2}\circ \text{h}^{-1}(J^{\vee}_{\alpha\beta}[-1])$. So far we have seen that in the level of $\text{h}^{-1}$ cohomology there exists a map $\psi^{\vee}:\text{h}^{-1}(\text{Cone}(f_{\alpha}^{*}g^{\vee}_{\alpha})[-1])\xrightarrow{\cong} \text{h}^{-1}(\text{Cone}(f_{\beta}^{*}g^{\vee}_{\alpha})[-1])$ such that $\lambda_{2}\circ \text{Im}(\psi^{\vee})=\text{Im}(\lambda_{1})$. 

Recall that by our notation, $\mathbb{G}^{\bullet}_{\alpha}:=\text{Cone}(g^{\vee}_{\alpha})[-1]$ and $\mathbb{G}^{\bullet}\mid_{\mathcal{U}_{\beta}}:=\text{Cone}(g^{\vee}_{\beta})[-1]$. Now dualize the construction and conclude that there exists a map $\psi: \text{h}^{1}(f_{\alpha}^{*}\mathbb{G}^{\bullet\vee}_{\alpha})\xrightarrow{\cong}\text{h}^{1}(f_{\beta}^{*}\mathbb{G}^{\bullet\vee}_{\beta})$, such that given a class $ob(f_{\alpha}^{*}\phi_{\alpha},g_{\alpha\beta},T,T')\in\text{h}^{1}(f^{*}_{\alpha}(\mathbb{L}^{\bullet}_{\mathfrak{N}_{\text{HFT}}}\mid_{\mathcal{U}_{\alpha}})^{\vee})$ (c.f. Diagram \eqref{gluable}) and for every point $p\in \mathcal{U}_{\alpha\beta}$ we have$$\psi\mid_{p}ob(f_{\alpha}^{*}\phi_{\alpha},g_{\alpha\beta},T,T')=ob(f_{\beta}^{*}\phi_{\beta},g_{\alpha\beta},T,T').$$This finishes the proof of Proposition \ref{semi-perf1}. 
\end{proof}
\begin{defn}\label{semi-perf}
\cite[Definition 3.1]{a70}. Let $X$ be a DM stack of finite type over an Artin stack $\mathcal{M}$. A semi perfect obstruction theory over $X\rightarrow \mathcal{M}$ consists 
of an \'{e}tale covering $\mathcal{U}=\coprod_{\alpha\in \Lambda}\mathcal{U}_{\alpha}$ of $X$ by schemes, and a truncated perfect relative obstruction theory$$\phi_{\alpha}:\mathbb{G}^{\bullet}_{\alpha}\rightarrow \mathbb{L}^{\bullet}_{\mathcal{U}_{\alpha}/\mathcal{M}}$$for each $\alpha\in \Lambda$ such that
\begin{enumerate}
\item for each $\alpha,\beta$ in $\Lambda$ there is an isomorphism$$\psi_{\alpha\beta}:\text{h}^{1}(\mathbb{G}^{\bullet\vee}_{\alpha}\mid_{\mathcal{U}_{\alpha\beta}})\xrightarrow{\cong}\text{h}^{1}( \mathbb{G}^{\bullet\vee}_{\beta}\mid_{\mathcal{U}_{\alpha\beta}})$$so that the collection $(\text{h}^{1}(\mathbb{G}^{\bullet\vee}_{\alpha}),\psi_{\alpha\beta})$ forms a descent datum of sheaves.
\item For any pair  $\alpha,\beta\in\Lambda$ the obstruction theories $\phi_{\alpha}\mid_{\mathcal{U}_{\alpha\beta}}$ and $\phi_{\beta}\mid_{\mathcal{U}_{\alpha\beta}}$ are $\nu$-equivalent.
\end{enumerate}
\end{defn}
The condition (1) above, that the $\nu$-equivalences we have constructed induce a descent datum of sheaves on $\text{h}^1$, requires that we carefully choose homotopies $h_{\alpha\beta}$ and $h^{\vee}_{\alpha\beta}$ on $\mathcal{U}_{\alpha\beta}$ so that the induced composite quasi-isomorphisms
$\psi_{\gamma\alpha}\circ\psi_{\beta\gamma}\circ\psi_{\alpha\beta}$ induce the identity maps on $\text{h}^1$.  

For now let us make the following assumption:
\begin{assumption}\label{descent assumption}
The homotopies $h_{\alpha\beta}$ and $h^{\vee}_{\alpha\beta}$ can be chosen so that the collection of data 
$(\text{h}^{1}(\mathbb{G}^{\bullet\vee}_{\alpha}),\psi_{\alpha\beta})$ forms a descent datum of sheaves.
\end{assumption}

\begin{theorem}\label{2step-trunc2}
Assuming the technical condition \ref{descent assumption}, the local deformation-obstruction theory in Theorem \ref{2step-trunc} satisfies the conditions of being a semi perfect deformation-obstruction theory in the sense of \cite[Definition 3.1]{a70} and hence, it defines a globally well-behaved virtual fundamental class over $\mathfrak{N}_{\text{HFT}}$.
\end{theorem}

\begin{proof}
 We need to show that both conditions in Definition \ref{semi-perf} are satisfied. For part (2) of Definition \ref{semi-perf}, apply Proposition \ref{semi-perf1} and conclude that $\phi_{\alpha}\mid_{\mathcal{U}_{\alpha\beta}}=f_{\alpha}^{*}\phi_{\alpha}$ and $\phi_{\beta}\mid_{\mathcal{U}_{\alpha\beta}}=f_{\beta}^{*}\phi_{\beta}$ are $\nu$-equivalent.  To prove part (1), first apply Proposition \ref{semi-perf1} and obtain the map$$\psi_{\alpha\beta}:\text{h}^{1}(\mathbb{G}^{\bullet\vee}_{\alpha}\mid_{\mathcal{U}_{\alpha\beta}})\xrightarrow{\cong}\text{h}^{1}( \mathbb{G}^{\bullet\vee}_{\beta}\mid_{\mathcal{U}_{\alpha\beta}}).$$
Now by Assumption \ref{descent assumption},  $(\text{h}^{1}(\mathbb{G}^{\bullet\vee}_{\alpha}),\psi_{\alpha\beta})$ forms a descent datum. 
\end{proof}

\begin{remark}
\emph{Making Assumption \ref{descent assumption} should, morally speaking, be unnecessary over more enhanced models of the moduli stacks.\;  The local models $\mathbb{G}_\alpha^\bullet$ can always be glued, up to higher homotopies, and thus should always give an $\infty$-stack in which the virtual normal cone lives. This can be seen in \cite{a92} and \cite{a93}. Another relevant discussion is by Sch\"{u}rg et al. \cite{a94} who study the advantage of using higher stacks in order to construct the virtual fundamental classes in more generality.  We expect that in the future a good intersection theory for $\infty$-stacks would allow us to construct a virtual cycle using this $\infty$-stack.  Such a construction is beyond the scope of the present article, however.}
\end{remark}

\begin{remark}\label{important}
\emph{Assumption \ref{descent assumption} holds true in the setting of Part II when the base variety $X$ is chosen to be a toric variety (see Lemma \ref{auto}). This is due to the fact that the torus equivariant highly frozen triples are identifiable with $r$ copies of torus equivariant PT pairs (c.f. Proposition \ref{finalword}). Moreover, in Part II a direct calculation of equivariant vertex for highly frozen triples is carried out and the result is shown to match precisely with the equivariant vertex associated to  $r$ copies of (twisted) PT deformation-obstruction theory.}
\end{remark}

\section{Part II (Calculations)}\label{torus-high}\label{sec11}
Let $X$ be given as a toric variety with an action of $\textbf{T}=\mathbb{C}^{*3}$. Consider the ample line bundle over $X$ given by $\mathcal{O}_{X}(1)$. By the usual arguments, and similar to the theory of stable pairs \cite{a17}, the action of $\textbf{T}$ on $X$ induces an action on $\mathfrak{N}_{\text{HFT}}$;
\begin{prop}\label{p1p2-action}
(\textit{Geometric torus action}) Let $X$ be given as a nonsingular toric threefold. Let $\textbf{T}$ be the $(\mathbb{C}^{*})^{3}$ action on $X$. Having fixed an equivariant structure on $\mathcal{O}_{X}(1)$, there exists an induced action of $\textbf{T}$ on moduli stack of stable highly frozen triples $\overline{m}^{\textbf{T}}:\textbf{T}\times\mathfrak{N}_{\operatorname{HFT}}\rightarrow\mathfrak{N}_{\operatorname{HFT}}$ given by pre-composing the pullback (via the action of $\textbf{T}$ on $X$) of triples  with the inverse of the isomorphism $\psi$.
\end{prop}
\begin{proof}
The proof is essentially followed by applying the strategy of \cite[Proposition 4.1]{a34} to the highly frozen triples. In other words, considering the highly frozen triples $(E,F,\phi,\psi)$, the action of $\textbf{T}$ induces an action (given by the pullback) on the moduli spaces of sheaves associated to $E$ and $F$ and this action respects the morphism $E\rightarrow F$ after pre-composing the pull back $t^{*}E\rightarrow t^{*}F$ with the isomorphism $\psi^{-1}: E\rightarrow t^{*}E$ for all $t\in \textbf{T}$. 
 \end{proof}
 \begin{prop}\label{equiv1}
Let $S$ be a parametrizing scheme of finite type over $\mathbb{C}$. Let $(\mathcal{E},\mathcal{F},\phi,\psi)_{S}$ denote a family of stable highly frozen triples over $S$. Suppose that for all $t=(\lambda_{1},\lambda_{2},\lambda_{3})\in \textbf{T}$ we have $t^{*}((\mathcal{E},\mathcal{F},\phi,\psi)_{S})\cong (\mathcal{E},\mathcal{F},\phi,\psi)_{S}$, then $(\mathcal{E},\mathcal{F},\phi,\psi)_{S}$ admits a $\textbf{T}$-equivariant structure. 
\end{prop}
 \begin{proof} We give an adaptation of the proof given in \cite[Lemma 3.3]{a74} to our case. By assumption for any $t\in \textbf{T}$ one has $t^{*}((\mathcal{E},\mathcal{F},\phi,\psi)_{S}\cong (\mathcal{E},\mathcal{F},\phi,\psi)_{S}.$ Let $\sigma:\textbf{T}\times X\rightarrow X$ denote the torus action on $X$ and $\pi_{2}:\textbf{T}\times X\rightarrow X$ be the projection onto the second factor. Let $q:X\times S\rightarrow S$ be the projection onto $S$. One needs to show that there exists a map:
\begin{equation}\label{line-bundle}
\rho:\text{Ext}^{0}_{id_{\textbf{T}}\times q}((\pi_{2}\times\text{id}_{S})^{*}(\mathcal{E},\mathcal{F},\phi,\psi)_{S},(\sigma\times \text{id}_{S})^{*}(\mathcal{E},\mathcal{F},\phi,\psi)_{S})\rightarrow \mathcal{O}_{\textbf{T}\times S},
\end{equation}
which is an isomorphism of line bundles over $\textbf{T}\times S$. Here 
\begin{align}
&
\text{Ext}^{0}_{id_{\textbf{T}}\times q}((\pi_{2}\times\text{id}_{S})^{*}(\mathcal{E},\mathcal{F},\phi,\psi)_{S},(\sigma\times \text{id}_{S})^{*}(\mathcal{E},\mathcal{F},\phi,\psi)_{S})\notag\\
&
:=R^{0}(q\times \text{id}_{\textbf{T}})_{*}(\mathscr{H}om((\pi_{2}\times\text{id}_{S})^{*}(\mathcal{E},\mathcal{F},\phi,\psi)_{S},(\sigma\times \text{id}_{S})^{*}(\mathcal{E},\mathcal{F},\phi,\psi)_{S})).
\end{align}By definition of $\mathfrak{N}_{\operatorname{HFT}}$, choosing a family of stable highly frozen triples over $S$ is equivalent to choosing a unique map $S\rightarrow \mathfrak{N}_{\operatorname{HFT}}$. Since $(\sigma\times \text{id}_{S})^{*}(\mathcal{E},\mathcal{F},\phi,\psi)_{S}$ and $(\pi_{2}\times \text{id}_{S})^{*}(\mathcal{E},\mathcal{F},\phi,\psi)_{S}$ are two families over $\mathfrak{N}_{\operatorname{HFT}}$, they both define maps $f:\textbf{T}\times S \rightarrow \mathfrak{N}_{\operatorname{HFT}}$ and $g:\textbf{T}\times S\rightarrow \mathfrak{N}_{\operatorname{HFT}}$ respectively. By the uniqueness property, both maps are uniquely isomorphic to each other. On the other hand by Lemma \ref{pf-auto-1} the complexes representing $\tau'$-stable highly frozen triples are simple objects hence:
\begin{align}
&
\text{Ext}^{0}_{id_{\textbf{T}}\times q}((\pi_{2}\times\text{id}_{S})^{*}(\mathcal{E},\mathcal{F},\phi,\psi)_{S},(\sigma\times \text{id}_{S})^{*}(\mathcal{E},\mathcal{F},\phi,\psi)_{S})\notag\\
&
\cong \text{Ext}^{0}_{id_{\textbf{T}}\times q}((\mathcal{E},\mathcal{F},\phi,\psi)_{\textbf{T}\times S},(\mathcal{E},\mathcal{F},\phi,\psi)_{\textbf{T}\times S})\cong \mathcal{O}_{\textbf{T}\times S}.
\end{align}
Now the inverse image of $1\in \mathcal{O}_{\textbf{T}\times S}$ via the map $\rho$ in \eqref{line-bundle} gives a section of $$\text{Ext}^{0}_{id_{\textbf{T}}\times q}((\mathcal{E},\mathcal{F},\phi,\psi)_{\textbf{T}\times S},(\mathcal{E},\mathcal{F},\phi,\psi)_{\textbf{T}\times S})$$ which induces a section of $$\text{Ext}^{0}_{id_{\textbf{T}}\times q}((\pi_{2}\times\text{id}_{S})^{*}(\mathcal{E},\mathcal{F},\phi,\psi)_{S},(\sigma\times \text{id}_{S})^{*}(\mathcal{E},\mathcal{F},\phi,\psi)_{S})$$ which induces a morphism $(\pi_{2}\times\text{id}_{S})^{*}(\mathcal{E},\mathcal{F},\phi,\psi)_{S}\rightarrow (\sigma\times \text{id}_{S})^{*}(\mathcal{E},\mathcal{F},\phi,\psi)_{S}$. Moreover, it can be checked that this morphism is an isomorphism for every point in the moduli scheme of stable highly frozen triples. Therefore, it is an isomorphism everywhere and this finishes the proof.\end{proof}\label{sec13}
\begin{defn}\label{T0-action}
(\textit{Non-geometric torus action}) Define an action $\sigma_{0}:\text{T}_{0}\times\mathfrak{N}_{\operatorname{HFT}}\rightarrow\mathfrak{N}_{\operatorname{HFT}}$ where $\text{T}_{0}=(\mathbb{C}^{*})^{r}$, and $\sigma_{0}$ acts on $\mathfrak{N}_{\operatorname{HFT}}$ by rescaling in components of $\mathcal{O}_{X}(-n)^{\oplus r}$. 
\end{defn}
\begin{prop}\label{equiv2}
Let $S$ be a parametrizing scheme of finite type over $\mathbb{C}$. Let $(\mathcal{E},\mathcal{F},\phi,\psi)_{S}$ denote a family of stable highly frozen triples over $S$. Suppose that for all $t_{0}=(z_{1},\cdots, z_{r})\in \text{T}_{0}$, $\sigma_{0}(t_{0},(\mathcal{E},\mathcal{F},\phi,\psi)_{S})\cong (\mathcal{E},\mathcal{F},\phi,\psi)_{S}$. Then $(\mathcal{E},\mathcal{F},\phi,\psi)_{S}$ admits a $\text{T}_{0}$-equivariant structure:
$$\sigma_{0}^{*}(\mathcal{E},\mathcal{F},\phi,\psi)_{S}\cong \tilde{p}_{2}^{*}(\mathcal{E},\mathcal{F},\phi,\psi)_{S},$$ where $\tilde{p}_{2}:\text{T}_{0}\times\mathfrak{N}_{\text{HFT}}\rightarrow\mathfrak{N}_{\operatorname{HFT}}$ is the projection onto the second factor.
\end{prop}
\begin{proof} Apply the proof of Proposition \ref{equiv1} to $\text{T}_{0}$ and the universal family $(\mathbb{E},\mathbb{F},\phi,\psi)$ and use the simpleness property of stable highly frozen triples. 
\end{proof}
Now use the notation $\mathcal{T}:=\textbf{T}\times \text{T}_{0}$.
\begin{prop}\label{finalword}
\emph A $\mathcal{T}$-equivariant stable highly frozen triple is decomposable into $r$ copies of $\textbf{T}$-equivariant highly frozen triples of rank 1, $\mathcal{O}_{X}(-n)_{i}\xrightarrow{s_{i}} F_{i}, 1\leq i \leq r$. In other words the $\mathcal{T}$-equivariant highly frozen triples satisfy the following identity:
\begin{equation}\label{oplus}
[\mathcal{O}^{\oplus r}_{X}(-n)\xrightarrow{s} F]^{\mathcal{T}}\cong\bigoplus_{i=1}^{r}\left[\mathcal{O}_{X}(-n)\xrightarrow{s_{i}} F_{i}\right]^{\textbf{T}}.
\end{equation}
\end{prop}
\begin{proof}
By Proposition \ref{equiv2}, the action of $\operatorname{T}_{0}$ on a point $p\in\mathfrak{N}_{\operatorname{HFT}}$  induces a $\operatorname{T}_{0}$-weight decomposition on $\mathcal{O}_{X}^{\oplus r}(-n)$. Let $(w_{1},\cdots,w_{r})$ denote the $r$-tuple of weights. In fact $w_{i}$ for $1\leq i\leq r$ are given by $r$ tuples$$(1,0,\cdots,0), (0,1,\cdots,0),\cdots, (0,\cdots,0,1).$$ Now let $M$ be the module associated to the sheaf $\mathcal{O}_{X}^{\oplus r}(-n)$, and denote by $M^{0}$ the module associated to the sheaf $\mathcal{O}_{X}(-n)$. The graded piece of $M^{\mathcal{T}}$ which sits in $w_{i}$ weight-space is given by the module $0\oplus \cdots \oplus M^{0}\oplus \cdots \oplus 0$, where $M^{0}$ sits in the $i$'th position which we denote by $M^{0}_{i}$. Moreover observe that, under the action of $\textbf{T}$, we obtain a $\textbf{T}$-graded weight decomposition of $M^{0}_{i}$ itself, given as$$M^{0}_{i}\cong \bigoplus_{(m_{1},m_{2},m_{3})}M^{0}_{i}(m_{1},m_{2},m_{3})$$Therefore, the $\mathcal{T}$-weight decomposition of $M^{\mathcal{T}}$ is given by
\begin{align}
&
M^{\mathcal{T}}\cong\bigoplus_{i=1}^{r}\left(\bigoplus_{(m_{1},m_{2},m_{3})}M^{0}_{i}(m_{1},m_{2},m_{3})\right)\notag\\
\end{align}
which means that, sheaf theoretically, the following $\mathcal{T}$-equivariant isomorphism holds true$$[\mathcal{O}_{X}^{\oplus r}(-n)]^{\mathcal{T}}\cong \bigoplus_{i=1}^{r}\mathcal{O}^{\textbf{T}}_{X}(-n).$$ 
Now apply the same argument as above to $F$ and use the property of morphisms between two graded sheaves of modules (c.f. \cite{DeMazur}) to see that $F$ will have a $\mathcal{T}$-weight decomposition compatible to that of $\mathcal{O}_{X}^{\oplus r}(-n)$ and this proves the statement of Proposition \ref{finalword}.
\end{proof}
\begin{remark}
By Proposition \ref{finalword} the $\tau'$-stable highly frozen triple on left hand side of isomorphism \eqref{oplus} satisfies the condition of having zero dimensional cockerel. Therefore, the cockerel sheaf induced by the right hand side of \eqref{oplus} is also zero dimensional, which means the $r$-fold sum of the highly frozen triples of rank $1$ appearing on the right hand side of \eqref{oplus} is stable in the sense of Lemma \ref{PT-stab}. 

Note that by Lemma \ref{lemma2}, our notion of $\tau'$-stability condition in Lemma \ref{lemma2} is obtained as the $q(m)\to \infty$ limit of Le Potier's notion of polynomial stability condition with strict inequality (c.f. Equation \eqref{nonsense2}). Therefore, we have avoided the strictly semistable objects in our construction. Moreover, the only automorphisms of our higher rank objects on both sides of isomorphism \eqref{oplus} are given by  identity. In order to see this, replace $F$ with $\oplus_{i=1}^{r}F_{i}$ in Lemma \ref{pf-auto-1}.
\end{remark}
\begin{lemma}\label{oplus-stable}
 Consider the right hand side of $\mathcal{T}$-equivariant isomorphism \eqref{oplus}. Then it is true that each $s_{i}:=\mathcal{O}^{\textbf{T}}_{X}(-n)\rightarrow F_{i}^{\textbf{T}}$ satisfies PT stability condition. 
 \end{lemma}
 \begin{proof}
 We show that $\text{Coker}(s_{i})$ has zero dimensional support. Given a $\mathcal{T}$-graded morphism in the sense of Proposition \ref{finalword}$$\mathcal{O}^{\textbf{T}}_{X}(-n)\oplus \cdots\oplus \mathcal{O}^{\textbf{T}}_{X}(-n)\xrightarrow{s_{1},\cdots, s_{r}}F^{\textbf{T}}_{1}\oplus \cdots \oplus F^{\textbf{T}}_{r},$$let $Q_{1,\cdots, r}:=\text{Coker}(s_{1}, \cdots, s_{r})$. Moreover, let $Q_{i}:=\text{Coker}(s_{i})$. Now let $\mathcal{Q}:=\text{Coker}(s)$ be the cokernel sheaf induced by the left hand side of isomorphism \eqref{oplus}. It is then true by Proposition \ref{finalword} that $$\mathcal{Q}\cong^{\mathcal{T}} Q_{1,\cdots, r},$$which implies that$$\text{Supp}(\mathcal{Q})\cong \bigcup_{i=1}^{r}\text{Supp}(Q_{i}).$$Now assume that there exists some $Q_{i}$ for which $d_{i}:=\text{dim}(\text{Supp}(Q_{i}))\geq 1$ , then we get a contradiction with $\mathcal{Q}$, being zero-dimensional, which contradicts the stability of $[\mathcal{O}^{\oplus r}_{X}(-n)\xrightarrow{s} F]^{\mathcal{T}}$. Therefore, $d_{i}=0$ for all $1\leq i\leq r$ which, together with Remark \ref{PT-stab}, proves the statement.
 \end{proof}
\begin{prop}\label{r-PT}
Let $\textbf{Q}$ denote a $\mathcal{T}$-fixed component of $\mathfrak{N}_{\text{HFT}}$. The following quasi-isomorphism holds true over $\mathcal{D}^{b}(\textbf{Q})$ $$\mathbb{G}^{\bullet}\cong \bigoplus_{i=1}^{r} (\mathbb{E}^{\bullet,\textbf{T}}_{PT})$$ where $\mathbb{E}^{\bullet,\textbf{T}}_{PT}$ is the $\textbf{T}$-fixed PT deformation-obstruction theory of perfect amplitude $[-1,0]$.
\end{prop}
\begin{proof}Let $\iota:\textbf{Q}\hookrightarrow \mathfrak{N}_{\text{HFT}}$ denote a natural embedding of a $\mathcal{T}$-fixed component of $\mathfrak{N}_{\text{HFT}}$. By Proposition \ref{finalword} over $\textbf{Q}$ we have $$[\mathcal{O}^{\oplus r}_{X}(-n)\rightarrow F]^{\mathcal{T}}\cong\bigoplus_{i=1}^{r}\left[\mathcal{O}_{X}(-n)\rightarrow F_{i}\right]^{\textbf{T}}.$$Now let $\mathcal{U}=\coprod_{\alpha}\mathcal{U}_{\alpha}$ be an atlas of affine schemes for $\textbf{Q}$. Then apply Proposition \ref{finalword} to the universal object $\mathbb{I}^{\bullet}\in \mathcal{D}^{b}(X\times \mathfrak{N}_{\text{HFT}})$ and conclude that $$\iota^*R\mathscr{H}om(\mathbb{I}^{\bullet},\mathbb{I}^{\bullet})^{\mathcal{T}}_{0}|_{\mathcal{U}_{\alpha}\times X}\cong \bigoplus_{i=1}^{r}R\mathscr{H}om(\mathbb{I}^{\bullet, \textbf{T}}_{i},\mathbb{I}^{\bullet, \textbf{T}}_{i})_{0},$$where $$\mathbb{I}^{\bullet, \textbf{T}}_{i}\cong [O_{\mathcal{U}_{\alpha}\times X}(-n)\rightarrow \mathcal{F}_{i}]^{\textbf{T}},$$is the universal $\textbf{T}$-equivariant PT stable pair.

Now by construction in Theorem \ref{reldef-f} obtain that $\iota^{*}\mathbb{E}^{\bullet}|_{\mathcal{U}_{\alpha}}$ is given by a 4 term complex of vector bundles with $\text{h}^{-2}(\iota^{*}\mathbb{E}^{\bullet}|_{\mathcal{U}_{\alpha}})\cong (\mathbb{G}_{m})^{r}\otimes \mathcal{O}_{\mathcal{U}_{\alpha}}$. Then, by construction in Theorem \ref{2step-trunc} and Diagram \eqref{q-isom} we can immediately see that $$\text{h}^{-1}(\iota^{*}\mathbb{G}^{\bullet}|_{\mathcal{U}_{\alpha}})\cong \bigoplus_{i=1}^{r} \text{Ext}^{1}(\mathbb{I}^{\bullet, \textbf{T}}_{i,\alpha},\mathbb{I}^{\bullet, \textbf{T}}_{i,\alpha})_{0}\,\,\,\text{and}\,\,\, \text{h}^{0}(\iota^{*}\mathbb{G}^{\bullet}|_{\mathcal{U}_{\alpha}})\cong \bigoplus_{i=1}^{r} \text{Ext}^{2}(\mathbb{I}^{\bullet, \textbf{T}}_{i,\alpha},\mathbb{I}^{\bullet, \textbf{T}}_{i,\alpha})_{0}.$$In other words we have the quasi-isomorphism in $\mathcal{D}^{b}(\mathcal{U}_{\alpha})$: 
\begin{equation}\label{local}
\iota^{*}\mathbb{G}^{\bullet}|_{\mathcal{U}_{\alpha}}\cong \bigoplus_{i=1}^{r} (\mathbb{E}^{\bullet,\textbf{T}}_{PT})|_{\mathcal{U}_{\alpha}}.
\end{equation}
 Now it is immediately seen that the local isomorphisms \eqref{local} glue to each other globally over $\textbf{Q}$.
\end{proof}
\begin{lemma}\label{auto}
Assumption \ref{descent assumption} in Part I holds true over $\textbf{Q}$. 
\end{lemma}
 \begin{proof}
 Use Proposition \ref{finalword} and Proposition \ref{r-PT} and obtain that $\iota^*\mathbb{G}^{\bullet}$ over $\textbf{Q}$ is given by $r$ copies of $\textbf{T}$-fixed PT deformation-obstruction theory of perfect amplitude $[-1,0]$. Now by construction of Chang-Li \cite[Definition 3.1]{a70} we know that a perfect obstruction theory is also semi-perfect which means that Assumption  \ref{descent assumption} is automatically satisfied for the restricted complex $\iota^{*}\mathbb{G}^{\bullet}$ globally over $\textbf{Q}$.
\end{proof}
\subsection{The threefold equivariant vertex for HFT}\label{sec15}
In this section let us assume $X$ is given as toric Calabi-Yau threefold with the action $\textbf{T}$. Following Proposition \ref{finalword}, the identification of the highly frozen triples of type $(r,P_{F})$ with $r$-fold copies of PT stable pairs makes it easy to see that the $\mathcal{T}$-fixed components of the moduli stack of highly frozen triples are obtained as $r$-fold product of $\textbf{T}$-fixed components of the moduli stack of stable pairs, which are conjectured by Pandharipande and Thomas in \cite[Conjecture 2]{a17} to be nonsingular and compact. Let $\textbf{Q}$ denote the $\mathcal{T}$-fixed locus of $\mathfrak{N}_{\text{HFT}}$. We assume that $\textbf{Q}$ is nonsingular, connected and compact. Let $\iota_{\textbf{Q}}: \textbf{Q}\hookrightarrow\mathfrak{N}_{\text{HFT}}$ denote the natural embedding. Let $\mathbb{G}^{\bullet}_{\textbf{Q}}:=(\iota_{\textbf{Q}})^{*}\mathbb{G}^{\bullet}$ where $\mathbb{G}^{\bullet}$ is the deformation-obstruction theory obtained in Theorem \ref{2step-trunc}. Let $\mathbb{G}^{\bullet,\mathcal{T}}_{\textbf{Q}}$ and $\mathbb{G}^{\bullet,\textbf{m}}_{\textbf{Q}}$ denote the sub-bundles of $\mathbb{G}^{\bullet}_{\textbf{Q}}$ with zero and nonzero characters respectively. By Theorem \ref{2step-trunc}, the $\mathcal{T}$-fixed deformation-obstruction theory restricted to $\textbf{Q}$ is given by a map of perfect complexes:
\begin{equation}\label{smooth-obs}
\mathbb{G}^{\bullet,\mathcal{T}}_{\textbf{Q}}\xrightarrow{\phi} \mathbb{L}^{\bullet}_{\textbf{Q}}.
\end{equation} 
Here $\mathbb{G}^{\bullet,\mathcal{T}}_{\textbf{Q}}$ is represented by a two term complex of vector bundles $G_{\textbf{Q}}^{-1,\mathcal{T}}\rightarrow G_{\textbf{Q}}^{0,\mathcal{T}}$. By the virtual localization formula \cite{a25}:

\begin{equation}\label{moving}
\bigg[\mathfrak{N}_{\text{HFT}}\bigg]^{vir}=\sum_{\textbf{Q}\subset\mathfrak{N}_{\operatorname{HFT}}}\iota_{\textbf{Q}*}\bigg(\frac{e(G^{\textbf{m}}_{1,\textbf{Q}})}{e(G^{\textbf{m}}_{0,\textbf{Q}})}\cap [\textbf{Q}]^{vir}\bigg).
\end{equation}
Where $G^{\textbf{m}}_{0,\textbf{Q}}$ and $G^{\textbf{m}}_{1,\textbf{Q}}$ denote the dual of $G^{0,\textbf{m}}_{\textbf{Q}}$ and $G^{-1,\textbf{m}}_{\textbf{Q}}$ respectively. Now we rewrite \eqref{moving} with respect to the Euler classes $e(G_{1,\textbf{Q}})$ and $e(G_{0,\textbf{Q}})$ where $G_{0,\textbf{Q}}$ and $G_{1,\textbf{Q}}$ denote the dual of $G^{0}_{\textbf{Q}}$ and $G^{-1}_{\textbf{Q}}$ respectively. In order to do so, we use the description of the virtual tangent space with respect to the $\mathcal{T}$-fixed deformation-obstruction theory.  If $\textbf{Q}$ is assumed to be nonsingular, then $\mathbb{L}^{\bullet}_{\textbf{Q}}:=0\rightarrow \Omega_{\textbf{Q}}.$ The $\mathcal{T}$-fixed deformation-obstruction theory \eqref{smooth-obs} induces a composite morphism $G_{\textbf{Q}}^{-1,\mathcal{T}}\rightarrow G_{\textbf{Q}}^{0,\mathcal{T}}\xrightarrow{\phi}\Omega_{\textbf{Q}}.$ The kernel of this composite morphism is the obstruction bundle $\textbf{K}$ and by definition $[\textbf{Q}]^{vir}=e(\textbf{K}^{\vee})\cap [\textbf{Q}].$ The $\mathcal{K}$-theory class of $\textbf{K}^{\vee}$ is given as follows:
\begin{equation}\label{bdle}
[\textbf{K}^{\vee}]=[G_{1,\textbf{Q}}^{\mathcal{T}}]-[G_{0,\textbf{Q}}^{\mathcal{T}}]+[T_{\textbf{Q}}],
\end{equation} 
where $G^{\mathcal{T}}_{0,\textbf{Q}}$ and $G^{\mathcal{T}}_{1,\textbf{Q}}$ denote the dual of $G^{0,\mathcal{T}}_{\textbf{Q}}$ and $G^{-1,\mathcal{T}}_{\textbf{Q}}$ respectively. Therefore one has:
\begin{equation}\label{dual-K}
e(\textbf{K}^{\vee})=\frac{e(G_{1,\textbf{Q}}^{\mathcal{T}})}{e(G_{0,\textbf{Q}}^{\mathcal{T}})}\cdot e(T_{\textbf{Q}}).
\end{equation}
By \eqref{moving} and \eqref{dual-K} the virtual fundamental class of $\mathfrak{N}_{\text{HFT}}$ is obtained as
\begin{equation}\label{virt-loc-form}
\bigg[\mathfrak{N}_{\text{HFT}}\bigg]^{vir}=\sum_{\textbf{Q}\subset\mathfrak{N}_{\operatorname{HFT}}}\iota_{\textbf{Q}*}\bigg(\frac{e(G_{1,\textbf{Q}})}{e(G_{0,\textbf{Q}})}\cdot e(T_{\textbf{Q}})\cap [\textbf{Q}]\bigg).
\end{equation}
Now we compute the difference $[G_{0,\textbf{Q}}]-[G_{1,\textbf{Q}}]$ in the $\mathcal{T}$-equivariant $\mathcal{K}$-theory of $\textbf{Q}$. Consider a point $p\in \textbf{Q}$ represented by the complex $I^{\bullet\mathcal{T}}:=[\mathcal{O}^{\oplus r}_{X}(-n)\rightarrow F]^{\mathcal{T}}.$ The difference $[G_{0,\textbf{Q}}]-[G_{1,\textbf{Q}}]$ over this point is the virtual tangent space at this point. We use the quasi isomorphism in diagram \eqref{q-isom} to compute the virtual tangent space:  
\begin{align}\label{virtg}
&
\mathcal{T}^{\textbf{Q}}_{I^{\bullet}}=[\text{Coker}(d')]-[\text{Ker}(d)]=
\left([\pi^{*}E^{1}]-[\pi^{*}E^{0}]+[\pi^{*}E^{-1}]-[\pi^{*}E^{-2}]\right)+\left(\cancel{[T_{\pi}]}-\cancel{[\Omega_{\pi}]}\right),
\end{align}
where $E^{i}$ for $i=-1,\cdots,2$ are described in \eqref{vec-cplx}, and the cancellation on the right hand side of  \eqref{virtg} is due to isomorphism of $\Omega_{\pi}$ and $T_{\pi}$, which is seen from their triviality. Now since the point $p\in \textbf{Q}$ is represented by $I^{\bullet,\mathcal{T}}$, the following identities hold true:
\begin{align}\label{virt-tang}
&
\mathcal{T}^{\textbf{Q}}_{I^{\bullet}}=\sum_{i=0}^{3}(-1)^{i}\cdot[\text{Ext}^{i}(I^{\bullet\mathcal{T}},I^{\bullet\mathcal{T}})_{0}]=[\chi(\mathcal{O}_{X},\mathcal{O}_{X})]-[\chi(I^{\bullet\mathcal{T}},I^{\bullet\mathcal{T}})].
\end{align}
\subsubsection{Computation of $[\chi(\mathcal{O}_{X},\mathcal{O}_{X})]-[\chi(I^{\bullet},I^{\bullet})]$}\label{sec16}
Using \v{C}ech cohomology with respect to an affine open cover $\bigcup_{\alpha} \mathcal{U}_{\alpha}\rightarrow\mathfrak{N}_{\operatorname{HFT}}$ we obtain:$$\chi(I^{\bullet},I^{\bullet})=\sum_{i,j=0}^{3}(-1)^{i+j}\mathfrak{C}^{i}(\mathcal{E}xt^{j}(I^{\bullet},I^{\bullet}))\,\,\,\text{and}\,\,\,\chi(\mathcal{O}_{X},\mathcal{O}_{X})=\sum_{i,j=0}^{3}(-1)^{i+j}\mathfrak{C}^{i}(\mathcal{E}xt^{j}(\mathcal{O}_{X},\mathcal{O}_{X})).$$By definition, the sheaf $F$ appearing in the stable highly frozen triples is pure of dimension 1. Therefore, the restriction of $F$ over the triple and quadruple intersections of $\mathcal{U}_{\alpha}$'s vanishes and over such intersections $I^{\bullet}\cong\mathcal{O}^{\oplus r}_{X}(-n)$. 
\begin{defn}
Define:
\begin{align}\label{chi-I}
 &
 \mathcal{T}^{1}_{[I^{\bullet}]}=\bigoplus_{\alpha}\left(\text{h}^{0}(\mathcal{U}_{\alpha},\mathcal{O}_{X})-\sum_{j}(-1)^{j}\text{h}^{0}(\mathcal{U}_{\alpha},\mathcal{E}xt^{j}(I^{\bullet},I^{\bullet})))\right)\notag\\
 &
\mathcal{T}^{2}_{[I^{\bullet}]}=\bigoplus_{\alpha,\beta}\left(\text{h}^{0}(\mathcal{U}_{\alpha\beta},\mathcal{O}_{X})-\sum_{j}(-1)^{j}\text{h}^{0}(\mathcal{U}_{\alpha\beta},\mathcal{E}xt^{j}(I^{\bullet},I^{\bullet})))\right)\notag\\
  &
\mathcal{T}^{3}_{[I^{\bullet}]}=\bigoplus_{\alpha,\beta,\gamma}\left((1-r^{2}) \text{h}^{0}(\mathcal{U}_{\alpha\beta\gamma},\mathcal{O}_{X})\right)\, \text{and} \,
\mathcal{T}^{4}_{[I^{\bullet}]}=\bigoplus_{\alpha,\beta,\gamma,\delta}(1-r^{2}) \text{h}^{0}(\mathcal{U}_{\alpha\beta\gamma\delta},\mathcal{O}_{X}).\notag\\
\end{align}
\end{defn}
By Definition \ref{chi-I} and Equation \eqref{virt-tang}, the virtual tangent space is obtained as:
\begin{equation}\label{virt-tang2}
\mathcal{T}_{[I^{\bullet}]}=\mathcal{T}^{1}_{[I^{\bullet}]}-\mathcal{T}^{2}_{[I^{\bullet}]}+\mathcal{T}^{3}_{[I^{\bullet}]}-\mathcal{T}^{4}_{[I^{\bullet}]}.
\end{equation}
Now let $(t_{1},t_{2},t_{3})$ be defined as the weights of $\textbf{T}$. Moreover, let $(w_{1},\cdots,w_{r})$ be defined as weight of the action of $\text{T}_{0}$. Here $w_{i}$ is given by tuples $(0,\cdots,1,\cdots,0)$ where $1$ is positioned in the $i$'th position in the tuple; We need to compute the $\mathcal{T}$-character of $\mathcal{T}^{i}_{[I^{\bullet}]}$ for $i=1,\cdots 4$ in \eqref{virt-tang2}. Choose a \v{C}ech cover $\mathcal{U}=\bigcup_{\alpha}\mathcal{U}_{\alpha}$ of $X$. The restriction of each copy of $O^{\textbf{T}}_{X}(-n)\rightarrow F^{\textbf{T}}_{i}$ to the underlying supporting curve $\mathcal{C}_{\alpha}$ of $F^{\textbf{T}}_{i}$ induces an exact sequence of the form:
\begin{equation}\label{exact}
0\rightarrow \mathcal{O}^{\textbf{T}}_{\mathcal{C}_{\alpha}}(-n)\rightarrow (F^{\textbf{T}}_{i})_{\alpha}\rightarrow (Q^{\textbf{T}}_{i})_{\alpha}\rightarrow 0,
\end{equation}
By $\tau'$-stability, the sheaf $(F^{\textbf{T}}_{i})_{\alpha}$ may be zero and if it is nonzero then the cokernel $(Q^{\textbf{T}}_{i})_{\alpha}$ has to have zero dimensional support. Note that, following Proposition \ref{finalword} and Lemma \ref{oplus-stable}, we have that $\textbf{Q}_{\alpha}:=\bigoplus_{i=1}^{r}(Q^{\textbf{T}}_{i})_{\alpha},$ where each $(Q^{\textbf{T}}_{i})_{\alpha}$ has zero dimensional support. Now we use the strategy similar to \cite[Section 4.4]{a18} and \cite[Section 4.7]{a15} to compute the $\textbf{T}$ character of each summand, $(F^{\textbf{T}}_{i})_{\alpha}$. Let $\textbf{F}^{\textbf{T}}_{i,\alpha}$ denote the $\textbf{T}$-character of each summand. Let $(\textbf{P}_{i})_{\alpha}(t_{1},t_{2},t_{3})$ denote the associated Poincar{\'e} polynomial of $(\mathbb{I}^{\bullet}_{i})_{\alpha}:=\bigg(\mathcal{O}^{\textbf{T}}_{X}(-n)\rightarrow F_{i}^{\textbf{T}}\bigg)\mid_{\alpha}.$ The Poincar{\'e} polynomial of $(\mathbb{I}^{\bullet}_{i})_{\alpha}$ is related to the $\textbf{T}$ character of $F_{i}$ as:
\begin{equation}
\textbf{F}^{\textbf{T}}_{i,\alpha}=\frac{C^{n}_{\alpha}+(\textbf{P}_{i})_{\alpha}}{(1-t_{1})(1-t_{2})(1-t_{3})},
\end{equation}
where the correction term $C^{n}_{\alpha}$ is the $\textbf{T}$-character of $\mathcal{O}_{X}(-n)$ which depends on the choice of the equivariant structure. Now the $\mathcal{T}$-character of $F_{i}$ is given by:
\begin{equation}\label{Fi}
\textbf{F}^{\mathcal{T}}_{i,\alpha}=w_{i}\cdot\textbf{F}^{\textbf{T}}_{i,\alpha}= \frac{C^{n}_{\alpha}\cdot w_{i}+w_{i}\cdot(\textbf{P}_{i})_{\alpha}}{(1-t_{1})(1-t_{2})(1-t_{3})},
\end{equation}
The $\textbf{T}$-character of each $tr_{\chi}((\mathbb{I}^{\bullet}_{i})_{\alpha},(\mathbb{I}^{\bullet}_{i})_{\alpha})$ is given as follows  \cite[Section 4.7]{a15}:
\begin{equation}
tr_{\chi}((\mathbb{I}^{\bullet}_{i})_{\alpha},(\mathbb{I}^{\bullet}_{i})_{\alpha})=\frac{w_{i}\cdot w_{i}^{-1}\cdot(\textbf{P}_{i})_{\alpha}\overline{(\textbf{P}_{i}})_{\alpha}}{(1-t_{1})(1-t_{2})(1-t_{3})}=\frac{(\textbf{P}_{i})_{\alpha}\overline{(\textbf{P}_{i}})_{\alpha}}{(1-t_{1})(1-t_{2})(1-t_{3})}.
\end{equation}
The dual bar operation is negation on $\mathcal{K}(\textbf{Q}\mid_{\mathcal{U}_{\alpha}})$ and $t_{i}\rightarrow \frac{1}{t_{i}}$ on the equivariant variables $t_{i}$. Since $\mathbb{I}^{\bullet,\mathcal{T}}_{\alpha}:=\bigoplus_{i=1}^{r}(\mathbb{I}^{\bullet,\textbf{T}}_{i})_{\alpha}$ the $\mathcal{T}$-character of $\chi(\mathbb{I}^{\bullet,\mathcal{T}}_{\alpha},\mathbb{I}^{\bullet,\mathcal{T}}_{\alpha})$ is obtained as a sum:
\begin{equation} 
tr_{\chi}(\mathbb{I}^{\bullet,\mathcal{T}}_{\alpha},\mathbb{I}^{\bullet,\mathcal{T}}_{\alpha})=\sum_{\begin{subarray}{1} 1\leq i\leq r\\ 1\leq j\leq r\end{subarray}}\frac{w_{i}w_{j}^{-1}\cdot(\textbf{P}_{i})_{\alpha}\overline{(\textbf{P}_{j}})_{\alpha}}{(1-t_{1})(1-t_{2})(1-t_{3})}.
\end{equation}
Moreover the $\mathcal{T}$-character of $F_{\alpha}$ appearing in $\mathbb{I}^{\bullet,\mathcal{T}}_{\alpha}$ is given by :
\begin{equation}
\textbf{F}^{\mathcal{T}}_{\alpha}=\frac{\sum_{i=1}^{r}w_{i}\cdot C^{n}_{\alpha}+\sum_{i=1}^{r} w_{i}\cdot(\textbf{P}_{i})_{\alpha}}{(1-t_{1})(1-t_{2})(1-t_{3})}.
\end{equation}
Now the $\mathcal{T}$-character of the $\alpha$-summand of $\mathcal{T}^{1}_{[I^{\bullet}]}$ in \eqref{chi-I} is given by:
\begin{equation}
\frac{1-\sum_{\begin{subarray}{1} 1\leq i\leq r\\ 1\leq j\leq r\end{subarray}}w_{i}w^{-1}_{j}\cdot(\textbf{P}_{i})_{\alpha}\overline{(\textbf{P}_{j})}_{\alpha}}{(1-t_{1})(1-t_{2})(1-t_{3})},
\end{equation} 
which using \eqref{Fi}, can eventually be written as a function of $\textbf{F}^{\mathcal{T}}_{\alpha}$:
\begin{align}\label{trace1}
tr_{R-_{\chi((\mathbb{I}^{\bullet,\mathcal{T}})_{\alpha},(\mathbb{I}^{\bullet,\mathcal{T}})_{\alpha})}}=&\textbf{F}^{\mathcal{T}}_{\alpha}\cdot (\sum_{j=1}^{r}w^{-1}_{j})\cdot\overline{C^{n}_{\alpha}}-\frac{\overline{\textbf{F}^{\mathcal{T}}_{\alpha}}\cdot(\sum_{i=1}^{r}w_{i})\cdot  C^{n}_{\alpha}}{t_{1}t_{2}t_{3}}\notag\\
&
+\textbf{F}^{\mathcal{T}}_{\alpha}\overline{\textbf{F}^{\mathcal{T}}_{\alpha}}\frac{(1-t_{1})(1-t_{2})(1-t_{3})}{t_{1}t_{2}t_{3}}+\frac{1-(\sum_{i,j=1}^{r}w_{i}w^{-1}_{j})\cdot C^{n}_{\alpha}\overline{C^{n}_{\alpha}}}{(1-t_{1})(1-t_{2})(1-t_{3})}\notag\\
\end{align}
Now we compute the $\mathcal{T}$-character of $\mathcal{T}^{2}_{[I^{\bullet}]}$, $\mathcal{T}^{3}_{[I^{\bullet}]}$ and $\mathcal{T}^{4}_{[I^{\bullet}]}$. Assume that $\mathcal{U}_{\alpha\beta}$ is the affine patch, over which the equivariant parameter $t_{1}$ is invertible. Given $F^{\mathcal{T}}=\bigoplus_{i=1}^{r}F^{\textbf{T}}_{i}$, let $(F^{\textbf{T}}_{i})_{\alpha\beta}$ denote the restriction of $F^{\textbf{T}}_{i}$ to $\mathcal{U}_{\alpha\beta}$. Let $\textbf{F}^{\textbf{T}}_{i,\alpha\beta}=\sum_{k_{2},k_{3}\in \mu_{\alpha\beta}}t_{2}^{k_{2}}t_{3}^{k_{3}}$ denote the $\textbf{T}$-character associated to this restriction (c.f. \cite[ Equation 4.10]{a15}). We have that $\textbf{F}^{\mathcal{T}}_{\alpha\beta}=\sum_{i=1}^{r}\textbf{F}^{\textbf{T}}_{i,\alpha\beta}\cdot w_{i}.$ 

By the same argument as above, one relates the $\mathcal{T}$-character of $\alpha\beta$'th summand of the virtual tangent space $\mathcal{T}^{2}_{[I^{\bullet}]}$ in \eqref{chi-I} to $\textbf{F}^{\mathcal{T}}_{\alpha\beta}$:
\begin{align}\label{trace2}
tr_{R-_{\chi((\mathbb{I}^{\bullet})_{\alpha\beta},(\mathbb{I}^{\bullet})_{\alpha\beta})}}=&\bigg[\textbf{F}^{\mathcal{T}}_{\alpha\beta}(\sum_{j=1}^{r}w^{-1}_{j})\cdot\overline{C^{n}_{\alpha\beta}}-\frac{\overline{\textbf{F}^{\mathcal{T}}_{\alpha\beta}}\cdot (\sum_{i=1}^{r}w_{i})\cdot  C^{n}_{\alpha\beta}}{t_{2}t_{3}}\notag\\
&
+\textbf{F}^{\mathcal{T}}_{\alpha\beta}\overline{\textbf{F}^{\mathcal{T}}_{\alpha\beta}}\frac{(1-t_{2})(1-t_{3})}{t_{2}t_{3}}+\frac{1-(\sum_{i,j=1}^{r}w_{i}w^{-1}_{j})\cdot C^{n}_{\alpha\beta}\overline{C^{n}_{\alpha\beta}}}{(1-t_{2})(1-t_{3})}\bigg]\cdot \delta(t_{1}),
\end{align}
here $C^{n}_{\alpha\beta}$ is a function of $n$ and the correction term that needs to be inserted into description of the Poincar{\'e} polynomial of $\mathcal{O}_{X}\mid_{\mathcal{U}_{\alpha\beta}}$ in order to obtain the Poincar{\'e} polynomial of $\mathcal{O}_{X}(-n)\mid_{\mathcal{U}_{\alpha\beta}}$. Here, we have used the notation $\delta(t_{1})=\sum_{k\in \mathbb{Z}}t^{k}_{i}.$ Now assume $\mathcal{U}_{\alpha\beta\gamma}$ is the affine patch, over which the equivariant parameters $t_{1}$ and $t_{2}$ are invertible. The $\alpha,\beta,\gamma$'th summand of $\mathcal{T}^{3}_{[I^{\bullet}]}$ in \eqref{chi-I} is obtained as follows:
\begin{align}\label{edge3}
&
tr_{R-_{\chi((\mathbb{I}^{\bullet})_{\alpha\beta\gamma},(\mathbb{I}^{\bullet})_{\alpha\beta\gamma})}}=
\frac{(1- \sum_{i,j=1}^{r}w_{i}w^{-1}_{j})}{(1-t_{3})}\delta(t_{1})\delta(t_{2}).
\end{align}
and the $\textbf{T}$-character of $\mathcal{T}^{4}_{[I^{\bullet}]}$ in \eqref{chi-I} is obtained as:
\begin{equation}\label{edge4}
tr_{R-_{\chi((\mathbb{I}^{\bullet})_{\alpha\beta\gamma\delta},(\mathbb{I}^{\bullet})_{\alpha\beta\gamma\delta})}}=
(1- \sum_{i,j=1}^{r}w_{i}w^{-1}_{j})\delta(t_{1})\delta(t_{2})\delta(t_{3}).
\end{equation}
Based on above discussion, the $\mathcal{T}$-character of the virtual tangent space over a point is obtained as follows:
\begin{align}\label{virtan}
tr_{R-\chi(,\mathbb{I}^{\bullet})}=\sum_{\alpha}&tr_{R-_{\chi((\mathbb{I}^{\bullet})_{\alpha},(\mathbb{I}^{\bullet})_{\alpha})}}-\sum_{\alpha,\beta}tr_{R-_{\chi((\mathbb{I}^{\bullet})_{\alpha\beta},(\mathbb{I}^{\bullet})_{\alpha\beta})}}\notag\\
&+\sum_{\alpha,\beta,\gamma}tr_{R-_{\chi((\mathbb{I}^{\bullet})_{\alpha\beta\gamma},(\mathbb{I}^{\bullet})_{\alpha\beta\gamma})}}-\sum_{\alpha,\beta,\gamma,\delta}tr_{R-_{\chi((\mathbb{I}^{\bullet})_{\alpha\beta\gamma\delta},(\mathbb{I}^{\bullet})_{\alpha\beta\gamma\delta})}}.\notag\\
\end{align}
\subsubsection{Description of equivariant vertex}
The $\mathcal{T}$-character of the virtual tangent space in \eqref{virtan} is equal to the addition of vertex contributions (the first summand on right hand side of \eqref{virtan}) and the remaining edge contributions. Similar to discussions in \cite[Section 4.6]{a18}, we need to redistribute the terms in \eqref{trace1}, \eqref{trace2}, \eqref{edge3} and \eqref{edge4} so that they become Laurent polynomials in the variables $t_{i}$.
Let us define  
\begin{align}\label{G_alph}
\text{G}_{\alpha\beta}=\textbf{F}^{\mathcal{T}}_{\alpha\beta}(\sum_{j=1}^{r}w^{-1}_{j})&\cdot\overline{C^{n}_{\alpha\beta}}-\frac{\overline{\textbf{F}^{\mathcal{T}}_{\alpha\beta}}\cdot (\sum_{i=1}^{r}w_{i})\cdot  C^{n}_{\alpha\beta}}{t_{2}t_{3}}\notag\\
&
+\textbf{F}^{\mathcal{T}}_{\alpha\beta}\overline{\textbf{F}^{\mathcal{T}}_{\alpha\beta}}\frac{(1-t_{2})(1-t_{3})}{t_{2}t_{3}}+\frac{1-(\sum_{i,j=1}^{r}w_{i}w^{-1}_{j})\cdot C^{n}_{\alpha\beta}\overline{C^{n}_{\alpha\beta}}}{(1-t_{2})(1-t_{3})}.
\end{align}
then the edge character \eqref{trace2} is written with respect to $\text{G}_{\alpha\beta}$ as 

\begin{equation}\label{ascend-descend1}
\frac{\text{G}_{\alpha\beta}(t_{2},t_{3})}{1-t_{1}}+t^{-1}_{1}\frac{\text{G}_{\alpha\beta}(t_{2},t_{3})}{1-t_{1}}.
\end{equation}
Now we expand $\left(\frac{\text{G}_{\alpha\beta}(t_{2},t_{3})}{1-t_{1}}\right)$ in ascending powers of $t_{1}$ and expand $\left(t_{1}^{-1}\frac{\text{G}_{\alpha\beta}(t_{2},t_{3})}{1-t_{1}^{-1}}\right)$ in descending powers of $t_{1}$. Equation \eqref{ascend-descend1} is exactly the same as \cite[Equation (4.11)]{a18}.  Similarly define $$\text{G}_{\alpha\beta\gamma}=\frac{(1- \sum_{i,j=1}^{r}w_{i}w^{-1}_{j})}{(1-t_{3})}.$$Hence, \eqref{edge3} is rewritten as
\begin{equation}\label{ascend-descend}
\left(\frac{\text{G}_{\alpha\beta\gamma}(t_{3})}{1-t_{1}}+t_{1}^{-1}\frac{\text{G}_{\alpha\beta\gamma}(t_{3})}{1-t_{1}^{-1}}\right)\frac{1}{1-t_{2}}+t_{2}^{-1}\left(\frac{\text{G}_{\alpha\beta\gamma}(t_{3})}{1-t_{1}}+t_{1}^{-1}\frac{\text{G}_{\alpha\beta\gamma}(t_{3})}{1-t_{1}^{-1}}\right)\frac{1}{1-t^{-1}_{2}},
\end{equation} 
where we expand $\left(\frac{\text{G}_{\alpha\beta\gamma}(t_{3})}{1-t_{1}}\right)$ in ascending powers of $t_{1}$ and expand $\left(t_{1}^{-1}\frac{\text{G}_{\alpha\beta\gamma}(t_{3})}{1-t_{1}^{-1}}\right)$ in descending powers of $t_{1}$. We follow the same strategy and expand the first term in \eqref{ascend-descend} in ascending powers of $t_{2}$ and the second term in descending powers of $t_{2}$. Finally let$$\text{G}_{\alpha\beta\gamma\delta}=(1- \sum_{i,j=1}^{r}w_{i}w^{-1}_{j}),$$ and obtain a similar redistribution with respect to the ascending and descending powers of $t_{3}$. Now for each $\mathcal{U}_{\alpha}$ define a new vertex character (compare with \cite[Equation 4.12]{a18}):
\begin{equation}\label{local-vertex}
V_{\alpha}=tr_{R-_{\chi((\mathbb{I}^{\bullet,\mathcal{T}})_{\alpha},(\mathbb{I}^{\bullet,\mathcal{T}})_{\alpha})}}+\sum_{i=1}^{3}\frac{\text{G}_{\alpha\beta_{i}}(t_{i'},t_{i^{\prime\prime}})}{1-t_{i}}
\end{equation}
where $\beta_{1},\beta_{2},\beta_{3}$ are the three neighboring vertices and $(t_{i},t_{i'},t_{i^{\prime\prime}})=(t_{1},t_{2},t_{3}).$ Moreover redefine the edge character $\text{E}_{\alpha\beta}$ (compare with \cite[Section 4.6]{a18}):
\begin{equation}
\text{E}_{\alpha\beta}=t_{1}^{-1}\frac{\text{G}_{\alpha\beta}(t_{2},t_{3})}{1-t^{-1}_{1}}-\frac{\text{G}_{\alpha\beta}(t_{2}t_{1}^{-m_{\alpha\beta}},t_{3}t_{1}^{-m'_{\alpha\beta}})}{1-t_{1}^{-1}}
\end{equation} 
Here the integers $m_{\alpha\beta}$ and $m'_{\alpha\beta}$ are determined by the normal bundle $\mathcal{N}_{\mathcal{C}_{\alpha\beta}\slash X}$ to the supporting curve $\mathcal{C}_{\alpha\beta}:=\text{Supp}(F_{\alpha\beta})$ given by: $\mathcal{N}_{\mathcal{C}_{\alpha\beta}\slash X}=\mathcal{O}(m_{\alpha\beta})\oplus \mathcal{O}(m'_{\alpha\beta}).$ 
The edge contributions $\text{E}_{\alpha\beta\gamma}$ and $\text{E}_{\alpha\beta\gamma\delta}$ would have a similar description as above by replacing $\text{G}_{\alpha\beta}$ with $\text{G}_{\alpha\beta\gamma}$ and $\text{G}_{\alpha\beta\gamma\delta}$ and redistributing in ascending and descending powers of $t_{2}$ and $t_{3}$ respectively. According to the above redistributions, the $\mathcal{T}$-character of the virtual tangent space in \eqref{virtan} can be rewritten as:
\begin{align}\label{character-total} 
tr_{R-\chi(\mathbb{I}^{\bullet},\mathbb{I}^{\bullet})}=\sum_{\alpha}V_{\alpha}+\sum_{\alpha\beta}\text{E}_{\alpha\beta}
+\sum_{\alpha\beta\gamma}\text{E}_{\alpha\beta\gamma\delta}+\sum_{\alpha\beta\gamma\delta}\text{E}_{\alpha\beta\gamma\delta}
\end{align}
\begin{defn}\label{HFT-vertex}
Given a torus fixed component $\textbf{Q}^{k}$ of the moduli stack of highly frozen triples (here $k$ denotes the length of the zero dimensional cokernel sheaf associated to the highly frozen triples) denote $V_{\textbf{Q}^{k}}=\sum_{\alpha}V_{\alpha}$ where $V_{\alpha}$ are defined as in \eqref{local-vertex}. By discussions in \cite[Section 4.7]{a18} one defines the contribution (to the equivariant 3-fold vertex) of the locus $\textbf{Q}^{k}$ as the evaluation of $V_{\textbf{Q}^{k}}$ over $\textbf{Q}^{k}$, i.e:
\begin{equation}\label{eq-vertex}
w(\textbf{Q}^{k})={\displaystyle \int_{\textbf{Q}^{k}}e(\text{T}_{\textbf{Q}^{k}})e(-V_{\textbf{Q}^{k}})}.
\end{equation}
Hence, the equivariant Calabi-Yau vertex associated to the moduli scheme of highly frozen triples is defined as:
\begin{equation}\label{CY3vertex}
W^{\text{HFT}}_{\textbf{Q}}=\sum_{k}w(\textbf{Q}^{k}).q^{k}
\end{equation} 
\end{defn}
\subsubsection{Application to local $\mathbb{P}^{1}$ and more computations}\label{sec17}
First we identify the associated equivariant data of highly frozen triples using a theorem of Pandharipande-Thomas \cite[Proposition 1.8]{a18}.
\begin{prop}\label{identyf}
Consider a stable $\mathcal{T}$-equivariant highly frozen triple $\mathcal{O}^{\oplus r}_{X}(-n)^{\mathcal{T}}\xrightarrow{\phi^{\mathcal{T}}} F^{\mathcal{T}}$ of type $(r, P_{F})$. Let $\mathcal{C}$ be the one dimensional support of $F$. Now consider the finite length $\mathcal{T}$-equivariant cokernel $\textbf{Q}$ given by $\text{Coker}(\phi)^{\mathcal{T}}$. Then $\textbf{Q}\cong Q^{\textbf{T}}_{1}\oplus \cdots \oplus Q^{\textbf{T}}_{r}$ such that each $Q^{\textbf{T}}_{i}$ for $i=1,\cdots,r$ is given as a subsheaf of
\begin{equation}\label{quasi-coh}
\mathcal{H}=\varinjlim_{l} \left(\mathscr{H}om(\mathfrak{m}^{l},\mathcal{O}_{\mathcal{C}})/\mathcal{O}_{\mathcal{C}}\right). 
\end{equation}
In other words, the equivariant data of a stable $\mathcal{T}$-equivariant highly frozen triple with support curve $\mathcal{C}$ is equivalent to the data of a subsheaf of ``$r$" copies of $\mathcal{H}$ in \eqref{quasi-coh}.
\end{prop}
\begin{proof} Since $\mathcal{O}^{\oplus r}_{X}(-n)^{\mathcal{T}}\rightarrow F^{\mathcal{T}}:=\bigoplus_{i=1}^{r}(\mathcal{O}^{\textbf{T}}_{X}(-n)\rightarrow F^{\textbf{T}}_{i}),$ each $\mathcal{O}^{\textbf{T}}_{X}(-n)\rightarrow F^{\textbf{T}}_{i}$ restricted to the supporting curve of $F_{i}$, is identified with $Q^{\textbf{T}}_{i}$ appearing in $0\rightarrow \mathcal{O}^{\textbf{T}}_{C}(-n)\rightarrow F^{\textbf{T}}_{i}\rightarrow Q^{\textbf{T}}_{i}\rightarrow 0$ and by \cite[Proposition 1.8]{a17} $Q_{i}$ is identified with a subsheaf of the quasi-coherent sheaf $\varinjlim_{l} \mathscr{H}om(\mathfrak{m}^{l},\mathcal{O}_{\mathcal{C}})/\mathcal{O}_{\mathcal{C}}.$ It is then seen that the cokernel of the original $\mathcal{T}$-equivariant highly frozen triple restricted to $\mathcal{C}$ and identified with $\bigoplus_{i=1}^{r}Q_{i}^{\textbf{T}}$, is  a subsheaf of the direct sum of $r$ copies of the quasi-coherent sheaf $\mathcal{H}$.
\end{proof}
Now  assume that $X$ is given as local $\mathbb{P}^{1}$. We use the combinatorial description of $X$, using the 3 dimensional Young tableaux diagrams as in \cite[Example 4.9]{a18}. There exists two affine patches covering $X$. The partitions associated to the Newton polyhedron of $X$ on each patch are given as three dimensional partitions with $\mu_{1}=(1),\mu_{2}=(0),\mu_{3}=(0)$ (compare with \cite[Example 4.9]{a18}). 

We compute the 1-legged equivariant vertex $W^{\text{HFT}}_{1,\emptyset,\emptyset}$ associated to the moduli scheme of highly frozen triples of type $(r,P_{F})$. This computation is the higher rank analog of the computation in \cite[Lemma 5]{a18}.

Let $\mathcal{U}_{\alpha},\mathcal{U}_{\beta}$ denote affine open charts over the divisors $0,\infty$ on the base $\mathbb{P}^{1}$ respectively. Let $\mathbb{C}^{*}$ act on $\mathbb{C}^{4}$ by $$t(x_{0},x_{1},x_{2},x_{3})=(tx_{0},tx_{1},t^{-1}x_{2},t^{-1}x_{3}).$$ We identify $X$ as a quotient $X\cong (\mathbb{C}^{4}\backslash Z)\slash \mathbb{C}^{*}$ where $Z\subset \mathbb{C}^{4}$ is obtained by setting $x_{0}=x_{1}=0$. Let $([x_{0}:x_{1}],x_{2},x_{3})$ denote the coordinates in $X$ where $[x_{0}:x_{1}]$ denote the homogeneous coordinates along the base $\mathbb{P}^{1}$ and $x_{2},x_{3}$ denote the fiber coordinates. Locally in the $\mathcal{U}_{\alpha}$ and $\mathcal{U}_{\beta}$ the defining coordinates are given as $(\frac{x_{1}}{x_{0}},x_{2}x_{0},x_{3}x_{0})$ and $(\frac{x_{0}}{x_{1}},x_{2}x_{1},x_{3}x_{1})$ respectively. 

Now denote the local coordinates over $\mathcal{U}_{\alpha}$ by $(\tilde{x}_{1},\tilde{x}_{2},\tilde{x}_{3})$ where $\tilde{x}_{1}=\frac{x_{1}}{x_{0}},\tilde{x}_{2}=x_{2}x_{0},\tilde{x}_{3}=x_{3}x_{0}$. Let $H\subset X$ denote the hyperplane obtained as the fiber of $X$ over $0\in \mathbb{P}^{1}$, i.e locally in $\mathcal{U}_{\alpha}$ by setting $\tilde{x}_{1}=0$. Throughout this calculation we fix the hyperplane $H$ as a choice of equivariant structure on $\mathcal{O}_{X}(1)$. Now consider the action of $\textbf{T}=\mathbb{C}^{3}$ on $X$ where locally over $\mathcal{U}_{\alpha}$ is given by $(\lambda_{1},\lambda_{2},\lambda_{3})\cdot \tilde{x}_{i}=\lambda_{i}\cdot \tilde{x}_{i}.$ 

We identify an action of $(\mathbb{C}^{*})^{2}$ on $X$ which preserves the Calabi-Yau form by considering a subtorus $\text{T}'\subset\textbf{T}$ such that $\text{T}'=\{(\lambda_{1},\lambda_{2},\lambda_{3})\in \textbf{T}\mid \lambda_{1}\lambda_{2}\lambda_{3}=1\}.$ Let $\tilde{t}_{1},\cdots,\tilde{t}_{3}$ denote the characters corresponding to the action of $\lambda_{i}$. Identify $\mathcal{O}_{X}(-1)\cong \mathcal{O}_{X}(-H)$. Then, locally over $\mathcal{U}_{\alpha}$, the Poincar{\'e} polynomial of $\mathcal{O}_{X}(-n)\mid_{\mathcal{U}_{\alpha}}$ is obtained as $$\frac{\tilde{t}_{1}^{n}}{(1-\tilde{t}_{1})(1-\tilde{t}_{2})(1-\tilde{t}_{3})},$$ and moreover, the Poincar{\'e} polynomial of $\mathcal{O}_{X}(-n)\mid_{\mathcal{U}_{\beta}}$ is obtained by:$$\frac{1}{(1-\tilde{t}_{1}^{-1})(1-(\tilde{t}_{2}\tilde{t}_{1}))(1-(\tilde{t}_{3}\tilde{t}_{1}))}.$$Note that in this case the correction terms $C^{n}_{\alpha}$ and $C^{n}_{\beta}$ in \eqref{trace1} (which depend on the choice of equivariant structure) are $\tilde{t}^{n}_{1}$ and $1$ respectively. Similarly, the $\textbf{T}$-character of the Poincar{\'e} polynomial of $\mathcal{O}_{X}(-n)\mid_{\mathcal{U}_{\alpha\beta}}$ is obtained as $$\left(\frac{1}{(1-\tilde{t}_{2})(1-\tilde{t}_{3})}\right)\delta(\tilde{t}_{1}).$$ Here the correction term $C^{n}_{\alpha\beta}$ in \eqref{trace2} is equal to 1. To compute the contributions in \eqref{local-vertex} we need to compute the trace characters in \eqref{trace1} over the two patches $\alpha$ and $\beta$ and the edge redistribution in \eqref{G_alph}. 
Let $\textbf{Q}^{k}$ denote the $\mathcal{T}$-fixed component of the moduli scheme of highly frozen triples over which the highly frozen triples $\mathcal{O}^{\oplus r}_{X}(-n)^{\mathcal{T}}\xrightarrow{\phi} F^{\mathcal{T}}$ satisfy the condition that $l(\text{Coker}(\phi)^{\mathcal{T}})=k$. By \eqref{character-total} the $\mathcal{T}$-equivariant vertex $V_{\textbf{Q}^{k}}$ is given by
\begin{align}\label{rank2vertex}
V_{\textbf{Q}^{k}}=\sum_{d_{1}+d_{2}=k}\Bigg((w^{-1}_{1}+w^{-1}_{2})\cdot\left(w_{1}\cdot \sum_{i=1}^{d_{1}}\tilde{t}^{-i-n}_{1}+w_{2}\cdot \sum_{i=1}^{d_{2}}\tilde{t}^{-i-n}_{1}\right)\notag\\
-(w_{1}+w_{2})\cdot\bigg(w^{-1}_{1}\cdot \sum_{i=0}^{d_{1}-1}\frac{\tilde{t}^{i+n}_{1}}{\tilde{t}_{2}\tilde{t}_{3}}
+w^{-1}_{2}\cdot \sum_{i=0}^{d_{2}-1}&\frac{\tilde{t}^{i+n}_{1}}{\tilde{t}_{2}\tilde{t}_{3}}\bigg)\Bigg)\notag\\
\end{align}
(compare with \cite[Lemma 5]{a18}). Now let $\text{s}_{i}$ for $i=1,2,3$ and $v_{j}$ for $j=1,\cdots, r$ denote the equivariant parameters corresponding to characters $\tilde{t}_{i}$ and $w_{j}$ respectively. By the definition of the equivariant vertex in \eqref{HFT-vertex} the coefficient of the degree $k$ term in the equivariant vertex in \eqref{CY3vertex} is obtained by the evaluation of the contribution of $V_{\textbf{Q}^{k}}$ on $\textbf{Q}^{k}$, i.e:
\begin{align}\label{lastnumber}
w(\textbf{Q}^{k})=&{\displaystyle{\int_{\textbf{Q}^{k}}}e(\text{T}_{\textbf{Q}^{k}})e(-V_{\textbf{Q}^{k}})}=
\displaystyle{\prod_{d_{1}+d_{2}=k}}\Bigg[\frac{\displaystyle{\prod_{j=1}^{r}}\bigg((-v_{j}+\prod_{l=1}^{r}v_{l})+\displaystyle{\prod_{i=1}^{d_{j}-1}}((i+n)s_{1})-(s_{2}+s_{3})\bigg)}{\displaystyle{\prod_{j=1}^{r}}\bigg((v_{j}+(-1)^{r-1}\prod_{l=1}^{r}v_{l})+\displaystyle{\prod_{i=1}^{d_{j}}}(-1)^{i}\cdot(i+n)s_{1}\bigg)}\Bigg].
\end{align}
\begin{remark}
Setting $v_{1}=v_{2}=1$ in \eqref{lastnumber} would result in the following equation:
\begin{align}\label{lastnumber3}
w(\textbf{Q}^{k})=
\displaystyle{\prod_{d_{1}+d_{2}=k}}\Bigg[\frac{\bigg(\displaystyle{\prod_{i=0}^{d_{1}-1}}((i+n)s_{1})-(s_{2}+s_{3})\bigg)\cdot\bigg(\displaystyle{\prod_{i=0}^{d_{2}-1}}((i+n)s_{1})-(s_{2}+s_{3})\bigg)}{\bigg(\displaystyle{\prod_{i=1}^{d_{1}}}(-1)^{i}\cdot(i+n)s_{1}\bigg)\cdot\bigg(\displaystyle{\prod_{i=1}^{d_{2}}}(-1)^{i}(i+n)s_{1}\bigg)}\Bigg].\notag\\
\end{align}
Now use the condition on Calabi-Yau torus and set $s_{1}+s_{2}+s_{3}=0$. This is equivalent to $ns_{1}-(s_{2}+s_{3})=-(n+1)(s_{2}+s_{3})$. Now use this fact to simplify \eqref{lastnumber3} and obtain the 1-legged equivariant vertex over local $\mathbb{P}^{1}$ associated to highly frozen triples:
 \begin{equation}\label{one-leg}
 W^{\text{HFT}}_{1,\emptyset,\emptyset}=\left((1+q)^{\frac{(n+1)(s_{2}+s_{3})}{s_{1}}}\right)^{r}
 \end{equation}
\end{remark}
\begin{remark}
\emph{The computation of equivariant vertex for more general local toric Calabi-Yau threefolds with outgoing partitions $<\mu_{1},\mu_{2},\mu_{3}>$ requires more detailed combinatorial calculations. However, one can compute the associated partition functions in those general cases if one fully understands the equivariant PT vertex in rank 1. In other words it is seen from our calculations that if the PT vertex with respect to the variable $q$ is given by $W^{P}_{<\mu_{1},\mu_{2},\mu_{3}>}=G(q)$ then the HFT equivariant vertex is obtained by $$W^{\text{HFT}}_{<\mu_{1},\mu_{2},\mu_{3}>}=(G(q))^{(n+1)(r)}.$$}    
\end{remark}

\bibliographystyle{amsalpha}
\bibliography{ref1}

\providecommand{\bysame}{\leavevmode\hbox to3em{\hrulefill}\thinspace}
\providecommand{\MR}{\relax\ifhmode\unskip\space\fi MR }
\providecommand{\MRhref}[2]{%
  \href{http://www.ams.org/mathscinet-getitem?mr=#1}{#2}
}
\providecommand{\href}[2]{#2}
\begin{thebibliography}{MNOP06}

\bibitem[Beh09]{a1}
Kai Behrend, \emph{Donaldson-{T}homas type invariants via microlocal geometry},
  Ann. of Math. (2) \textbf{170} (2009), no.~3, 1307--1338. \MR{2600874
  (2011d:14098)}

\bibitem[BF97]{a2}
K.~Behrend and B.~Fantechi, \emph{The intrinsic normal cone}, Invent. Math.
  \textbf{128} (1997), 45--88.

\bibitem[CL11]{a70}
Huai-liang Chang and Jun Li, \emph{Semi-perfect obstruction theory and
  {D}onaldson-{T}homas invariants of derived objects}, Comm. Anal. Geom.
  \textbf{19} (2011), no.~4, 807--830. \MR{2880216}

\bibitem[CRdB10]{a68}
Damien Calaque, Carlo~A. Rossi, and Michel~Van den Bergh, \emph{{Hochschild}
  cohomology for {Lie} algebroids}, Int Math. Res. Not. \textbf{2010, 21}
  (2010), 4098--4136.

\bibitem[ea87]{a71}
A~Borel et~al, \emph{Algebraic {D}-modules}, Perspectives in Mathematics
  (1987).

\bibitem[GP99]{a25}
T.~Graber and R.~Pandharipande, \emph{Localization of virtual classes}, Invent.
  Math. \textbf{135} (1999), no.~2, 487--518. \MR{1666787 (2000h:14005)}

\bibitem[GP11]{DeMazur}
Philippe Gille and Patrick Polo (eds.), \emph{Sch\'emas en groupes ({SGA} 3).
  {T}ome {I}. {P}ropri\'et\'es g\'en\'erales des sch\'emas en groupes},
  Documents Math\'ematiques (Paris) [Mathematical Documents (Paris)], 7,
  Soci\'et\'e Math\'ematique de France, Paris, 2011, S{\'e}minaire de
  G{\'e}om{\'e}trie Alg{\'e}brique du Bois Marie 1962--64. [Algebraic Geometry
  Seminar of Bois Marie 1962--64], A seminar directed by M. Demazure and A.
  Grothendieck with the collaboration of M. Artin, J.-E. Bertin, P. Gabriel, M.
  Raynaud and J-P. Serre, Revised and annotated edition of the 1970 French
  original. \MR{2867621}

\bibitem[HL10]{a9}
Daniel Huybrechts and Manfred Lehn, \emph{The geometry of moduli spaces of
  sheaves}, second ed., Cambridge Mathematical Library, Cambridge University
  Press, Cambridge, 2010. \MR{2665168 (2011e:14017)}

\bibitem[HT10]{a10}
Daniel Huybrechts and Richard~P. Thomas, \emph{Deformation-obstruction theory
  for complexes via {A}tiyah and {K}odaira-{S}pencer classes}, Math. Ann.
  \textbf{346} (2010), no.~3, 545--569. \MR{2578562 (2011b:14030)}

\bibitem[Ill71]{a29}
L.~Illusie, \emph{Complexe cotangent et deformations. {I}}, Springer Lecture
  Notes in Math \textbf{239} (1971).

\bibitem[Koo11]{a34}
Martijn Kool, \emph{Fixed point loci of moduli spaces of sheaves on toric
  varieties}, Adv. Math. \textbf{227} (2011), no.~4, 1700--1755. \MR{2799810
  (2012k:14021)}

\bibitem[LMB00]{a61}
G{\'e}rard Laumon and Laurent Moret-Bailly, \emph{Champs alg\'ebriques},
  Ergebnisse der Mathematik und ihrer Grenzgebiete. 3. Folge. A Series of
  Modern Surveys in Mathematics [Results in Mathematics and Related Areas. 3rd
  Series. A Series of Modern Surveys in Mathematics], vol.~39, Springer-Verlag,
  Berlin, 2000. \MR{1771927 (2001f:14006)}

\bibitem[LP95]{a13}
J.~Le~Potier, \emph{Faisceaux semi-stables et syst\`emes coh\'erents}, Vector
  bundles in algebraic geometry ({D}urham, 1993), London Math. Soc. Lecture
  Note Ser., vol. 208, Cambridge Univ. Press, Cambridge, 1995, pp.~179--239.
  \MR{1338417 (96h:14010)}

\bibitem[MNOP06]{a15}
D.~Maulik, N.~Nekrasov, A.~Okounkov, and R.~Pandharipande,
  \emph{Gromov-{W}itten theory and {D}onaldson-{T}homas theory. {I}}, Compos.
  Math. \textbf{142} (2006), no.~5, 1263--1285. \MR{2264664 (2007i:14061)}

\bibitem[Nev02]{a74}
Tom Nevins, \emph{Moduli spaces of framed sheaves on certain ruled surfaces
  over elliptic curves}, Int. J. Math. \textbf{13, 10} (2002), 1117--1151.

\bibitem[Nos07]{a43}
Francesco Noseda, \emph{A proposal for virtual fundamental class for {A}rtin
  stacks}, PhD Thesis (2007).

\bibitem[Ols07]{a66}
Martin Olsson, \emph{Sheaves on {A}rtin stacks}, J. Reine Angew. Math.
  \textbf{603} (2007), 55--112. \MR{2312554 (2008b:14002)}

\bibitem[PP12]{a109}
R.~Pandharipande and A.~Pixton, \emph{{Gromov-Witten/Pairs correspondence for
  the quintic 3-fold}}, arXiv:1206.5490 (2012).

\bibitem[PT09a]{a17}
R.~Pandharipande and R.~P. Thomas, \emph{Curve counting via stable pairs in the
  derived category}, Invent. Math. \textbf{178} (2009), no.~2, 407--447.
  \MR{2545686 (2010h:14089)}

\bibitem[PT09b]{a18}
Rahul Pandharipande and Richard~P. Thomas, \emph{The 3-fold vertex via stable
  pairs}, Geom. Topol. \textbf{13} (2009), no.~4, 1835--1876. \MR{2497313
  (2010a:14091)}

\bibitem[PT14]{KKVPT}
R.~Pandharipande and R.~P. Thomas, \emph{{The Katz-Klemm-Vafa conjecture for K3
  surfaces}}, arXiv:1404.6698 (2014).

\bibitem[STV11]{a94}
Timo Schurg, Bertrand Toen, and Gabriele Vezzosi, \emph{{Derived algebraic
  geometry, determinants of perfect complexes, and applications to obstruction
  theories for maps and complexes}}, {arXiv:1102.1150v4} (2011).

\bibitem[vis04]{a67}
Angelo vistoli, \emph{Notes on {Grothendieck} topologies, fibered categories
  and descent theory}, arXiv:0412512v4 (2004).

\bibitem[Wan10]{a62}
Malte Wandel, \emph{Moduli spaces of stable pairs in {Donaldson}-{Thomas}
  theory}, arXiv:1011.3328v1 (2010).

\bibitem[Wis11]{a93}
Jonathan Wise, \emph{{The deformation theory of sheaves commutative rings II}},
  {arXiv:1102.2924v1} (2011).

\bibitem[Wis12]{a92}
Jonathan Wise, \emph{The deformation theory of sheaves of commutative rings},
  J. Algebra \textbf{352} (2012), 180--191. \MR{2862181 (2012k:13036)}

\end{thebibliography}

\noindent {\tt{sheshmani.1@math.osu.edu}} \\
\noindent{\small Ohio State University, 600 Math Tower, 231 West 18th Avenue, Columbus, OH 43210-1174} \\

\end{document}